\documentclass[runningheads]{llncs}
\usepackage{graphicx}
\usepackage{amsmath}
\usepackage{amssymb}
\usepackage{amsfonts}
\usepackage{bbm}
\usepackage{xspace}
\usepackage{subfigure}
\usepackage{mathtools}
\usepackage{bm}
\usepackage{stmaryrd}
\usepackage[colorlinks=true,
            linkcolor=blue,
            urlcolor=blue,
            citecolor=blue]{hyperref}
\usepackage{braket}
\usepackage{cleveref}
\usepackage{multirow}
\usepackage{xparse}
\usepackage{xcolor}
\usepackage{comment}
\usepackage{stackrel}
\usepackage{enumitem}

\usepackage[utf8]{inputenc}
\usepackage{pgfplots}
\pgfplotsset{compat=newest}
\usepgfplotslibrary{groupplots}
\usepgfplotslibrary{dateplot}
\usepackage{nicefrac}
\usepackage{booktabs}

\usepackage{tcolorbox}
\usepackage{mathabx}

\crefname{equation}{}{}
\crefname{table}{Table}{Tables}
\crefname{figure}{Figure}{Figures}
\crefname{section}{Section}{Sections}

\crefname{theorem}{Theorem}{Theorems}
\crefname{remark}{Remark}{Remarks}
\crefname{lemma}{Lemma}{Lemmas}
\crefname{proposition}{Proposition}{Propositions}
\crefname{definition}{Definition}{Definitions}
\crefname{observation}{Observation}{Observations}

\newcommand{\lrp}[1]{\left(#1\right)}
\newcommand{\ds}{\displaystyle}
\newcommand{\dsum}{\ds \sum}

\newcommand{\lrbr}[1]{\llbracket #1 \rrbracket}

\newcommand{\y}{z}
\newcommand{\z}{z}

\newcommand{\xmin}{\ubar{x}}
\newcommand{\ymin}{\ubar{y}}

\newcommand{\xmax}{\bar{x}}
\newcommand{\ymax}{\bar{y}}

\newcommand{\xhat}{\hat{x}}
\newcommand{\yhat}{\hat{\y}}
\newcommand{\zhat}{\hat{\z}}

\newcommand{\lx}{l_x}
\newcommand{\lxi}{l_{x_i}}
\newcommand{\ly}{l_y}
\newcommand{\lyj}{l_{y_j}}

\newcommand{\Del}[2][L]{{\Delta_{#2}^#1}}

\usepackage{todonotes}

\DeclareMathOperator{\proj}{proj}
\DeclareMathOperator{\conv}{conv}

\newcommand{\F}{f}
\newcommand{\capF}{F}
\newcommand{\funcf}{F}

\newcommand{\IPtiny}{{\text{\normalfont \fontsize{6}{6}\selectfont IP}}}
\newcommand{\LPtiny}{{\text{\normalfont \fontsize{6}{6}\selectfont LP}}}
\newcommand{\PLP}{P^{\LPtiny}} 
\newcommand{\PIP}{P^{\IPtiny}}

\newcommand{\lxk}{l_{x_k}} 
\newcommand{\lyl}{l_{x_l}}

\newcommand{\morsireform}{\textnormal{Bin2}\xspace}
\newcommand{\zellmerreform}{\textnormal{Bin3}\xspace}
\newcommand{\HybS}{\textnormal{HybS}\xspace}

\newcommand{\NMDT}{\textnormal{NMDT}\xspace}
\newcommand{\DNMDT}{\textnormal{D-NMDT}\xspace}

\makeatletter
\newcommand{\fcdot}{\,\cdot\,}
\newcommand{\fcarg}[1]{\def\fc@rg{#1}\ifx\fc@rg\empty\fcdot\else\fc@rg\fi}
\makeatother
\newcommand{\abs}[1]{\lvert\fcarg{#1}\rvert}

\newcommand{\field}{\mathbbm}
\newcommand{\reals}{\field{R}}
\newcommand{\R}{\reals} %
\newcommand{\N}{\field{N}} %

\newcommand{\abbr}[1][abbrev]{#1.\ }%

\newcommand{\eg}{\abbr[e.g]}

\newcommand{\ie}{\abbr[i.e]}

\newcommand{\wrt}{\abbr[w.r.t]}
\newcommand{\st}{\mathrm{s.t.}}
\newcommand{\pwl}{\abbr[pwl]}

\renewcommand{\Set}[1]{\left\{#1\right\}}

\makeatletter
\newcommand{\objref}[4]{\def\obj@rg{#4}%
  #1\ifx\obj@rg\empty#2\else#3\xspace\ref{#4}--\fi\ref}
\makeatother
\newcommand{\Sobjref}[1]{\objref{#1}{~}{s}}

\newcommand{\Tabref}[1][]{\Sobjref{Table}{#1}}

\DeclareRobustCommand{\ubar}[1]{\text{\b{$#1$}}}

\newcommand{\define}{\coloneqq}%

\newcommand{\non}{non-}%

\DeclareMathOperator{\epi}{epi}
\DeclareMathOperator{\gra}{gra}
\DeclareMathOperator{\hyp}{hyp}

\DeclareMathOperator{\vol}{vol}

\newcommand{\proju}{\proj_{\bm u}}

\newcommand{\averageErrorWidth}{average error width }
\newcommand{\averageErrorWidths}{average error widths }

\newcounter{claims}

\usepackage[utf8]{inputenc}
\usepackage{pgfplots}
\DeclareUnicodeCharacter{2212}{−}
\usepgfplotslibrary{groupplots,dateplot}
\usetikzlibrary{patterns,shapes.arrows}
\usetikzlibrary{positioning,chains,fit,shapes,calc}
\pgfplotsset{compat=newest}

\usepackage{tikzscale}
\usepackage{filecontents}

\ExplSyntaxOn
\NewDocumentCommand{\mref}{m}{\textup{\quinn_mref:n {#1}}}
\seq_new:N \l_quinn_mref_seq
\cs_new:Npn \quinn_mref:n #1
 {
  \seq_set_split:Nnn \l_quinn_mref_seq { , } { #1 }
  \seq_pop_right:NN \l_quinn_mref_seq \l_tmpa_tl
  ( %
  \seq_map_inline:Nn \l_quinn_mref_seq
    { \ref{##1},\nobreakspace } %
  \exp_args:NV \ref \l_tmpa_tl %
  ) %
 }
\ExplSyntaxOff

\begin{document}

\title{Enhancements of Discretization Approaches for Non-Convex Mixed-Integer Quadratically Constraint Quadratic Programming: Part II\thanks{B. Beach and R. Hildebrand are supported by AFOSR grant FA9550-21-0107. Furthermore, we acknowledge financial support
by the Bavarian Ministry of Economic Affairs, Regional Development and Energy through the Center for Analytics -- Data -- Applications (ADA-Center) within the framework of ``BAYERN DIGITAL II''.}}
\titlerunning{Enhancements of Discretization Approaches for Non-Convex MIQCQPs}
\author{
Benjamin Beach\inst{1}
\and
Robert Burlacu\inst{2} %
\and
Andreas B\"armann\inst{3} %
\and
Lukas Hager \inst{3}
\and
Robert Hildebrand\inst{1}
}
\authorrunning{ 
B. Beach,
R. Burlacu,
A. B\"armann,
L. Hager,
R. Hildebrand}
\institute{
Grado Department of Industrial and Systems Engineering, Virginia Tech, Blacksburg, Virginia, USA\\
\email{\{bben6,rhil\}@vt.edu}
\and
Fraunhofer Institute for Integrated Circuits IIS, D-90411 N\"urnberg, Germany
\email{robert.burlacu@iis.fraunhofer.de}\\ \and
Friedrich-Alexander-Universit\"at Erlangen-N\"urnberg, D-91058 Erlangen, Germany 
\email{andreas.baermann@math.uni-erlangen.de,
lukas.hager@fau.de}\\
}
\maketitle              %
\begin{abstract}
This is Part II of a study on mixed-integer programming (MIP) relaxation techniques for the solution of \non convex mixed-integer quadratically constrained quadratic programs (MIQCQPs).
We set the focus on MIP relaxation methods for \non convex continuous variable products and extend the well-known MIP relaxation
\emph{normalized multiparametric disaggregation technique} (\NMDT), applying a sophisticated
 discretization to both variables.
We refer to this approach as \emph{doubly discretized normalized multiparametric disaggregation technique} (D-NMDT). 
In a comprehensive theoretical analysis, we underline the theoretical advantages of the enhanced method \DNMDT compared to \NMDT.
Furthermore, we perform a broad computational study
to demonstrate its effectiveness
in terms of producing tight dual bounds for MIQCQPs.
Finally, we compare \DNMDT to the separable MIP relaxations from Part I and a state-of-the-art MIQCQP solver.

\end{abstract}

\section{Introduction}
\label{ssec:pdisc}

In this work, we study relaxations
of general mixed-integer quadratically constrained quadratic programs (MIQCQPs).
More precisely, we consider discretization techniques for \non convex MIQCQPs
that allow for relaxations  of the set of feasible solutions
based on mixed-integer programming (MIP) formulations.

We enhance the \emph{normalized multiparametric disaggregation technique} (\NMDT) introduced in \cite{castro2015-nmdt}.
\NMDT is a  \emph{McCormick relaxation} based MIP relaxation approach, which is applied to form relaxations
of the quadratic equations $ \y = x^2 $ and $ \z = xy $. The McCormick relaxation is a set of four inequalities that describe the convex hull of the feasible points of the equation $\z=xy$ in the satisfying finite lower and upper bounds on $x$ and $y$, see \cite{McCormick1976}. 
We extend \NMDT by applying a discretization to both variables. We refer to the latter as \emph{doubly discretized \NMDT} (\DNMDT).
Both MIP formulations, \NMDT and \DNMDT, can be applied to~MIQCQPs to form an MIP relaxation
by introducing auxiliary variables and one such quadratic equation for each quadratic term 
in the MIQCQP.
Such an MIP relaxation can then be solved with a standard MIP solver.
We analyze these MIP relaxation approaches theoretically and computationally
with respect to the quality of the dual bound they deliver for MIQCQPs.

For a thorough discussion of background on discretization and piecewise linear techniques in MIQCQPs, please refer to Part I~\cite{Part_I}.

\noindent\textbf{Contribution}
We extend \NMDT by a discretization of both variables, called \DNMDT.
We analyze both MIP relaxations
in terms of the dual bound they impose for \non convex~MIQCQPs.
In a theoretical analysis, we show that \DNMDT requires fewer binary variables 
and yields better \emph{linear programming} (LP) relaxations at identical relaxation errors compared to \NMDT.
Finally, we perform an extensive numerical study where we use \NMDT and \DNMDT
to generate MIP relaxations of \non convex MIQCQPs. 
We show that \DNMDT has clear advantages, such as tighter dual bounds, shorter runtimes, and it finds more feasible solutions to the original MIQCQPs when combined with a callback function that uses the \non linear programming (NLP) solver IPOPT \cite{ipopt}.
These effects become even more apparent in dense instances with many variable products.
Moreover, we combine \NMDT and \DNMDT with the \emph{tighten sawtooth epigraph relaxation} from Part I \cite{Part_I} to obtain even tighter relaxations for $\z=x^2$ terms in MIQCQPs. This tightening leads to improved results in the computational study.

\noindent\textbf{Outline}
In \Cref{sec:preliminaries} and \Cref{sec:form-core} we review several useful concepts, notations, and core formulations from Part I~\cite{Part_I}.  
In \Cref{sec:direct}, we recall the \NMDT MIP relaxation and introduce the new MIP relaxation \DNMDT.
In \Cref{sec:theory}, we prove various properties about the strengths of the MIP relaxations focusing on volume, sharpness, and optimal choice of breakpoints.
In \Cref{sec:computations}, we present our computational study.

\section{MIP Formulations}
\label{sec:preliminaries}

We follow Part I~\cite{Part_I} for notation used in this work.  We provide this section here for completeness of this article.  

We study relaxations
of general mixed-integer quadratically constrained\\ quadratic programs (MIQCQPs),
which are defined as
\begin{equation}
    \label{eqn:generic-problem}
    \begin{array}{rll}
        \ds \min & x'Q_0 x + c_0'x + d_0'y,\\
            \text{s.t.} & x'Q_j x + c_j'x + d_j'y + b_j \le 0 \quad& j \in 1, \ldots, m,\\
            & x_i \in [\xmin_i, \xmax_i] & i \in 1, \ldots, n,\\
            & y \in \{0, 1\}^k,
    \end{array}
\end{equation}
for $ Q_0, Q_j \in \R^{n \times n} $, $ c_0, c_j \in \R^n $, $ d_0, d_j \in \R^k $
and $ b_j \in \R $, $ j = 1, \ldots m $.
Throughout this article, we use the following convenient notation:
for any two integers $ i \leq j $, we define $ \lrbr{i, j} \define \{i, i + 1, \ldots, j\} $,
and for an integer $ i \geq 1 $ we define $ \lrbr{i} \define \lrbr{1, i} $.
We will denote sets using capital letters,
variables using lower case letters
and vectors of variables using bold face.
For a vector $ \bm u = (u_1, \ldots, u_n) $ and some index set $ I \subseteq \lrbr{n} $,
we write $ \bm u_I \define (u_i)_{i \in I} $.
Thus, \eg $ \bm u_{\lrbr{i}} = (u_1, \ldots, u_i) $.
Furthermore, %
we introduce the following notation:
for a function $ \funcf \colon X \to \R $ and a subset $ B \subseteq X $,
let $ \gra_B(\funcf) $, $ \epi_B(\funcf) $ and~$ \hyp_B(\funcf) $
denote the \emph{graph} and the \emph{epigraph}
of the function~$\funcf$ over the set~$B$, respectively.
That is,
\begin{align*}
    &\gra_B(\funcf) \define \{(\bm u,\y) \in B \times \R: \y = \funcf(\bm u)\},\ \ \\
    &\epi_B(\funcf) \define \{(\bm u,\y) \in B \times \R: \y \ge \funcf(\bm u)\}.%
\end{align*}
In the following, we introduce MIP formulations as we will use them to represent these sets
as well as the different notions of the strength of an MIP formulation explored in this work.

\label{sec:fstrength}

We will study mixed-integer linear sets, so-called \emph{mixed-integer programming (MIP) formulations},
of the form
\begin{equation*}
    \PIP \define \{(\bm u, \bm v, \bm z)
		\in \R^{d + 1} \times [0, 1]^p \times \{0, 1\}^q :
		A (\bm u, \bm v, \bm z) \leq b\}
\end{equation*}
for some matrix~$A$ and vector~$b$ of suitable dimensions.
The \emph{linear programming (LP) relaxation} or \emph{continuous relaxation} $ \PLP $ of $ \PIP $
is given by
\begin{equation*}
    \PLP \define \{(\bm u, \bm v, \bm z)
		\in \R^{d + 1} \times [0, 1]^p \times [0, 1]^q :
		A (\bm u, \bm v, \bm z) \leq b\}.
\end{equation*}
We will often focus on the projections of these sets onto the variables $ \bm u $, \ie
\begin{equation}
    \proj_{\bm u}(\PIP) \define \{\bm u \in \R^{d + 1} :
        \exists (\bm v, \bm z) \in [0, 1]^p \times \{0, 1\}^q
	    \quad
	    \st
	    \quad (\bm u, \bm v, \bm z) \in \PIP\}. 
\end{equation}
The corresponding \emph{projected linear relaxation} $ \proj_{\bm u}(\PLP) $ onto the $ \bm u $-space 
is defined accordingly.

In order to assess the quality of an MIP formulation,
we will work with several possible measures of formulation strength.
First, we define notions of sharpness,
as in \cite{Beach2020-compact,Huchette:2018}.
These relate to the tightness of the LP relaxation of an MIP formulation.
Whereas properties such as total unimodularity
guarantee an LP relaxation to be a complete description for the mixed-integer points in the full space,
we are interested here in LP relaxations that are tight description of the mixed-integer points
in the projected space.  
\begin{definition}
\label{def:sharp}
    We say that the MIP formulation $ \PIP $ is \emph{sharp} if
    \begin{equation*}
        \proj_{\bm u}(\PLP) = \conv(\proju(\PIP)).
    \end{equation*}
    holds. Further, we call it \emph{hereditarily sharp} if,
    for all $ I \subseteq \lrbr{L} $ and $ \hat{\bm z} \in \{0, 1\}^{|I|} $, we have
    \begin{equation*}
        \proj_{\bm u}(\PLP|_{\bm{z}_I = \hat{\bm z}})
            = \conv\left(\proj_{\bm u}(\PIP|_{\bm{z}_I = \hat{\bm z}})\right).
    \end{equation*}
\end{definition}
Sharpness expresses a tightness at the root node of a branch-and-bound tree.  
Hereditarily sharp means that fixing any subset of binary variables to~$0$ or~$1$
preserves sharpness, and therefore this means sharpness is preserved
throughout a branch-and-bound tree. 

In this article, we study certain non-polyhedral sets $ U \subseteq \R^{d + 1} $
and will develop MIP formulations~$ \PIP $ to form relaxations of~$U$ in the projected space,
as defined in the following.
\begin{definition}
	\label{def:mipmodel_relaxation}
	For a set $ U \subseteq \R^{d + 1} $
	we say that an MIP formulation~$ \PIP $
	is an \emph{MIP relaxation} of~$U$ if
	\begin{equation*}
    	U \subseteq \proj_{\bm u}(\PIP) .
	\end{equation*}
\end{definition}
Given a function $ \funcf\colon [0, 1]^d \to \R $, we will mostly consider
\begin{equation*}
    U = \gra_{[0, 1]^d}(\funcf) \subseteq \R^{d + 1}.
\end{equation*}
In particular, we will focus on either
\begin{equation*}
    U = \{(x, \y) \in [0, 1]^2 : \y = x^2\}
        \quad \text{or} \quad U = \{(x, y, \z) \in [0, 1]^2 : \z = xy\}.
\end{equation*}
We now define several quantities to measure the error of an MIP relaxation.
\begin{definition}
    \label{def:errors}
    For an MIP relaxation $ \PIP $ of a set $ U \subseteq \R^{d + 1} $,
    let $ \bar{\bm u} \in \proju(\PIP) $.
    We then define the \emph{pointwise error} of $ \bar{\bm u} $ as
    \begin{equation*}
        \mathcal E(\bar{\bm u}, U) \define \min\{\abs{\bm u_{d + 1} - \bar {\bm u}_{d + 1}} :
            \bm u \in U, {\bm u}_{\lrbr{d}} = \bar{\bm u}_{\lrbr{d}}\}.
    \end{equation*}
    This enables us to define the following two error measures for $ \PIP $ \wrt $U$:
    \begin{enumerate}
        \item The \emph{maximum error} of $ \PIP $ \wrt $U$ is defined as
            \begin{equation*}
                \mathcal E^{\max}(\PIP, U) \define \max_{\bm u \in \proju(\PIP)}  \mathcal E(\bm u, U).
            \end{equation*}
        \item The \emph{average error} of $ \PIP $ \wrt $U$
            is defined as
            \begin{equation*}
                \mathcal E^{\text{avg}}(\PIP, U) \define \vol(\PIP \setminus U).
            \end{equation*}
    \end{enumerate}
\end{definition}
Via integral calculus, the second, volume-based error measure
can be interpreted as the average pointwise error of all points $ \bm u \in \proju(\PIP) $.
Note that whenever the volume of~$U$ is zero (\ie it is a lower-dimensional set),
the average error just reduces to the volume of~$ \PIP $.

Both of the defined error quantities for an MIP relaxation~$ \PIP $
can also be used to measure the tightness of the corresponding LP relaxation~$ \PLP $.
In \Cref{sec:theory}, we use these to compare formulations
when $ \PLP $ is not sharp.

\section{Core Relaxations}
\label{sec:form-core}%

In the definition of the MIP relaxations studied in this work, 
we will frequently consider equations of the form~$\z=xy$
for continuous or integer variables~$x$ and~$y$
within certain bounds $ D_x \define [\xmin, \xmax] $
and $ D_y \define [\ymin, \ymax] $, respectively.
To this end, we will often use the function $ F\colon D \to \R,\, F(x, y) = xy $,
$ D \define D_x \times D_y $,
and refer to the set of feasible solutions to the equation $ \z = xy $
via the graph of~$F$, \ie $ \gra_D(F) = \{(x, y, \z) \in D \times \R: \z = xy\} $.
In order to simplify the exposition, we will, for example, often write $ \gra_D(xy) $
or refer to a relaxation of the equation $ \z = xy $ instead of $ \gra_D(F) $.
We will do this similarly for
the univariate function $ f\colon D_x \to \R,\, f(x) = x^2 $
and equations of the form $ \y = x^2 $, for example.
For inequalities, like $\z\geq xy$ or $\z=x^2$, we can use the epigraph.

Furthermore,
we repeatedly make use of several ``core'' formulations
for specific sets of feasible points.
They are introduced in the following.

\subsection{McCormick Envelopes}
\label{sssec:McCormick}

The convex hull of the equation $ \z = xy $ for $ (x, y) \in D $
is given by a set of linear equations known as the McCormick envelope, see \cite{McCormick1976}:
\begin{tcolorbox}[colback = white]
\begin{equation}
\mathcal{M}(x, y) \define \left\{(x, y, \z) \in [\xmin, \xmax] \times [\ymin, \ymax] \times \R : 
    \eqref{eq:McCormick}
    \right\}.
    \label{eq:McCormick-set}
\end{equation}
\begin{equation}
    \begin{aligned}
          &\xmin \cdot y + x \cdot \ymin - \xmin \cdot \ymin \leq& \z &&\le \xmax \cdot y + x \cdot \ymin - \xmax \cdot \ymin,\\
           &\xmax \cdot y + x \cdot \ymax - \xmax \cdot \ymax \leq& \z &&\le \xmin \cdot y + x \cdot \ymax - \xmin \cdot \ymax.
    \end{aligned}
    \label{eq:McCormick}
\end{equation}
\end{tcolorbox}
In case one of the variables, here $ \beta $, is binary,
the McCormick envelope of $ \z = x \beta $ simplifies to
\begin{tcolorbox}[colback = white]
\begin{equation}
\mathcal{M}(x, \beta) = \left\{(x, \beta, \z) \in [\xmin, \xmax] \times [0, 1] \times \R : 
    \eqref{eq:McCormick-bin}
    \right\}.
    \label{eq:McCormick-bin-set}
\end{equation}
\begin{equation}
    \label{eq:McCormick-bin}
    \begin{aligned}
         \xmin \cdot \beta \le  & \, \z \le \xmax\cdot \beta, \\
  x - \xmax\cdot (1-\beta)  \le &\, \z \le  x - \xmin\cdot (1 - \beta).
    \end{aligned}
\end{equation}
\end{tcolorbox}
For univariate continuous quadratic equations $ \y = x^2 $, it simplifies to
\begin{tcolorbox}[colback = white]
\begin{equation}
\mathcal{M}(x, x) = \left\{(x, \y) \in [\xmin, \xmax] \times \R : 
    \eqref{eq:McCormick-sq}
    \right\}.
    \label{eq:McCormick-set-sq}
\end{equation}
\begin{equation}
    \label{eq:McCormick-sq}
    \begin{aligned}
         \y &\ge 2\xmin \cdot x - \xmin^2,\\
         \y &\ge 2\xmax \cdot x - \xmax^2,\\  
         \y &\le x (\xmax + \xmin) - \xmax \cdot \xmin.
    \end{aligned}
\end{equation}
\end{tcolorbox}
\subsection{Sawtooth-Based MIP Formulations}
\label{ssec:Sawtooth}

Next, we state an MIP relaxation for equations of the form $ \y \geq x^2 $
that requires only logarithmically-many auxiliary variables and constraints in the number of linear segments. 
It makes use of an elegant \pwl formulation
for $ \gra_{[0, 1]}(x^2) $ from \cite{Yarotsky-2016}
using the recursively defined \emph{sawtooth} function presented in \cite{Telgarsky2015}
to formulate the approximation of $ \gra_{[0, 1]}(x^2) $,
as described in \cite{Beach2020-compact}.
We will use this formulation to further strengthen the relaxation of $\z=x^2$ by \NMDT or \DNMDT.
To this end, we define a formulation parameterized by the depth~$ L \in \N $:
\begin{tcolorbox}[colback = white]
\begin{equation}
    S^L \define \Set{(x, \bm{g}, \bm{\alpha}) \in [0, 1] \times [0, 1]^{L + 1} \times \{0, 1\}^L :
        \eqref{eqn:sawtooth-formulation}}
\end{equation}
\begin{equation}
    \label{eqn:sawtooth-formulation}
    \begin{array}{rll}
        g_0 &= x\\
        2(g_{j - 1} - \alpha_j) &\le g_j \le 2 g_{j - 1} &\quad j = 1, \ldots, L,\\
        2(\alpha_j - g_{j - 1}) &\le g_j \le 2(1 - g_{j - 1}) &\quad j = 1, \ldots, L.
    \end{array}
\end{equation}
\end{tcolorbox}
Note that, by construction in \cite{Yarotsky-2016,Beach2020-compact},
$ S^L $ is defined such that when $ \bm \alpha \in \{0, 1\}^L $,
the relationship between~$ g_j$ and $ g_{j - 1} $
is $ g_j = \min\{2g_{j - 1}, 2(1 - g_{j - 1})\} $ for $ j = 1, \ldots, L $,
which means that it is given by the ``tooth'' function
$ G\colon [0, 1] \to [0, 1],\, G(x) = \min\{2x, 2(1 - x)\} $.
Therefore, each~$ g_j $ represents the output of a ``sawtooth'' function of~$x$,
as described in \cite{Yarotsky-2016,Telgarsky2015},
\ie when $ \bm \alpha \in \{0, 1\}^L $, we have
\begin{equation}
    g_j = G^j(x) \quad \text{for } G^j \define \underbrace{G \circ G \circ \ldots \circ G}_j.
    \label{eq:g}
\end{equation}
Now, we define the function $ \capF^L \colon [0, 1] \to [0, 1]$,
\begin{equation}
    \capF^L(x) \define x - \sum_{j = 1}^L 2^{-2j} G^j(x),
    \label{eq:def_F^j}
\end{equation}
which is a close approximation to~$ x^2 $. 

\noindent Using the relationships~\eqref{eq:g} and~\eqref{eq:def_F^j} between~$x$ and~$ \bm g $,
any constraint of the form $ \y = x^2 $ can be approximated via the function
\begin{tcolorbox}[colback = white]
$ \F^L\colon [0, 1] \times [0, 1]^{L + 1} \to [0, 1] $, 
    \begin{equation}
        \F^L(x, \bm g) = x - \sum_{j = 1}^L 2^{-2j} g_j,\quad \text{for an integer } L \geq 0.
        \label{eq:FL}
    \end{equation}
\end{tcolorbox}
\noindent

Now, we consider the LP relaxation of~$ S^L $,
where each variable~$ \alpha_j $ is relaxed to the interval $ [0, 1] $.
Then, via the constraints~\eqref{eqn:sawtooth-formulation},
we see that the weakest lower bounds on each~$ g_j $ \wrt $ g_{j - 1} $
can be attained via setting $ \alpha_j = g_{j - 1} $, yielding a lower bound of~$0$.
Thus, after projecting out $ \bm \alpha $,
the LP relaxation of $ S^L $ in terms of just~$x$ and~$ \bm g $
can be stated as 
\begin{tcolorbox}[colback = white]
\begin{equation*}
    T^L = \Set{(x, \bm{g}) \in [0, 1] \times [0, 1]^{L + 1}:
        \eqref{eqn:sawtooth-formulation-LP-constr}},
\end{equation*}
\begin{equation} \label{eqn:sawtooth-formulation-LP-constr}
\begin{array}{rll}
        g_0 &= x\\
        g_j &\le 2(1 - g_{j - 1}) & \quad j = 1, \ldots, L\\
        g_j &\le 2g_{j - 1} & \quad j = 1, \ldots, L.
\end{array} 
\end{equation}
\end{tcolorbox}

The LP relaxation $T^L$ is sharp by \cite[Theorem 1]{Part_I}.
Thus, $T^L$
yields the same lower bound on~$ \y $ as the MIP formulation $S^L$
due to sharpness and the convexity of~$ F^L $.
This allows us to define an LP outer approximation
for inequalities of the form $ \y \geq x^2 $:
\begin{definition}[Sawtooth Epigraph Relaxation, SER]
    \label{def:sawtooth-epi}
    Given some $ L \in \N $,
    the \emph{depth-$L$ sawtooth epigraph relaxation} for $ \y \geq x^2 $
    on the interval $ x \in [0, 1] $ is given by
    \begin{tcolorbox}[colback = white]
        \begin{equation}
            \label{eq:sawtooth-epi-relax}
            Q^L \define \Set{(x, \y) \in [0, 1] \times \R : \exists \bm g \in [0, 1]^{L + 1} :
                \eqref{eq:sawtooth-epi-relax-constr}},
        \end{equation}
        \begin{equation}
            \label{eq:sawtooth-epi-relax-constr}
            \begin{array}{rll}
                \y &\ge \F^j(x, \bm g) - 2^{-2j - 2} & \quad j = 0, \ldots, L\\
                \y &\ge 0,\quad
                \y \ge 2x - 1\\
                (x, \bm{g}) &\in T^L.
            \end{array}
        \end{equation}
    \end{tcolorbox}
\end{definition}

\begin{figure}
    \centering
    \begin{tikzpicture}
\pgfplotsset{%
    width=0.55\textwidth,
}
\definecolor{color0}{rgb}{0.12156862745098,0.466666666666667,0.705882352941177}
\definecolor{color1}{rgb}{1,0.498039215686275,0.0549019607843137}
\definecolor{color2}{rgb}{0.172549019607843,0.627450980392157,0.172549019607843}
\definecolor{color3}{rgb}{0.83921568627451,0.152941176470588,0.156862745098039}
\definecolor{color4}{rgb}{0.580392156862745,0.403921568627451,0.741176470588235}

\begin{axis}[
legend cell align={left},
legend style={fill opacity=0.8, draw opacity=1, text opacity=1, 
at={(2.2,0.8)}, %
draw=white!80!black},
tick align=outside,
tick pos=left,
x grid style={white!69.0196078431373!black},
xmin=-0.05, xmax=1.05,
xtick style={color=black},
xtick={0,0.125,0.25,0.375,0.5,0.625,0.75,0.875,1},
xticklabels={0,\(\displaystyle {1}/{8}\),\(\displaystyle {2}/{8}\),\(\displaystyle {3}/{8}\),\(\displaystyle {4}/{8}\),\(\displaystyle {5}/{8}\),\(\displaystyle {6}/{8}\),\(\displaystyle {7}/{8}\),1},
y grid style={white!69.0196078431373!black},
ymin=-0.1, ymax=1.1,
ytick style={color=black},
ytick={0,0.125,0.25,0.375,0.5,0.625,0.75,0.875,1},
yticklabels={0,\(\displaystyle {1}/{8}\),\(\displaystyle {2}/{8}\),\(\displaystyle {3}/{8}\),\(\displaystyle {4}/{8}\),\(\displaystyle {5}/{8}\),\(\displaystyle {6}/{8}\),\(\displaystyle {7}/{8}\),1}
]
\path [draw=white!50.1960784313725!black, fill=white!50.1960784313725!black]
(axis cs:0,0)
--(axis cs:0,0)
--(axis cs:0.125,0)
--(axis cs:0.25,0.0625)
--(axis cs:0.375,0.125)
--(axis cs:0.5,0.25)
--(axis cs:0.625,0.375)
--(axis cs:0.75,0.5625)
--(axis cs:0.875,0.75)
--(axis cs:1,1)
--(axis cs:1,1)
--(axis cs:1,1)
--(axis cs:0,1)
--(axis cs:0,0)
--cycle;
\addlegendimage{area legend, draw=white!50.1960784313725!black, fill=white!50.1960784313725!black}
\addlegendentry{SER with $L=1$}

\addplot [semithick, color0]
table {%
0 -0.25
1 0.75
};
\addlegendentry{$F^0 - 2^{-2}$}
\addplot [semithick, color1]
table {%
0 -0.0625
0.5 0.1875
1 0.9375
};
\addlegendentry{$F^1 - 2^{-4}$}
\addplot [semithick, color3]
table {%
0 0
1 0
};
\addlegendentry{$0$}
\addplot [semithick, color4]
table {%
0 -1
1 1
};
\addlegendentry{$2x-1$}
\end{axis}

\end{tikzpicture}
    \begin{tikzpicture}
\pgfplotsset{%
    width=0.55\textwidth,
}
\definecolor{color0}{rgb}{0.12156862745098,0.466666666666667,0.705882352941177}
\definecolor{color1}{rgb}{1,0.498039215686275,0.0549019607843137}
\definecolor{color2}{rgb}{0.172549019607843,0.627450980392157,0.172549019607843}
\definecolor{color3}{rgb}{0.83921568627451,0.152941176470588,0.156862745098039}
\definecolor{color4}{rgb}{0.580392156862745,0.403921568627451,0.741176470588235}

\begin{axis}[
legend cell align={left},
legend style={fill opacity=0.8, draw opacity=1, text opacity=1, 
at={(2.2,0.8)}, %
draw=white!80!black},
tick align=outside,
tick pos=left,
x grid style={white!69.0196078431373!black},
xmin=-0.05, xmax=1.05,
xtick style={color=black},
xtick={0,0.125,0.25,0.375,0.5,0.625,0.75,0.875,1},
xticklabels={0,\(\displaystyle {1}/{8}\),\(\displaystyle {2}/{8}\),\(\displaystyle {3}/{8}\),\(\displaystyle {4}/{8}\),\(\displaystyle {5}/{8}\),\(\displaystyle {6}/{8}\),\(\displaystyle {7}/{8}\),1},
y grid style={white!69.0196078431373!black},
ymin=-0.1, ymax=1.1,
ytick style={color=black},
ytick={0,0.125,0.25,0.375,0.5,0.625,0.75,0.875,1},
yticklabels={0,\(\displaystyle {1}/{8}\),\(\displaystyle {2}/{8}\),\(\displaystyle {3}/{8}\),\(\displaystyle {4}/{8}\),\(\displaystyle {5}/{8}\),\(\displaystyle {6}/{8}\),\(\displaystyle {7}/{8}\),1}
]
\path [draw=blue, fill=blue, opacity=0.2]
(axis cs:0,0)
--(axis cs:0,0)
--(axis cs:0.0625,0)
--(axis cs:0.125,0.015625)
--(axis cs:0.1875,0.03125)
--(axis cs:0.25,0.0625)
--(axis cs:0.3125,0.09375)
--(axis cs:0.375,0.140625)
--(axis cs:0.4375,0.1875)
--(axis cs:0.5,0.25)
--(axis cs:0.5625,0.3125)
--(axis cs:0.625,0.390625)
--(axis cs:0.6875,0.46875)
--(axis cs:0.75,0.5625)
--(axis cs:0.8125,0.65625)
--(axis cs:0.875,0.765625)
--(axis cs:0.9375,0.875)
--(axis cs:1,1)
--(axis cs:1,1)
--(axis cs:0,1)
--(axis cs:0,0)
--cycle;
\addlegendimage{area legend, draw=blue, fill=blue, opacity=0.2}
\addlegendentry{SER with $L=2$}

\addplot [semithick, color0]
table {%
0 -0.25
1 0.75
};
\addlegendentry{$F^0 - 2^{-2}$}
\addplot [semithick, color1]
table {%
0 -0.0625
0.5 0.1875
1 0.9375
};
\addlegendentry{$F^1 - 2^{-4}$}
\addplot [semithick, color2]
table {%
0 -0.015625
0.25 0.046875
0.5 0.234375
0.75 0.546875
1 0.984375
};
\addlegendentry{$F^2 - 2^{-6}$}
\addplot [semithick, color3]
table {%
0 0
1 0
};
\addlegendentry{$0$}
\addplot [semithick, color4]
table {%
0 -1
1 1
};
\addlegendentry{$2x-1$}
\end{axis}

\end{tikzpicture}
    \caption{The sawtooth epigraph relaxations $Q^L$ for $L=1$ and $L = 2$.
        By increasing~$ L $ , we tighten the lower bound by creating more inequalities.
        This is done by only adding linearly-many variables and inequalities in the extended formulation
        to gain exponentially-many equally spaced cuts in the projection.}
    \label{fig:SSR}
\end{figure}
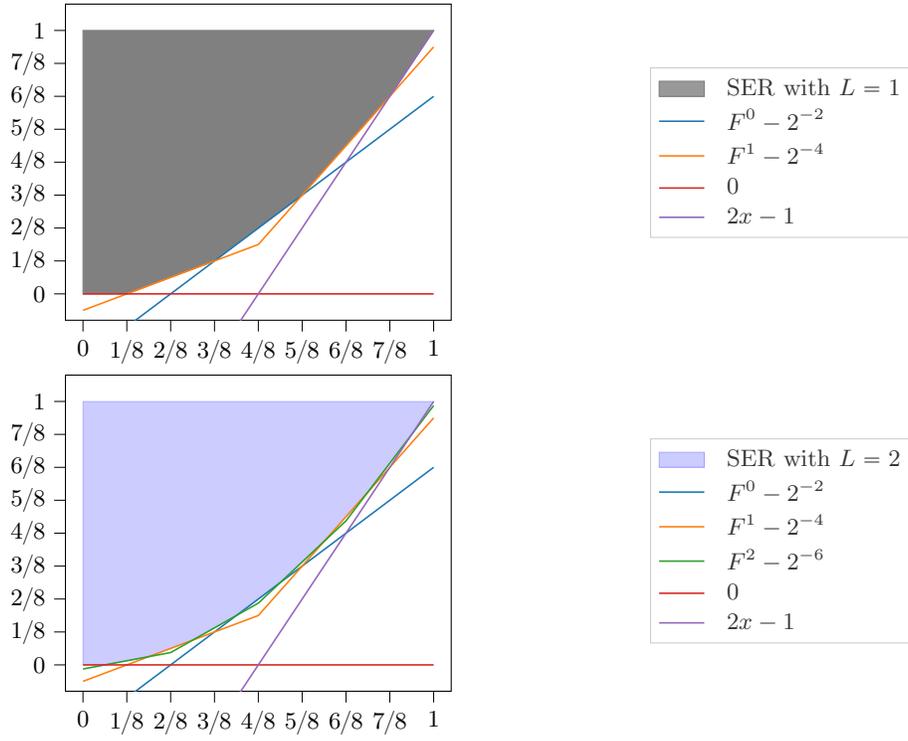
In \cite{Part_I} it is shown that
that the maximum error for the sawtooth epigraph relaxation is $ 2^{-2L - 4} $.

\section{MIP Relaxations for Non-Convex MIQCQPs}
\label{sec:direct}

In this section, we present MIP relaxations for bivariate equations of the form~$\z=xy$ and univariate equations of the form $\z = x^2$.
For convenience, we define a \emph{completely dense} MIQCQP
as an MIQCQP for which all terms of the form~$ x_i^2 $ and~$ x_i x_j $
appear in either the objective or in some constraint.

We proceed as follows.
First, we recall the well-known MIP relaxation technique \NMDT.
Then, we introduce an enhanced version of it, called \emph{\DNMDT},
which is designed to reduce the number of binary variables
required to reach the same level of approximation accuracy
compared to \emph{\NMDT}
for completely dense MIQCQPs.
Finally, we define the two \emph{tightened} variants of \NMDT and \DNMDT, for which we also incorporate the sawtooth epigrpaph relaxation~\eqref{eq:sawtooth-epi-relax} for all $\z=x_i^2$ terms. We call these methods T-\NMDT and T-\DNMDT, respectively.
We will mention the corresponding maximum errors of the presented MIP relaxations and derive them in detail in \cref{ssec:max_error}.

\subsection{Base-2 NMDT}
\label{ssec:direct-nmdt}
The Normalized Multiparametric Disaggregation Technique (\NMDT) was introduced
by Castro~\cite{castro2015-nmdt}. Later it was used in \cite{Beach2020-compact,Beach2020}.
along with its univariate form (see \cite[Appendix A]{Beach2020-compact}).
While in \cite{castro2015-nmdt} a base of 10 was chosen for the discretization, in \cite{Beach2020-compact,Beach2020} \NMDT is described with a base of 2. We use the latter here and provide both the bivariate and univariate definition of base-2 \NMDT according to~\cite{Beach2020-compact} here.

In \NMDT, the key idea for relaxing $ \z = xy $ is to discretize one variable,
\eg $x$, using binary variables $ \bm \beta \in \{0, 1\}^L $ and a residual term $ \Del x $
and then relaxing the resulting products $ \beta_i y $ and $ \Del x y $ using McCormick envelopes. 
The following derivation of \NMDT can be transferred one-to-one to bases different to~$2$.
We start with the base-2 discretization of the variable~$x$:
\begin{equation*}
    x = \dsum_{j = 1}^L 2^{-j} \beta_j + \Del x.
\end{equation*}
Then we multiply by $y$ to obtain the exact representation
\begin{equation}
    \label{eq:NMDT-xy-exact}
    \begin{array}{rll}
         x &= \dsum_{j = 1}^L 2^{-j} \beta_j + \Del x,\,
         \z = \dsum_{j = 1}^L 2^{-j} \beta_j y + \Del x y\\
         \Del x &\in [0, 2^{-L}],\,
         \bm \beta \in \{0, 1\}^L.
    \end{array}
\end{equation}
Next, we use McCormick envelopes to model all remaining product terms,
$ \beta_i y $ and $ \Del x \cdot y $, to obtain the final formulation.
\begin{definition}[NMDT, \cite{castro2015-nmdt}]
    The MIP relaxation \emph{NMDT} of $ \z = xy $ with $ x \in [0, 1] $, $ y \in [0, 1] $
    and a depth of $ L \in \N $ is defined as follows:
    \begin{tcolorbox}[colback = white]
        \begin{equation}
            \label{eq:NMDT}
            \begin{array}{rll}
                x &= \dsum_{j = 1}^L 2^{-j}\beta_j + \Del x\\
                \z &= \dsum_{j = 1}^L 2^{-j}u_j + \Del z\\
                (y, \beta_j, u_j) &\in \mathcal{M}(y, \beta_j) & j = 1, \ldots, L\\
                (\Del x, y, \Del z) &\in \mathcal{M}(\Del x, y)\\
                \Del x &\in [0, 2^{-L}], \quad
                y \in [0, 1], \quad
                \bm \beta \in \{0, 1\}^L.
            \end{array}
        \end{equation}
    \end{tcolorbox}
\end{definition}
Since McCormick envelopes are exact reformulations of the variable products
if at least one of the variables is required to be binary,
the maximum error of \NMDT with respect to $ \z = xy $
is purely due to the McCormick relaxation of $ \Del z = \Del x \cdot y $,
with a value of $ 2^{-L - 2} $.

An advantage of the \NMDT approach compared to the separable formulations from Part I %
is that it requires fewer binary variables
to reach the desired level of accuracy for \emph{bipartite} MIQCQPs,
for which the quadratic part in each constraint is of the form~$ \bm x^T Q \bm y $.
This is due to the fact that one has only to discretize either $ \bm x \in \R^n $ or $ \bm y \in \R^m $. 
Thus, to reach a maximum error of $ 2^{-2L - 2} $ for each bilinear term,
\NMDT requires only $ 2L\min\{m, n\} $ binary variables
instead of the $ L(m + n) $ variables
required by the approaches \DNMDT (see \Cref{ssec:direct-dnmdt}) or \HybS (from Part I)g.
In contrast, \NMDT requires twice the number of binary variables
to reach the same level of accuracy if all quadratic terms~$ x_i x_k $ and~$ x_l^2 $
with $ k = 1, \ldots, n $ and $ l = 1, \ldots, m $ must be modelled,
for example if~$ Q$ is dense, see \cref{tab:full-char}.

Next, we show how to model univariate quadratic equations $ \y = x^2 $ with the \NMDT technique:
\begin{definition}[Univariate NMDT (\cite{castro2015-nmdt})]
\label{def:nmdt-univ}
The MIP relaxation \emph{NMDT} of $ \y = x^2 $ with $ x \in [0, 1] $
and a depth of $ L \in \N $ is defined as follows:
\begin{tcolorbox}[colback = white]
\begin{equation}
    \label{eq:NMDT-xsq}
    \begin{array}{rll}
         x &= \dsum_{j = 1}^L 2^{-j}\beta_j + \Del x\\
         \y &= \dsum_{j = 1}^L 2^{-j}u_j + \Del \y\\
         (x, \beta_j, u_j) &\in \mathcal{M}(x, \beta_j) & j = 1, \ldots, L\\
         (\Del x, x, \Del \y) &\in \mathcal{M}(\Del x, x)\\
         \Del x &\in [0, 2^{-L}], \quad
         x \in [0, 1], \quad
         \bm \beta \in \{0, 1\}^L.
    \end{array}
\end{equation}
\end{tcolorbox}
\end{definition}
Note that for any depth $L$, the univariate formulation NMDT
yields a maximum error of slightly less than $ 2^{-L - 2} $ instead of the $ 2^{-2L - 2} $
in the sawtooth relaxation from \cite{Part_I}.
Further, the formulation NMDT is not sharp.
For example at $ x = \tfrac 12 $, its LP relaxation
admits the solution $ \beta_j = \tfrac 12 $ for all $ j \in \lrbr{L} $,
$ \Del x = 2^{-L - 1} $, $ u_j = 0 $ for all $ j \in \lrbr{L} $,
$ \Del \y = 0 $ and $ \y = 0 $,
which is not in the convex hull of $ \gra_{[0, 1]}(x^2) $.

However, we can tighten the lower bound on $\z$ in~\eqref{eq:NMDT-xsq}
by adding the sawtooth epigraph relaxation~\eqref{eq:sawtooth-epi-relax}
of depth~$ L_1 $ (with $ L_1 \ge L$),
\ie $ (x, \y) \in Q^{L_1}$. 
We refer to \NMDT with this lower-bound tightening for univariate quadratic terms as \emph{T-NMDT}.
\begin{definition}[Univariate T-NMDT]
    \label{def:T-nmdt-univ}
    The MIP relaxation \emph{T-NMDT} of $ \y = x^2 $ with $ x \in [0, 1] $
    and a depth of $ L, L_1 \in \N $ with $ L_1 \geq L $ is defined as follows:
    \begin{tcolorbox}[colback = white]
        \begin{equation}
            \label{eq:TNMDT-xsq}
            \begin{array}{rll}
                (x, \Delta_x^L, \y, \Delta_\y^L, \bm u, \bm \beta) \text{ satisfy } \eqref{eq:NMDT-xsq}\\
                (x, \y) \in Q^{L_1}.
            \end{array}
        \end{equation}
    \end{tcolorbox}
\end{definition}

\subsection{Doubly Discretized NMDT}
\label{ssec:direct-dnmdt}

The key idea behind the novel MIP relaxation \emph{Doubly Discretized NMDT (\DNMDT)} for $ \z = xy $
is to further increase the accuracy of NMDT by discretizing the second variable~$y$ as well,
which leads to a \emph{double} NMDT substitution,
namely in the $ \Del x y $-term. 
In this way, for problems where NMDT would require discretizing all $ x_i $-variables,
\eg if we have some dense constraint,
we can double the accuracy of the relaxation for the equations $ \z_{ij} = x_i x_j $
without adding additional binary variables
by taking advantage of the fact that both variables are discretized anyway.
In NMDT, we could choose to discretize either $x$ or $y$ for each equation of the form $ \z = xy $.
For \DNMDT, we consider both options of discretization, and then,
by introducing a parameter $ \lambda \in [0, 1] $,
we can model a hybrid version of the two resulting MIP relaxations.
Namely, we write
\begin{equation*}
    xy = \lambda xy + (1 - \lambda) xy,
\end{equation*}
then discretize $y$ first in the relaxation of $ \lambda xy $
and $x$ first in the relaxation of $ (1 - \lambda) xy $. 
Finally, the complete MIP relaxation \DNMDT
is obtained by relaxing the resulting products 
via McCormick envelopes
(see \Cref{ssec:derivations} for the detailed derivation).
\begin{definition}[D-NMDT]
    The MIP relaxation \emph{D-NMDT} of $ \z = xy $ with $ x, y \in [0, 1] $, a depth of $ L \in \N $
    and the parameter $ \lambda \in [0, 1] $ is defined as follows:
    \begin{tcolorbox}[colback = white]
        \begin{equation}
            \label{eq:D-NMDT}
            \begin{array}{rll}
                x &= \dsum_{j = 1}^L 2^{-j}\beta^x_j + \Del x, \quad
                y = \dsum_{j = 1}^L 2^{-j}\beta^y_j + \Del y\\
                \z &= \dsum_{j = 1}^L 2^{-j} (u_j + v_j) + \Del z\\
                \left(\lambda \Del y + (1 - \lambda) y, \beta^x_j, u_j\right) &\in \mathcal{M}\left(\lambda \Del y + (1 - \lambda) y, \beta^x_j\right) \ \ \  j = 1, \ldots, L\\
                \left((1 - \lambda) \Del x + \lambda x, \beta^y_j, v_j\right) &\in \mathcal{M}\left((1 - \lambda) \Del x + \lambda x, \beta^y_j\right) \ \ \   j = 1, \ldots, L\\
                \left(\Del x, \Del y, \Del z\right) &\in \mathcal{M}\left(\Del x, \Del y\right)\\
                \Del x, \Del y &\in [0, 2^{-L}], \quad
                x, y \in [0, 1], \quad
                \bm \beta^x, \bm \beta^y \in \{0, 1\}^L.
            \end{array}
        \end{equation}
    \end{tcolorbox}
\end{definition}
As McCormick envelopes are exact reformulations of bilinear products
if one of the variables is binary,
we only make an error
in the relaxation of the continuous variable product $ \Del x \Del y $.
This yields a maximum error of $ 2^{-2L - 2} $ for \DNMDT.
For bounds on the terms $ (1 - \lambda) \Del x + \lambda x $
and $ \lambda \Del y + (1 - \lambda) y $,
see \Cref{app:gen-bnds}.
\begin{remark}
    For our implementation of the \DNMDT technique used in \Cref{sec:computations},
    we set $\lambda = \tfrac 12 $
    for the sake of formulation symmetry in~$x$ and~$y$.
\end{remark}

To model the univariate quadratic terms with this method,
we set $ y = x $ in $ \z = xy $ and get an MIP relaxation for $ \z = x^2 $,
The resulting MIP relaxation is stronger than the univariate \NMDT approach from \Cref{def:nmdt-univ}, which we will prove later.
\begin{definition}[Univariate D-NMDT]
\label{def:D-NMDT-univ}
The MIP relaxation \emph{D-NMDT} of $ \y = x^2 $ with $ x \in [0, 1] $
and a depth of $ L \in \N $ is defined as follows:
    \begin{tcolorbox}[colback = white]
        \begin{equation}
            \label{eq:D-NMDT-xsq}
            \begin{array}{rll}
                 x &= \dsum_{j = 1}^L 2^{-j}\beta_j + \Del x\\
                 \y &= \dsum_{j = 1}^L 2^{-j} u_j + \Del \y\\
                 (\Del x + x, \beta_j, u_j) &\in \mathcal{M}(\Del x + x, \beta_j) &\quad j = 1, \ldots, L\\
                 (\Del x, \Del \y) &\in \mathcal{M}(\Del x, \Del x)\\
                 \Del x &\in [0, 2^{-L}], \quad
                 x \in [0, 1], \quad
                 \bm \beta \in \{0, 1\}^L.
            \end{array}
        \end{equation}
    \end{tcolorbox}
\end{definition}
Again, as McCormick envelopes are exact reformulations of bilinear products
if one of the variables is required to be binary,
we only make an error
in the relaxation of the continuous variable product $ \Del x \Del x $.
This yields a maximum error of $ 2^{-2L - 2} $ for univariate \DNMDT.
Note that the upper bound of this formulation
is formed by exactly the same \pwl approximation for $ \y = x^2 $
as the sawtooth formulations.
Unfortunately, the univariate \DNMDT is not sharp;
for example, at $ x = \tfrac 12$, its LP relaxation
admits the solution $ \beta_j = \tfrac 12~$ for all $ j \in \lrbr{L} $, $ \Del x = 2^{-L - 1} $,
$ \Del \y = 0 $, $ u_j = 0 $ for all $ j \in \lrbr{L} $ and $ \y = 0 $,
which is not in the convex hull of $ \gra_{[0, 1]}(x^2) $. 

To formulate a tightened version of \DNMDT,
we tighten the lower bound on $\z$ in~\eqref{eq:D-NMDT-xsq},
by removing all McCormick lower bounds
and adding the sawtooth epigraph relaxation~\eqref{eq:sawtooth-epi-relax}
of depth $ L_1 $ (with $ L_1 \ge L$).
\begin{definition}[Univariate T-D-NMDT]
    \label{def:TD-NMDT-univ}
    The MIP relaxation \emph{T-\DNMDT} of $ \y = x^2 $
    with $ x \in [0, 1] $ and depths $ L, L_1 \in \N $
    with $ L_1 \geq L $ is defined as follows:
    \begin{tcolorbox}[colback = white]
        \begin{equation}
            \label{eq:TD-NMDT-xsq}
            \begin{array}{rll}
                \begin{array}{rll}
                    (x, \Delta_x^L, \y, \Delta_\y^L, \bm u, \bm \beta) \text{ satisfy } \eqref{eq:D-NMDT-xsq}\\
                    (x, \y) \in Q^{L_1}.
                \end{array}
            \end{array}
        \end{equation}
    \end{tcolorbox}
\end{definition}

In \Cref{tab:full-char} in \Cref{sec:theory}, we give a summary
of the number of binary variables and constraints as well as the accuracy of each MIP relaxation
when applied to a dense MIQCQP of the form~\eqref{eqn:generic-problem}.
\begin{remark}[Binary Variables and Dense MIQCQPs]
    When modelling Problem~\eqref{eqn:generic-problem}
    using the MIP relaxations \NMDT and \DNMDT, for each variable $x_i$, we will need a discretization
    of the form $ x_i = \sum_{j = 1}^L 2^{-j}\beta_j + \Delta^L_{x_i} $ with $ \beta \in \{0, 1\}^L $.
    Thus, both of these formulations use $nL$~binary variables in the case of a dense MIQCQP.
    However, the improved binarizations in \DNMDT reduces the errors exponentially compared to \NMDT.
    
    Note that it is possible that some preprocessing or reformulation,
    such as via a convex quadratic reformulation (QCR)
    may improve the number of binary variables needed.
    We do not use such reformulations in this work,
    but just focus on applying our MIP relaxations as is.
\end{remark}

\section{Theoretical Analysis}
\label{sec:theory}
In this section, we give a theoretical analysis of the presented MIP relaxations
for the equation $ \z = xy $ over $ x, y \in [0, 1] $
as well as the equation $ \z = x^2 $ over $ x \in [0, 1] $, respectively,
in order to allow for a comparison of structural properties between them.
In particular, we  analyze their maximum error, \averageErrorWidths,
formulation strengths, \ie (hereditary) sharpness and LP relaxation volumes,
as well as the optimal placement of breakpoints to minimize \averageErrorWidths.
Our results are summarized in \cref{tab:full-char}, which also includes the results for the separable methods \HybS, \morsireform, and \zellmerreform from Part I \cite{Part_I}.
\begin{table}[h]
    \centering
    \def\arraystretch{1.5}
    \begin{tabular}{ccccc}
        \toprule
        MIP relax. &  \# Bin.\ variables & \# Constraints & Max.\ err. & Avg.\ err. width\\
        \midrule
        NMDT & $nL$ & $n (\tfrac 12 (5n+7) + 2(n+1)L)$ & $2^{-L-2}$ & $\tfrac 16 2^{-L}$\\
        \midrule
        \DNMDT & $nL$ & $n(\tfrac 12(5n+5) + 4nL)$ & $2^{-2L-2}$ & $\tfrac 16 2^{-2L}$\\
        \midrule
        \HybS & $nL$ & $n(\tfrac 12(5n-3) + 2n(L+L_1))$ & $2^{-2L-2}$ & $\tfrac 13 2^{-2L}$\\
        \midrule
        \morsireform & $\tfrac{1}{2}(n^2 + 1)L$ & $n(\tfrac 12(3n-1) + (n+1)(L+L_1))$ & $2^{-2L-1}$ & $\tfrac{1}{2}2^{-2L}$\\
        \midrule
        \zellmerreform & $\tfrac{1}{2}(n^2 + 1)L$ & $n(\tfrac 12(3n-1) + (n+1)(L+L_1))$ & $2^{-2L-1}$ & $\tfrac{1}{2}2^{-2L}$\\
       \bottomrule
    \end{tabular}
    \caption{A summary of characteristics of the different MIP relaxations for $\z=xy$.
        Binary variables and constraints are given in the worst-case,
        in which every possible quadratic term is modelled,
        for example if some matrix~$ Q_i $ is dense.
        The \averageErrorWidths for \HybS, \morsireform and \zellmerreform
        with respect to $ \gra_{[0, 1]^2}(xy) $ are calculated for $ L_1 \to \infty $
        and without the McCormick envelopes added.
        Finally, the \averageErrorWidths for \morsireform and \zellmerreform apply only to $ L \ge 1 $;
        the corresponding volumes are $ \tfrac{7}{12} $ for $ L = 0 $.
        Finite~$L_1$ leads to slightly increased error bounds
        for the methods \morsireform, \zellmerreform and \HybS.}%
    \label{tab:full-char}
\end{table}

\subsection{Maximum Error}
\label{ssec:max_error}
We start by discussing the maximum errors.
We will derive the maximum errors of the \NMDT-based formulations by reducing the error calculations to the error of a single McCormick relaxation per grid piece.
In general, for the equation $ \z = xy $
over a grid piece $ [\xmin, \xmax] \times [\ymin, \ymax] $,
the maximum under- and overestimation is $ \tfrac{1}{4}(\xmax - \xmin)(\ymax - \ymin) $,
attained at $ (x, y) = (\tfrac{1}{2}(\xmin + \xmax), \tfrac{1}{2}(\ymin + \ymax)) $,
see \eg \cite[page 23]{Linderoth:2005}.

For \NMDT, to show that the maximum error can be computed from a single McCormick relaxation, we fix $ \bm \beta \in \{0, 1\}^L $  in \cref{eq:NMDT}
and observe two facts: 
(1) we get $x = k 2^{-L} + \Del x$ for some integer $k$ and therefore $x$ varies only with $ \Delta_x^L \in [0, 2^{-L}]$,
and (2) the McCormick relaxation $ (y, \beta_i, u_i) \in \mathcal M(y, \beta_i) $ is exact
for each $ i = 1, \ldots, L $,
i.e., the relaxation equals $ u_i = y \beta_i $.
These two facts imply that the only error incurred on this small interval stems from the single McCormick relaxation~$ (\Del x, y, \Del z) \in \mathcal M(\Delta_x^L, y) $ over regions of the form  $(\Del x, y) \in [0, 2^{-L}] \times [0,1]$.
This yields a maximum error of $\tfrac 14(2^{-L} \cdot 1) = 2^{-L-2}$.
Similarly, for \DNMDT and univariate \NMDT and \DNMDT, one can also show that all errors come from the McCormick relaxations of the continuous error terms.
The maximum errors of the different MIP relaxations are listed in the following propositions.

\begin{proposition}
    The maximum error in the \NMDT MIP relaxation for $ \z = xy $
    with $ x, y \in [0, 1] $ is $\tfrac 14(2^{-L} \cdot 1) = 2^{-L-2}$.
\end{proposition}

Likewise, for \DNMDT, the maximum error in $ \z = xy $ is purely in the McCormick relaxation of the term $\left(\Del x, \Del y, \Del z\right) \in \mathcal{M}\left(\Del x, \Del y\right)$ over the region $(\Del x, \Del y) \in [0, 2^{-L}]\times [0, 2^{-L}]$, yielding a maximum error of $\tfrac 14(2^{-L} \cdot 2^{-L}) = 2^{-2L-2}$.

\begin{proposition}
    The maximum error in the \DNMDT MIP relaxation for $ \z = xy $
    with $ x, y \in [0, 1] $ is $\tfrac 14(2^{-L} \cdot 2^{-L}) = 2^{-2L-2}$.
\end{proposition}

For univariate \DNMDT, the maximum error in $z=x^2$  arises from the McCormick relaxation $(\Del x, \Del \y) \in \mathcal{M}(\Del x, \Del x)$ over the interval $\Del x \in [0, 2^{-L}]$, yielding a maximum error of $ 2^{-2L - 2}$.

\begin{proposition}
    The maximum error in the univariate \DNMDT MIP relaxation for $ \z = xy $
    with $ x, y \in [0, 1] $ is $ 2^{-2L - 2}$.
\end{proposition}

Finally, for univariate \NMDT, the error is incurred by the McCormick relaxation $(\Del x, x, \Del \y) \in \mathcal{M}(\Del x, x)$ over the box $(\Del x, x) \in [0, 2^{-L}] \times [0,1]$ with $x = k 2^{-L} + \Del x$ for some $k \in \{0, \dots, 2^{-L}-1\}$. 
Over this box, the error-maximizing point $ (x, \Del x) = (\tfrac 12, 2^{-L-1}) $ derived in \cite{Linderoth:2005} is not feasible,
as  $ x = \tfrac 12 $ implies $ \Del x = 0 $. %
In fact, we can show that the maximum error is slightly less than the expected $ 2^{-L - 2} $.
To prove this, we focus on the maximum error of the underestimating part of the McCormick envelope with respect to $x\Del x$ and skip the overestimating part as it works analogously.
By \cref{eq:McCormick}, the McCormick relaxation underestimator over the box $(\Del x, x) \in  [0, 2^{-L}] \times [0,1]$ is given as $$\max_{\substack{\Del x \in [0, 2^{-L}],\\k\in \{0,\ldots, 2^{L}-1\}}} \{0,\Del x - 2^{-L}(1-x)| x=k2^{-L}+ \Del x\}.$$
The underestimator is zero at points in the domain where 
\begin{equation}
\label{MCUE:1}
    \Del x \leq -2^{-L}x+2^{-L}=2^{-L}(1-2^{-L}k- \Del x)
\end{equation}
holds and $\Del x - 2^{-L}(1-2^{-L}k- \Del x)$ at the rest of the domain.
The maximum error of the McCormick underestimation is
\begin{align*}
    &\max_{\substack{\Del x \in [0, 2^{-L}],\\k\in \{0,\ldots, 2^{L}-1\}}} \{ x \Del x - \max\{0, \Del x - 2^{-L}(1-x)\} | x=k2^{-L}+ \Del x \}\\
    =&\max_{\substack{\Del x \in [0, 2^{-L}],\\k\in \{0,\ldots, 2^{L}-1\}}} \{ 2^{-L}k \Del x + (\Del x)^2 - \max\{0, \Del x - 2^{-L}(1-\Del x - k2^{-L})\} \}.%
\end{align*}
First, we determine the maximum error on the piece where the McCormick underestimator is the zero function.
In the $(\Del x, k)$ space the region described by the inequality \eqref{MCUE:1} equals $\Del x \leq \frac{2^L-k}{2^L+4^L}$.
Now suppose we are at some point in this region, then we can increase the error function $2^{-L}k \Del x + (\Del x)^2 -0$ by increasing either $k$ or $\Del x$. 
Consequently, the maximum error is attained 
if $\Del x = \frac{2^L-k}{2^L+4^L}$.
The error 
at these points can be purely expressed as a quadratic function in $k$:
$$x \Del x - 0 = (2^{-L}k+ \Del x) \Del x = \left(2^{-L}k+ \frac{2^L-k}{2^L+4^L}\right) \left(\frac{2^L-k}{2^L+4^L}\right).$$
It is maximized and symmetric at $k^*=\frac{1}{2}(2^{L}-1)= 2^{L-1} - \frac{1}{2}$.
Since $k^* \not\in \N$ for any $L\geq 1$, the maximum error is attained at $k_1=2^{L-1}-1$ and $k_2=2^{L-1}$.
It has a value of
$2^{-L-2} - 2^{-3L-2} (1+2^{-L})^{-2}$.
We can use the same reasoning for the region  $\Del x \geq \frac{2^L-k}{2^L+4^L}$ and the increase in the error function by decreasing either $k$ or $\Del x$ and obtaining the same maximum error at the same points.
The values $k_1$ and $k_2$ correspond to $$(\Del x, x)= \left( \tfrac{1}{2(2^L + 1)}, \tfrac 12 \pm \tfrac{1}{2(2^L + 1)}\right).$$
The maximum overestimation error with the McCormick envelope, where the proof works very similarly, is obtained at $(\Del x, x)=(\frac{1}{4},\frac{1}{4})$ and $(\Del x, x)=(\frac{1}{4},\frac{3}{4})$ with a value of $2^{-4}$ if $L=1$.
However, for $L \ge 2$ the value is somewhat lower, namely $2^{-L-2} - 2^{-3L-2} (1-2^{-L})^{-2}$ attained at
\begin{equation*}
(\Del x, x) = \lrp{ \tfrac{1}{2(2^L - 1)},\tfrac 12  \pm \tfrac{1}{2(2^L - 1)}}\text{ if } L \ge 2.
\end{equation*}
The maximum error is therefore set by the underestimation.
We summarize these findings in the following proposition.
\begin{proposition}
    The maximum error in the univariate \NMDT relaxation for $ \z = xy $
    with $ x, y \in [0, 1] $ is  $2^{-L-2} - 2^{-3L-2} (1+2^{-L})^{-2}$. 
\end{proposition}
A summary of the maximum error analysis results can be found in \Cref{tab:full-char}.
It should be noted that for a fixed depth~$L$,
\HybS and \DNMDT provide the smallest maximum errors
among the considered MIP relaxations in our study.

\subsection{Average Error Width and Minimizing the Average Error Width}
\label{section:MIPerror_and_volume}
In this section, we will study the \averageErrorWidth of the considered MIP relaxation. 
In \cref{def:errors} the \averageErrorWidth is defined as the volume enclosed by the projected MIP relaxation.
We consider it to be an additional measure of the quality of a MIP relaxation besides the maximum error.

For equations of the form $ \y = x^2 $, univariate \DNMDT gives piecewise McCormick relaxations.
In \cite[Proposition 5]{Beach2020-compact},
it is shown that uniform discretization is optimal for fixed numbers of breakpoints.
However, for univariate \NMDT the calculation of the volume is much more complicated, so we omit it here.

Next, we compute the \averageErrorWidths of \NMDT and \DNMDT for the equation $\z=xy$. 
Then we prove that the uniform discretizations,
which are used in the definition of \NMDT and \DNMDT,
are indeed optimal in terms of the minimizing the volume of the projected MIP relaxation
if the number of discretization points is fixed (\ie if~$L$ and~$ L_1 $ are fixed).

\begin{proposition}
    Let $ \PIP_{\NMDT} $ and $ \PIP_{\DNMDT} $ be the MIP relaxations of \NMDT and \DNMDT
    for $ \z = xy $ for some $ L \geq 0 $ as defined in~\eqref{eq:NMDT} and~\eqref{eq:D-NMDT},
    respectively.
    Their respective \averageErrorWidths are 
    \begin{equation*}
        \mathcal E^{\text{avg}}(\PIP_{\NMDT}, \gra_{[0, 1]^2}(xy)) = \tfrac 16 2^{-L-2}
    \end{equation*}
    and
    \begin{equation*}
        \mathcal E^{\text{avg}}(\PIP_{\DNMDT}, \gra_{[0, 1]^2}(xy)) = \tfrac 16 2^{-2L - 2}.
    \end{equation*}
\end{proposition}
\begin{proof}
    Note that the discretization in \NMDT and \DNMDT yields piecewise McCormick relaxations
    over a uniformly spaced grid, where each grid piece
    corresponds to some fixed integer solution $ \bm \beta^x, \bm \beta^y \in \{0, 1\}^L$,
    $ \Del x, \Del y \in [0, 2^{-L}] $. 
    The volume of of the McCormick envelope over a single grid piece is $ \tfrac 16 \lx^2 \ly^2 $,
    where~$ \lx $ is its $x$-length and $ \ly $ is its $y$-length
    (see \eg \cite[page 22]{Linderoth:2005}).
    The \averageErrorWidth is then the sum over all grid piece volumes.
    Now, for \NMDT we have $ 2^L $~grid pieces with $ \ly = 1 $ and $ \lx = 2^{-L} $,
    yielding a volume per grid piece of $ \tfrac 16 2^{-2L} $
    and thus a total volume of $ \tfrac 16 2^{-L} $. 
    Similarly, for \DNMDT we have $ 2^{2L} $~grid pieces with $ \lx = \ly=2^{-L} $,
    which yields a volume per grid piece of $ \tfrac 16 2^{-4L} $
    and thus a total volume of $ \tfrac 16 2^{-2L} $.
    \hfill \qed
\end{proof}
When applied to $ \gra_{[0, 1]^2}(xy) $, \NMDT and \DNMDT are both \emph{piecewise McCormick} relaxations,
defined as
\begin{equation*}
    \bigcup_{k \in \lrbr{n}, l \in \lrbr{m}} \mathcal M([x_{k - 1}, x_k], [y_{l - 1}, y_l]),
\end{equation*}
where we use the notation $ \mathcal M([x_{k - 1}, x_k], [y_{l - 1}, y_l]) $
to mean the McCormick envelope $ \mathcal M(x, y) $
with $ x \in [x_{k - 1}, x_k] $ and $ y \in [y_{l - 1}, y_l] $,
for $ 0 = x_0 < x_1 < \ldots < x_n = 1 $ and $ 0 = y_0 < y_1 < \ldots < y_m = 1 $.

We now prove that a uniform placement of breakpoints
minimizes the \averageErrorWidth in a piecewise McCormick relaxation.
For $ n = 2^L $ and $ m = 1 $, this yields precisely the \NMDT relaxation of depth~$L$,
and if $ n = m = 2^L $, then this yields precisely the \DNMDT relaxation of depth~$L$.
Hence, they are optimal discretizations.
The \averageErrorWidth in \NMDT is $ \frac{1}{6n} = \frac{1}{6}2^{-L} $,
and $ \frac{1}{6n^2} = \frac{1}{6}2^{-2L} $ in \DNMDT.
This follows from the proof below.

\begin{theorem}
\label{prop:opt_pw_McCormick}
    Let $ 0 = x_0 < x_1 < \ldots < x_n = 1 $ and $ 0 = y_0 < y_1 < \ldots < y_m = 1 $
    be sets of breakpoints.
    Then a uniform spacing of these breakpoints minimizes the \averageErrorWidth
    over all piecewise McCormick relaxations of~$ \gra_{[0, 1]^2}(xy) $.
\end{theorem}
\begin{proof}
    Let $ \lxk \define [x_{k - 1}, x_k] $ and $ \lyl \define [y_{l - 1}, y_l] $
    with $ k \in \lrbr{n} $ and $ \lxk \in \lrbr{m} $
    be the lengths of the grid pieces $ [x_{k - 1}, x_k] \times [y_{l - 1}, y_l] $.
    The volume of the McCormick envelope $ \mathcal M([x_{k - 1}, x_k], [y_{l - 1}, y_l]) $
    over a single grid piece is $ \tfrac16 \lxk^2 \lyl^2 $, see \cite[page 22]{Linderoth:2005}.
    Therefore, the problem of minimizing the \averageErrorWidth of a piecewise McCormick relaxation
    can be formulated as
    \begin{equation} 
    \label{eqn:pw_mccormick_opt}
        \begin{array}{rll}
            \ds \tfrac 16 \min & \sum_{i = 1}^n \sum_{j = 1}^m \lxk^2 \lyl^2\\
            \text{s.t.} & \sum_{k = 1}^n \lxk =1\\
            & \sum_{l = 1}^m \lyl =1\\
            & \lxk \geq 0 & k = 1, \ldots, n\\
            & \lyl \geq 0 & l = 1, \ldots, m.
        \end{array}
    \end{equation}
    The objective function in~\eqref{eqn:pw_mccormick_opt}
    sums the \averageErrorWidths over the single grid pieces
    while the constraints ensure that all single grid lengths sum up to~$1$
    and are greater than or equal to~$0$.
    Rewriting it to
    \begin{equation} 
    \label{eqn:pw_mccormick_opt_dec}
        \begin{array}{rll}
            \ds \tfrac 16 \min & (\sum_{i = 1}^n \lxi^2)\cdot(\sum_{j = 1}^m \lyj^2)\\
            \text{s.t.} & \sum_{i = 1}^n \lxi =1\\
            & \sum_{j = 1}^m \lyj = 1\\
            & \lxi \geq 0 & i = 1, \ldots, n\\
            & \lyj \geq 0 & j = 1, \ldots, m.
        \end{array}
    \end{equation}
    lets~\eqref{eqn:pw_mccormick_opt_dec} decompose into the two independent convex subproblems
    \begin{equation} 
    \label{eqn:pw_mccormick_opt_x}
        \begin{array}{rll}
            \ds \tfrac 16 \min & \sum_{i = 1}^n \lxi^2\\
            \text{s.t.} & \sum_{i = 1}^n \lxi =1\\
            & \lxi \geq 0 & i = 1, \ldots, n,\\
        \end{array}
    \end{equation}
    \begin{equation} 
    \label{eqn:pw_mccormick_opt_y}
        \begin{array}{rll}
            \ds \tfrac 16 \min & \sum_{j = 1}^m \lyj^2\\
            \text{s.t.} & \sum_{j = 1}^m \lyj =1\\
            & \lyj \geq 0 & j = 1, \ldots, m.
        \end{array}
    \end{equation}
    Applying the KKT conditions to~\eqref{eqn:pw_mccormick_opt_x} and~\eqref{eqn:pw_mccormick_opt_y},
    which are sufficient for global optimality here,
    directly shows that a uniform placement of the breakpoints with $ \lxi = \tfrac{1}{n} $
    and $ \lyj = \tfrac{1}{m} $ is optimal for~\eqref{eqn:pw_mccormick_opt}.
    The total \averageErrorWidth is then~$ \tfrac{1}{6nm} $.
\end{proof}
\begin{corollary}
    \label{prop:optimal_nmdt}
    Let $ 0 = x_0 < x_1 < \ldots < x_n = 1 $ and $ 0 = y_0 < y_1 = 1 $
    be sets of breakpoints with $ n = 2^L $ and $ \PIP_L $ a depth-$L$ \NMDT MIP relaxation
    of $ \gra_{[0, 1]^2}(xy) $ from \eqref{eq:NMDT}.
    Then~$ \PIP_L $ is an optimal piecewise McCormick relaxation
    with an \averageErrorWidth of $ \mathcal E^{\text{avg}}(\PIP_L, \gra_{[0, 1]^2}(xy)) = \frac{1}{6}2^{-L} $.
\end{corollary}
\begin{corollary}
    \label{prop:optimal_dnmdt}
    Let $ 0 = x_0 < x_1 < \ldots < x_n = 1 $ and $ 0 = y_0 < y_1 < \ldots < y_m = 1 $
    be sets of breakpoints with $ n = m = 2^L $ and $ \PIP_L $ a depth-$L$ \DNMDT MIP relaxation
    of $ \gra_{[0, 1]^2}(xy) $ from \eqref{eq:D-NMDT}.
    Then~$ \PIP_L $ is an optimal piecewise McCormick relaxation
    with an \averageErrorWidth of $ \mathcal E^{\text{avg}}(\PIP_L, \gra_{[0, 1]^2}(xy)) = \frac{1}{6}2^{-2L} $.
\end{corollary}

We summarize the key results of \Cref{section:MIPerror_and_volume}
in the remark below and in \Cref{tab:full-char}.
\begin{remark}[Tightness of MIP Relaxations]
For an equation $ \y = x^2 $ and a fixed depth $L$, the tightened sawtooth relaxation \cite[Definition 7]{Part_I},
and the separable formulations from Part I that employ it, have the smallest volume in the projected MIP relaxation
among all studied formulations: they are equivalent in upper bound, with a tightened lower bound,
compared to univariate \NMDT and \DNMDT.
For $ \z = xy $, \DNMDT is the tightest formulation,
as it yields the convex hull of $ \gra_D(xy) $
on each grid piece $ D = [k^x 2^{-L}, (k^x + 1) 2^{-L}] \times [k^y 2^{-L}, (k^y + 1) 2^{-L}] $,
$ k^x, k^y \in \lrbr{0, 2^L - 1} $.
Combining these facts, T-\DNMDT is the tightest relaxation presented for the full MIQCQP.\hfill$\diamond$ 
\end{remark}

\subsection{Formulation Strength: LP relaxations}
\label{subsec:strength}

In the previous section, we discussed maximum error and \averageErrorWidths
incurred from using certain discretizations.
We will now consider the strength of the resulting MIP relaxations
by analyzing their LP relaxation.
First, we will check for sharpness and later compare them via the volume of the projected LP relaxation.
Sharpness means that the projected LP relaxation
equals the convex hull of the set to be formulated, here $\gra(xy)$ or $\gra(x^2)$.
If we now consider the volume of a projected LP relaxation,
it can minimally be the volume of the convex hull,
which precisely holds if the formulation is sharp.
If a formulation is not sharp,
the volume of the projected LP relaxation yields a measure
of how much the formulation is ``not sharp''.
The volume of LP relaxation as a measure of a MIP relaxation strength
was previously used in \cite{Hager-2021}.

We start with the core formulations from \Cref{sec:form-core}.
It is well known that the McCormick relaxation
yields the convex hull of the feasible set of $ \z = xy $ over box domains $ D =[\xmin, \xmax] \times[\ymin, \ymax] $. 
Therefore, it is obviously sharp. The volume is $\nicefrac{1}{6} (\xmax-\xmin)(\ymax-\ymin)$.
In \cite{Part_I} it is further shown that the sawtooth epigraph relaxation is also sharp.
Since the epigraph of $f$ is an unbounded set, we do not discuss volume here.
Next, we look at the formulations from \cref{sec:direct}.
As shown in \Cref{ssec:direct-nmdt,ssec:direct-dnmdt},
the univariate verisions of \NMDT and \DNMDT are not sharp.
As shown in \cite{Beach2020-compact}, univariate \NMDT and therefore also univariate \DNMDT have an LP relaxation volume of $ \tfrac 14 2^{-2L} $.
The LP relaxations of \NMDT and \DNMDT for $\z = xy$ yield the McCormick envelope over $D$,
and thus they are sharp.
The LP relaxation volumes of \NMDT and \DNMDT for $\z = xy$ is thus $\nicefrac{1}{6} (\xmax-\xmin)(\ymax-\ymin)$ and independent of the choice of $L$.

\section{Computational Results}
\label{sec:computations}

In order to test the MIP relaxations from \Cref{sec:direct}
with respect to their ability to determine dual bounds,
we now perform an indicative computational study.
More precisely, we will derive MIP relaxations of \non convex MIQCQP instances.
The MIP relaxations are then solved using Gurobi \cite{gurobi} as an MIP solver to determine dual bounds and a callback function that uses the \non linear programming (NLP) solver IPOPT \cite{ipopt} to find a feasible solution for the MIQCQP.
The MIP relaxation methods are tested for several discretization depths. 
To compare the considered methods to state-of-the-art spatial branching based solvers, we also run Gurobi as an MIQCQP solver.

All instances were solved in Python 3.8.3,
via Gurobi 9.5.1 and IPOPT~3.12.13 on the `Woody' cluster,
using the ``Kaby Lake” nodes with two Xeon E3-1240 v6 chips
(4~cores, HT~disabled), running at 3.7~GHZ with 32~GB of~RAM.
For more information, see the \href{https://hpc.fau.de/systems-services/systems-documentation-instructions/clusters/woody-cluster/}{Woody~Cluster~Website of Friedrich-Alexander-Universit\"at Erlangen-N\"urnberg}.
The global relative optimality tolerance in Gurobi was set to the default value of ~0.01\%, for all MIPs and MIQCQPs.

\subsection{Study Design}
\label{sec:study_design}
In the following, we explain the design of our study and go into detail regarding the instance set as well as the various parameter configurations.
\ \\

\noindent\textbf{Instances.}
We consider a three-part benchmark set of 60 instances:
20 \non convex boxQP instances from \cite{Dong-Luo-2018,Beach2020-compact,Chen2012} and earlier works,
20 AC optimal power flow (ACOPF) instances from the NESTA benchmark set (v0.7.0) (see \cite{NESTA}),
previously used in \cite{aigner2020solving}, and 20 MIQCQP instancess from the QPLIB \cite{qplib}. 
In \cref{sec:instance_set} you will find links that contain download options and detailed descriptions of the instances. For an overview of the IDs of all instances, see \cref{table_instance}.
The benchmark set is equally divided into 30 sparse and 30 dense instances.
We refer to dense instances if either the objective function and/or at least one quadratic function in the constraint set is of the form $x^\top Q x$, where $x \in \mathbb{R}^n$ are all variables of the problem and $Q \in \mathbb{R}^{n, n}$ is a matrix with at least 25\% of its entries being nonzero.
\ \\

\noindent\textbf{Parameters.}
For each instance, we solve the resulting MIP relaxation of each method from \Cref{sec:direct} using various approximation depths of $ L \in \{1, 2, 4, 6\} $ and a time limit of 8~hours. 
All MIP relaxations are solved twice. Once in the standard versions from \Cref{sec:direct} and once with a tightened underestimator version for univariate quadratic terms where $ L_1 = \max\{2, 1.5L\} $. 
Note that the tightened MIP relaxations T-\NMDT and T-\DNMDT are equivalent to the \non tightened MIP relaxations \NMDT and \DNMDT
when applied to bilinear terms of the form $z=xy$.
However, they differ from them in that all lower bounding McCormick constraints in the univariate quadratic terms of the form $z=x^2$
are replaced by a tighter sawtooth epigraph relaxation~\eqref{eq:sawtooth-epi-relax}
as described in \Cref{ssec:direct-nmdt,ssec:direct-dnmdt}.
Furthermore, we include \HybS, the most performant separable MIP relaxation from Part I, in the study.
However, we do not apply tightening to \HybS, as it was shown in Part I that this does not result in computational improvements.

In \Tabref{tab:study_structure}, one can see an overview of the different parameters in our study.
In total, we have 24~parameter configurations for 60~original problems.
However, as we do not apply tightening to \HybS we end up with 1200 MIP~instances.
For the comparison with Gurobi as a state-of-the-art MIQCQP solver, we solve an additional 480~MIP instances and 120~MIQCQP instances.
These additional MIP instances arise from disabling the cuts in Gurobi for the winner of the \NMDT-based methods and \HybS.
The 120 MIQCQP instances are built by solving all 60 benchmark problems once with cuts enabled and once with cuts disabled.
\begin{table}[h]
\caption{In the study, we consider the parameters cuts, depth, and formulation to create MIP relaxations for 60 MIQCQP instances.}

\begin{center}
    \fbox{
\begin{minipage}[t]{.35\textwidth}
\underline{\textbf{Depth}}\\
$L= 1,\, 2,\,4,\, 6$\\
$L_1 = L$\\
Tightened:\\
$L= 1,\, 2,\,4,\, 6$\\
$L_1 = \max \{2, 1.5L\}$
\end{minipage}%
\begin{minipage}[t]{.3\textwidth}
\underline{\textbf{Formulation}}\\
\HybS\\
\NMDT\\
\DNMDT
\end{minipage}%
\begin{minipage}[t]{.3\textwidth}
\underline{\textbf{Instances}}\\
boxQP (20 instances)\\
ACOPF (20 instances)\\
QPLIB (20 instances)
\end{minipage}%
}
\end{center}
\label{tab:study_structure}
\end{table}
See Subsection~\ref{subsec:comp_with_gurobi} for more details on the latter.
\ \\

\noindent\textbf{Callback function.}
Solving all MIP relaxations, we use a callback function with the local NLP solver IPOPT that works as follows: given any MIP-feasible solution,
the callback function fixes any integer variables from the original problem (before applying any of the discretization techniques from this work) according to this solution and then solves the resulting NLP locally via IPOPT in an attempt to find a feasible solution for the original MIQCQP problem.

\subsection{Results \label{sec:results}} 
In the following, we present the results of our study.
In particular, we aim to answer the following questions regarding dual bounds:
\begin{itemize}
    \item Is our enhanced method \DNMDT computationally superior to its predecessors \NMDT?
    \item Is it beneficial to use tightened versions of the \NMDT and \DNMDT, i.e., to choose $L_1>L$?
    \item How do the studied methods compare to the state-of-the-art MIQCQP solver Gurobi?
\end{itemize}

We provide performance profile plots as proposed by Dolan and More \cite{Dolan2002} to illustrate the results of the computational study regarding the dual bounds, see \cref{dnmdt_vs_nmdt_all} - \cref{gurobi_dense}.
The performance profiles work as follows:
Let $d_{p,s}$ be the best dual bound obtained by MIP relaxation or MIQCQP solver $s$ for instance $p$ after a certain time limit. 
With the performance ratio $r_{p,s} \define d_{p,s} / \min_s d_{p,s}$, the performance profile function value $P(\tau)$ is the percentage of problems solved by approach $s$ such that the ratios $r_{p,s}$ are within a factor $\tau \in \mathbb{R}$ of the best possible ratios.
All performance profiles are generated with the help of \emph{Perprof-py} by Siqueira et al. \cite{perprof}.
The plots are divided into two blocks, one for \NMDT-based methods and one for the comparison against \HybS and Gurobi as an MIQCQP solver. 
In addition to the performance profiles across all instances, we also show performance profiles for the dense and sparse subsets of the instance set.

Although the main criterion of the study is the dual bound, we also discuss run times. Here, we use the shifted geometric mean, which is a common measure for comparing two different MIP-based solution approaches. The shifted geometric mean of $n$ numbers $t_1,\ldots, t_n$ with shift $s$ is defined as $\big(\prod_{i=1}^n (t_i+s)\big)^{1/n} - s$.
It has the advantage that it is neither affected by very large outliers (in contrast to the arithmetic mean) nor by very small outliers (in contrast to the geometric mean).
We use a typical shift $s = 10$.
Moreover, we only include those instances in the computation of the shifted geometric mean, where at least one solution method delivered an optimal solution within the run time limit of $8$ hours.

Finally, we will highlight some important results regarding primal bounds in the comparison of our methods with Gurobi \cite{gurobi} as an MIQCQP solver.

\subsubsection{NMDT-based MIP relaxations}

We start our analysis of the results by looking at the \NMDT-based MIP relaxations.
In \cref{dnmdt_vs_nmdt_all} we show performance profiles for the dual bounds that are obtained by the different \NMDT-based MIP relaxations.
The plot is based on all 60 instances of the benchmark set.
Starting from $L=2$, we can see that both \DNMDT and T-\DNMDT deliver notably tighter bounds within the run time limit of 8 hours.
The largest difference is at $L=4$, where \DNMDT and T-\DNMDT are able to find dual bounds that are within a factor $1.05$ of the overall best bounds for nearly all instances.
In contrast, \NMDT and T-\NMDT require a corresponding factor of more than $1.1$.
In addition, the tightened versions perform somewhat better than the corresponding counterparts, especially for $L=4$.

\begin{figure}[h]
    \begin{center}
        \begin{minipage}{0.475\textwidth}
            \begin{center}
                 \begin{tikzpicture}
  \begin{axis}[const plot,
  cycle list={
  {blue!40!gray,solid},
  {red,dashed},
  {black,dotted},
  {brown,dashdotted}},
    xmin=1, xmax=1.25,
    ymin=-0.003, ymax=1.003,
    ymajorgrids,
    ytick={0,0.2,0.4,0.6,0.8,1.0},
    xlabel={$\tau$},
    ylabel={$P(\tau)$},
,%
    legend pos={south east},
    legend style={font=\tiny},
    width=\textwidth
    ]
\draw node[right,draw,align=left] {$L=1$\\};

  \addplot+[mark=none, thick, green!40!gray, dashed] coordinates {
    (1.0000,0.4333)
    (1.0000,0.5000)
    (1.0003,0.5000)
    (1.0005,0.5167)
    (1.0006,0.5333)
    (1.0007,0.5333)
    (1.0008,0.5500)
    (1.0014,0.5500)
    (1.0015,0.5667)
    (1.0017,0.5833)
    (1.0020,0.5833)
    (1.0022,0.6000)
    (1.0030,0.6000)
    (1.0032,0.6167)
    (1.0040,0.6167)
    (1.0046,0.6333)
    (1.0059,0.6333)
    (1.0060,0.6500)
    (1.0061,0.6667)
    (1.0083,0.6667)
    (1.0086,0.6833)
    (1.0089,0.7000)
    (1.0093,0.7167)
    (1.0096,0.7333)
    (1.0105,0.7500)
    (1.0117,0.7500)
    (1.0119,0.7667)
    (1.0178,0.7667)
    (1.0185,0.7833)
    (1.0254,0.8000)
    (1.0306,0.8000)
    (1.0335,0.8167)
    (1.0336,0.8167)
    (1.0364,0.8333)
    (1.0428,0.8333)
    (1.0430,0.8500)
    (1.0438,0.8500)
    (1.0511,0.8667)
    (1.0525,0.8833)
    (1.0627,0.8833)
    (1.0634,0.9000)
    (1.0645,0.9167)
    (1.0692,0.9167)
    (1.0708,0.9333)
    (1.0726,0.9333)
    (1.0784,0.9500)
    (1.0789,0.9500)
    (1.0795,0.9667)
    (1.0820,0.9667)
    (1.0820,0.9833)
    (1.2154,0.9833)
    (1.25,1.0000)
  };
  \addlegendentry{D-NMDT}
  \addplot+[mark=none, thick, blue!40!gray, dashed] coordinates {
    (1.0000,0.3833)
    (1.0000,0.4333)
    (1.0001,0.4333)
    (1.0003,0.4500)
    (1.0006,0.4500)
    (1.0007,0.4667)
    (1.0008,0.4667)
    (1.0010,0.4833)
    (1.0012,0.4833)
    (1.0014,0.5000)
    (1.0019,0.5000)
    (1.0020,0.5167)
    (1.0032,0.5167)
    (1.0034,0.5333)
    (1.0035,0.5500)
    (1.0040,0.5667)
    (1.0062,0.5667)
    (1.0067,0.5833)
    (1.0072,0.5833)
    (1.0078,0.6000)
    (1.0083,0.6167)
    (1.0105,0.6167)
    (1.0108,0.6333)
    (1.0116,0.6500)
    (1.0117,0.6667)
    (1.0119,0.6667)
    (1.0122,0.6833)
    (1.0125,0.7000)
    (1.0126,0.7000)
    (1.0152,0.7167)
    (1.0178,0.7333)
    (1.0335,0.7333)
    (1.0336,0.7500)
    (1.0364,0.7500)
    (1.0364,0.7667)
    (1.0373,0.7833)
    (1.0377,0.8000)
    (1.0553,0.8000)
    (1.0592,0.8167)
    (1.0610,0.8167)
    (1.0627,0.8333)
    (1.0676,0.8333)
    (1.0692,0.8500)
    (1.0784,0.8500)
    (1.0789,0.8667)
    (1.0805,0.8667)
    (1.0807,0.8833)
    (1.0820,0.8833)
    (1.0968,0.9000)
    (1.1000,0.9167)
    (1.1012,0.9167)
    (1.1065,0.9333)
    (1.1166,0.9500)
    (1.1693,0.9667)
    (1.1992,0.9833)
    (1.2065,0.9833)
    (1.2110,1.0000)
    (1.25,1.0000)
  };
  \addlegendentry{NMDT}
  \addplot+[mark=none, thick, green!40!gray, densely dotted] coordinates {
    (1.0000,0.5500)
    (1.0000,0.5667)
    (1.0001,0.5833)
    (1.0006,0.5833)
    (1.0006,0.6000)
    (1.0010,0.6000)
    (1.0012,0.6167)
    (1.0020,0.6167)
    (1.0022,0.6333)
    (1.0022,0.6500)
    (1.0030,0.6667)
    (1.0051,0.6667)
    (1.0059,0.6833)
    (1.0061,0.6833)
    (1.0061,0.7000)
    (1.0062,0.7167)
    (1.0067,0.7167)
    (1.0072,0.7333)
    (1.0078,0.7333)
    (1.0078,0.7500)
    (1.0117,0.7500)
    (1.0119,0.7667)
    (1.0254,0.7667)
    (1.0254,0.7833)
    (1.0303,0.8000)
    (1.0336,0.8000)
    (1.0364,0.8167)
    (1.0380,0.8167)
    (1.0387,0.8333)
    (1.0430,0.8333)
    (1.0438,0.8500)
    (1.0525,0.8500)
    (1.0536,0.8667)
    (1.0550,0.8833)
    (1.0592,0.8833)
    (1.0610,0.9000)
    (1.0645,0.9000)
    (1.0645,0.9167)
    (1.0726,0.9167)
    (1.0784,0.9333)
    (1.0807,0.9333)
    (1.0820,0.9500)
    (1.2154,0.9500)
    (1.25,1.0000)
  };
  \addlegendentry{T-D-NMDT}
  \addplot+[mark=none, thick, blue!40!gray, densely dotted] coordinates {
    (1.0000,0.6333)
    (1.0000,0.6500)
    (1.0007,0.6500)
    (1.0007,0.6667)
    (1.0012,0.6667)
    (1.0014,0.6833)
    (1.0017,0.6833)
    (1.0019,0.7000)
    (1.0034,0.7000)
    (1.0035,0.7167)
    (1.0046,0.7167)
    (1.0051,0.7333)
    (1.0078,0.7333)
    (1.0083,0.7500)
    (1.0125,0.7500)
    (1.0126,0.7667)
    (1.0152,0.7833)
    (1.0303,0.7833)
    (1.0306,0.8000)
    (1.0377,0.8000)
    (1.0380,0.8167)
    (1.0387,0.8167)
    (1.0428,0.8333)
    (1.0550,0.8333)
    (1.0553,0.8500)
    (1.0645,0.8500)
    (1.0676,0.8667)
    (1.0692,0.8833)
    (1.0708,0.8833)
    (1.0726,0.9000)
    (1.0795,0.9000)
    (1.0805,0.9167)
    (1.1000,0.9167)
    (1.1012,0.9333)
    (1.1992,0.9333)
    (1.2065,0.9500)
    (1.2110,0.9500)
    (1.2154,0.9667)
    (1.25,1.0000)
  };
  \addlegendentry{T-NMDT}
  \end{axis}
\end{tikzpicture}
            \end{center}
        \end{minipage}
        \hfill
        \begin{minipage}{0.475\textwidth}
            \begin{center}
                 \begin{tikzpicture}
  \begin{axis}[const plot,
  cycle list={
  {blue!40!gray,solid},
  {red,dashed},
  {black,dotted},
  {brown,dashdotted}},
    xmin=1, xmax=1.25,
    ymin=-0.003, ymax=1.003,
    ymajorgrids,
    ytick={0,0.2,0.4,0.6,0.8,1.0},
    xlabel={$\tau$},
    ylabel={$P(\tau)$},
,
    legend pos={south east},
    legend style={font=\tiny},
    width=\textwidth
    ]
  \addplot+[mark=none, thick, green!40!gray, dashed] coordinates {
    (1.0000,0.3833)
    (1.0000,0.4333)
    (1.0001,0.4667)
    (1.0003,0.4667)
    (1.0005,0.4833)
    (1.0010,0.5000)
    (1.0016,0.5000)
    (1.0016,0.5167)
    (1.0023,0.5167)
    (1.0028,0.5333)
    (1.0029,0.5500)
    (1.0032,0.5667)
    (1.0032,0.5833)
    (1.0035,0.5833)
    (1.0036,0.6000)
    (1.0047,0.6000)
    (1.0048,0.6167)
    (1.0051,0.6167)
    (1.0051,0.6333)
    (1.0054,0.6333)
    (1.0056,0.6500)
    (1.0060,0.6500)
    (1.0063,0.6667)
    (1.0107,0.6667)
    (1.0107,0.6833)
    (1.0110,0.7000)
    (1.0114,0.7000)
    (1.0130,0.7167)
    (1.0172,0.7167)
    (1.0175,0.7333)
    (1.0209,0.7333)
    (1.0223,0.7500)
    (1.0254,0.7500)
    (1.0255,0.7667)
    (1.0263,0.7833)
    (1.0295,0.7833)
    (1.0300,0.8000)
    (1.0364,0.8167)
    (1.0469,0.8167)
    (1.0470,0.8333)
    (1.0676,0.8333)
    (1.0676,0.8500)
    (1.0686,0.8500)
    (1.0706,0.8667)
    (1.0971,0.8667)
    (1.1072,0.8833)
    (1.1257,0.8833)
    (1.1271,0.9000)
    (1.1548,0.9000)
    (1.1620,0.9167)
    (1.1633,0.9167)
    (1.1682,0.9333)
    (1.1766,0.9333)
    (1.1766,0.9500)
    (1.1858,0.9500)
    (1.1918,0.9667)
    (1.2351,0.9667)
    (1.2448,0.9833)
    (1.25,1.0000)
  };
  \addlegendentry{D-NMDT}
  \addplot+[mark=none, thick, blue!40!gray, dashed] coordinates {
    (1.0000,0.2500)
    (1.0000,0.3000)
    (1.0010,0.3000)
    (1.0012,0.3167)
    (1.0016,0.3333)
    (1.0019,0.3333)
    (1.0023,0.3500)
    (1.0032,0.3500)
    (1.0032,0.3667)
    (1.0035,0.3833)
    (1.0036,0.3833)
    (1.0041,0.4000)
    (1.0044,0.4167)
    (1.0046,0.4333)
    (1.0051,0.4333)
    (1.0054,0.4500)
    (1.0058,0.4500)
    (1.0060,0.4667)
    (1.0066,0.4667)
    (1.0066,0.4833)
    (1.0088,0.4833)
    (1.0090,0.5000)
    (1.0097,0.5167)
    (1.0110,0.5167)
    (1.0114,0.5333)
    (1.0130,0.5333)
    (1.0138,0.5500)
    (1.0159,0.5667)
    (1.0172,0.5833)
    (1.0179,0.5833)
    (1.0209,0.6000)
    (1.0266,0.6000)
    (1.0266,0.6167)
    (1.0295,0.6167)
    (1.0300,0.6333)
    (1.0364,0.6333)
    (1.0364,0.6500)
    (1.0461,0.6667)
    (1.0470,0.6667)
    (1.0471,0.6833)
    (1.0606,0.6833)
    (1.0652,0.7000)
    (1.0664,0.7167)
    (1.0676,0.7167)
    (1.0686,0.7333)
    (1.0706,0.7333)
    (1.0811,0.7500)
    (1.0815,0.7500)
    (1.0854,0.7667)
    (1.0861,0.7667)
    (1.0870,0.7833)
    (1.0934,0.7833)
    (1.0959,0.8000)
    (1.0971,0.8167)
    (1.1072,0.8333)
    (1.1096,0.8333)
    (1.1123,0.8500)
    (1.1125,0.8667)
    (1.1248,0.8667)
    (1.1257,0.8833)
    (1.1271,0.8833)
    (1.1285,0.9000)
    (1.1326,0.9167)
    (1.1548,0.9333)
    (1.1813,0.9333)
    (1.1853,0.9500)
    (1.1918,0.9500)
    (1.1949,0.9667)
    (1.2014,0.9833)
    (1.2079,0.9833)
    (1.2229,1.0000)
    (1.25,1.0000)
  };
  \addlegendentry{NMDT}
  \addplot+[mark=none, thick, green!40!gray, densely dotted] coordinates {
    (1.0000,0.5167)
    (1.0000,0.5667)
    (1.0001,0.6000)
    (1.0003,0.6000)
    (1.0003,0.6167)
    (1.0016,0.6167)
    (1.0019,0.6333)
    (1.0032,0.6333)
    (1.0032,0.6667)
    (1.0046,0.6667)
    (1.0047,0.6833)
    (1.0048,0.6833)
    (1.0051,0.7000)
    (1.0056,0.7000)
    (1.0058,0.7167)
    (1.0063,0.7167)
    (1.0066,0.7333)
    (1.0088,0.7333)
    (1.0088,0.7500)
    (1.0097,0.7500)
    (1.0107,0.7667)
    (1.0175,0.7667)
    (1.0179,0.7833)
    (1.0209,0.7833)
    (1.0223,0.8000)
    (1.0226,0.8167)
    (1.0254,0.8333)
    (1.0266,0.8333)
    (1.0295,0.8500)
    (1.0300,0.8667)
    (1.0461,0.8667)
    (1.0469,0.8833)
    (1.0508,0.8833)
    (1.0550,0.9000)
    (1.0664,0.9000)
    (1.0676,0.9167)
    (1.0971,0.9167)
    (1.1072,0.9333)
    (1.1096,0.9500)
    (1.1682,0.9500)
    (1.1766,0.9667)
    (1.2229,0.9667)
    (1.2351,0.9833)
    (1.2448,0.9833)
    (1.25,1.0000)
  };
  \addlegendentry{T-D-NMDT}
  \addplot+[mark=none, thick, blue!40!gray, densely dotted] coordinates {
    (1.0000,0.4833)
    (1.0000,0.5833)
    (1.0001,0.5833)
    (1.0003,0.6000)
    (1.0019,0.6000)
    (1.0019,0.6167)
    (1.0023,0.6333)
    (1.0051,0.6333)
    (1.0054,0.6500)
    (1.0066,0.6500)
    (1.0088,0.6667)
    (1.0090,0.6833)
    (1.0130,0.6833)
    (1.0138,0.7000)
    (1.0159,0.7167)
    (1.0179,0.7167)
    (1.0209,0.7333)
    (1.0263,0.7333)
    (1.0266,0.7500)
    (1.0471,0.7500)
    (1.0508,0.7667)
    (1.0508,0.7833)
    (1.0550,0.7833)
    (1.0606,0.8000)
    (1.0811,0.8000)
    (1.0815,0.8167)
    (1.0854,0.8167)
    (1.0861,0.8333)
    (1.0870,0.8333)
    (1.0915,0.8500)
    (1.0934,0.8667)
    (1.1125,0.8667)
    (1.1131,0.8833)
    (1.1187,0.9000)
    (1.1237,0.9167)
    (1.1248,0.9333)
    (1.1620,0.9333)
    (1.1633,0.9500)
    (1.1766,0.9500)
    (1.1813,0.9667)
    (1.1853,0.9667)
    (1.1858,0.9833)
    (1.2014,0.9833)
    (1.2079,1.0000)
    (1.25,1.0000)
  };
  \addlegendentry{T-NMDT}
  \draw node[right,draw,align=left] {$L=2$\\};
  \end{axis}
\end{tikzpicture}
            \end{center}
        \end{minipage}
    \end{center}
    \hfill
    \begin{center}
        \begin{minipage}{0.475\textwidth}
            \begin{center}
                 \begin{tikzpicture}
  \begin{axis}[const plot,
  cycle list={
  {blue!40!gray,solid},
  {red,dashed},
  {black,dotted},
  {brown,dashdotted}},
    xmin=1, xmax=1.25,
    ymin=-0.003, ymax=1.003,
    ymajorgrids,
    ytick={0,0.2,0.4,0.6,0.8,1.0},
    xlabel={$\tau$},
    ylabel={$P(\tau)$},
,
    legend pos={south east},
    legend style={font=\tiny},
    width=\textwidth
    ]
  \addplot+[mark=none, thick, green!40!gray, dashed] coordinates {
    (1.0000,0.4167)
    (1.0000,0.4667)
    (1.0001,0.4667)
    (1.0001,0.5000)
    (1.0002,0.5167)
    (1.0003,0.5167)
    (1.0003,0.5333)
    (1.0004,0.5333)
    (1.0004,0.5667)
    (1.0009,0.5667)
    (1.0009,0.5833)
    (1.0010,0.5833)
    (1.0011,0.6000)
    (1.0012,0.6000)
    (1.0013,0.6167)
    (1.0014,0.6167)
    (1.0014,0.6333)
    (1.0016,0.6500)
    (1.0017,0.6667)
    (1.0025,0.6667)
    (1.0025,0.6833)
    (1.0035,0.6833)
    (1.0036,0.7000)
    (1.0038,0.7167)
    (1.0040,0.7333)
    (1.0043,0.7333)
    (1.0043,0.7500)
    (1.0044,0.7667)
    (1.0055,0.7667)
    (1.0058,0.7833)
    (1.0059,0.8000)
    (1.0061,0.8000)
    (1.0064,0.8167)
    (1.0064,0.8333)
    (1.0086,0.8333)
    (1.0087,0.8500)
    (1.0119,0.8500)
    (1.0124,0.8667)
    (1.0127,0.8667)
    (1.0137,0.8833)
    (1.0155,0.9000)
    (1.0190,0.9000)
    (1.0198,0.9167)
    (1.0256,0.9167)
    (1.0292,0.9333)
    (1.0308,0.9333)
    (1.0329,0.9500)
    (1.0337,0.9500)
    (1.0362,0.9667)
    (1.0754,0.9667)
    (1.0765,0.9833)
    (1.2310,0.9833)
    (1.2358,1.0000)
    (1.3760,1.0000)
  };
  \addlegendentry{D-NMDT}
  \addplot+[mark=none, thick, blue!40!gray, dashed] coordinates {
    (1.0000,0.2667)
    (1.0000,0.3333)
    (1.0001,0.3333)
    (1.0001,0.3667)
    (1.0003,0.3667)
    (1.0004,0.3833)
    (1.0008,0.3833)
    (1.0009,0.4000)
    (1.0010,0.4167)
    (1.0013,0.4167)
    (1.0014,0.4333)
    (1.0018,0.4333)
    (1.0018,0.4500)
    (1.0021,0.4500)
    (1.0021,0.4667)
    (1.0023,0.4833)
    (1.0025,0.5000)
    (1.0040,0.5000)
    (1.0043,0.5167)
    (1.0044,0.5167)
    (1.0048,0.5333)
    (1.0049,0.5333)
    (1.0050,0.5500)
    (1.0059,0.5500)
    (1.0060,0.5667)
    (1.0061,0.5833)
    (1.0064,0.5833)
    (1.0066,0.6000)
    (1.0073,0.6167)
    (1.0077,0.6333)
    (1.0079,0.6500)
    (1.0086,0.6667)
    (1.0087,0.6667)
    (1.0097,0.6833)
    (1.0112,0.7000)
    (1.0114,0.7000)
    (1.0119,0.7167)
    (1.0222,0.7167)
    (1.0235,0.7333)
    (1.0251,0.7333)
    (1.0256,0.7500)
    (1.0292,0.7500)
    (1.0295,0.7667)
    (1.0296,0.7833)
    (1.0308,0.8000)
    (1.0329,0.8000)
    (1.0337,0.8167)
    (1.0470,0.8167)
    (1.0574,0.8333)
    (1.0605,0.8333)
    (1.0627,0.8500)
    (1.0645,0.8667)
    (1.0765,0.8667)
    (1.0819,0.8833)
    (1.0842,0.8833)
    (1.0871,0.9000)
    (1.0881,0.9167)
    (1.1020,0.9167)
    (1.1043,0.9333)
    (1.1169,0.9333)
    (1.1197,0.9500)
    (1.1216,0.9667)
    (1.3760,1.0000)
  };
  \addlegendentry{NMDT}
  \addplot+[mark=none, thick, green!40!gray, densely dotted] coordinates {
    (1.0000,0.5667)
    (1.0000,0.6333)
    (1.0001,0.6500)
    (1.0001,0.7000)
    (1.0002,0.7000)
    (1.0002,0.7167)
    (1.0003,0.7333)
    (1.0003,0.7500)
    (1.0004,0.7500)
    (1.0004,0.7833)
    (1.0009,0.7833)
    (1.0009,0.8167)
    (1.0029,0.8167)
    (1.0030,0.8333)
    (1.0034,0.8333)
    (1.0035,0.8500)
    (1.0050,0.8500)
    (1.0050,0.8667)
    (1.0126,0.8667)
    (1.0127,0.8833)
    (1.0198,0.8833)
    (1.0222,0.9000)
    (1.0235,0.9000)
    (1.0237,0.9167)
    (1.0251,0.9333)
    (1.0574,0.9333)
    (1.0605,0.9500)
    (1.1043,0.9500)
    (1.1102,0.9667)
    (1.1216,0.9667)
    (1.1638,0.9833)
    (1.2234,1.0000)
    (1.3760,1.0000)
  };
  \addlegendentry{T-D-NMDT}
  \addplot+[mark=none, thick, blue!40!gray, densely dotted] coordinates {
    (1.0000,0.3000)
    (1.0000,0.4667)
    (1.0001,0.4667)
    (1.0001,0.5000)
    (1.0004,0.5000)
    (1.0004,0.5167)
    (1.0008,0.5333)
    (1.0011,0.5333)
    (1.0012,0.5500)
    (1.0017,0.5500)
    (1.0018,0.5667)
    (1.0021,0.5833)
    (1.0025,0.5833)
    (1.0026,0.6000)
    (1.0029,0.6167)
    (1.0030,0.6167)
    (1.0031,0.6333)
    (1.0034,0.6500)
    (1.0048,0.6500)
    (1.0048,0.6667)
    (1.0049,0.6833)
    (1.0050,0.6833)
    (1.0055,0.7000)
    (1.0073,0.7000)
    (1.0077,0.7167)
    (1.0112,0.7167)
    (1.0113,0.7333)
    (1.0114,0.7500)
    (1.0124,0.7500)
    (1.0126,0.7667)
    (1.0155,0.7667)
    (1.0162,0.7833)
    (1.0190,0.8000)
    (1.0362,0.8000)
    (1.0367,0.8167)
    (1.0470,0.8333)
    (1.0645,0.8333)
    (1.0703,0.8500)
    (1.0743,0.8667)
    (1.0754,0.8833)
    (1.0819,0.8833)
    (1.0842,0.9000)
    (1.0881,0.9000)
    (1.0921,0.9167)
    (1.1020,0.9333)
    (1.1102,0.9333)
    (1.1169,0.9500)
    (1.2234,0.9500)
    (1.2310,0.9667)
    (1.2358,0.9667)
    (1.25,1.0000)
  };
  \addlegendentry{T-NMDT}
  \draw node[right,draw,align=left] {$L=4$\\};
  \end{axis}
\end{tikzpicture}
            \end{center}
        \end{minipage}
        \hfill
        \begin{minipage}{0.475\textwidth}
            \begin{center}
                 \begin{tikzpicture}
  \begin{axis}[const plot,
  cycle list={
  {blue!40!gray,solid},
  {red,dashed},
  {black,dotted},
  {brown,dashdotted}},
    xmin=1, xmax=1.25,
    ymin=-0.003, ymax=1.003,
    ymajorgrids,
    ytick={0,0.2,0.4,0.6,0.8,1.0},
     xlabel={$\tau$},
    ylabel={$P(\tau)$},
,
    legend pos={south east},
    legend style={font=\tiny},
    width=\textwidth
    ]
  \addplot+[mark=none, thick, green!40!gray, dashed] coordinates {
    (1.0000,0.4333)
    (1.0000,0.5167)
    (1.0001,0.5167)
    (1.0001,0.6000)
    (1.0003,0.6000)
    (1.0003,0.6667)
    (1.0004,0.6667)
    (1.0004,0.6833)
    (1.0006,0.6833)
    (1.0008,0.7000)
    (1.0009,0.7000)
    (1.0014,0.7167)
    (1.0019,0.7167)
    (1.0019,0.7333)
    (1.0026,0.7333)
    (1.0027,0.7500)
    (1.0030,0.7500)
    (1.0036,0.7667)
    (1.0038,0.7667)
    (1.0039,0.7833)
    (1.0043,0.7833)
    (1.0044,0.8000)
    (1.0111,0.8000)
    (1.0120,0.8167)
    (1.0136,0.8167)
    (1.0136,0.8333)
    (1.0145,0.8500)
    (1.0181,0.8500)
    (1.0186,0.8667)
    (1.0222,0.8667)
    (1.0229,0.8833)
    (1.0250,0.9000)
    (1.0259,0.9167)
    (1.0326,0.9167)
    (1.0354,0.9333)
    (1.0685,0.9333)
    (1.0696,0.9500)
    (1.0983,0.9500)
    (1.1094,0.9667)
    (1.1595,0.9667)
    (1.1867,0.9833)
    (1.25,1.0000)
  };
  \addlegendentry{D-NMDT}
  \addplot+[mark=none, thick, blue!40!gray, dashed] coordinates {
    (1.0000,0.2833)
    (1.0000,0.3333)
    (1.0001,0.3333)
    (1.0001,0.3833)
    (1.0003,0.3833)
    (1.0003,0.4000)
    (1.0004,0.4167)
    (1.0004,0.4500)
    (1.0005,0.4667)
    (1.0006,0.4833)
    (1.0008,0.4833)
    (1.0008,0.5000)
    (1.0009,0.5167)
    (1.0019,0.5167)
    (1.0020,0.5333)
    (1.0027,0.5333)
    (1.0028,0.5500)
    (1.0039,0.5500)
    (1.0043,0.5667)
    (1.0046,0.5667)
    (1.0052,0.5833)
    (1.0053,0.5833)
    (1.0055,0.6000)
    (1.0060,0.6167)
    (1.0088,0.6167)
    (1.0096,0.6333)
    (1.0101,0.6500)
    (1.0102,0.6667)
    (1.0109,0.6667)
    (1.0110,0.6833)
    (1.0110,0.7000)
    (1.0120,0.7000)
    (1.0129,0.7167)
    (1.0129,0.7333)
    (1.0145,0.7333)
    (1.0152,0.7500)
    (1.0181,0.7667)
    (1.0186,0.7667)
    (1.0198,0.7833)
    (1.0199,0.8000)
    (1.0203,0.8000)
    (1.0215,0.8167)
    (1.0222,0.8333)
    (1.0269,0.8333)
    (1.0283,0.8500)
    (1.0360,0.8500)
    (1.0433,0.8667)
    (1.0540,0.8667)
    (1.0544,0.8833)
    (1.0616,0.9000)
    (1.0628,0.9167)
    (1.0685,0.9333)
    (1.0707,0.9333)
    (1.0719,0.9500)
    (1.0804,0.9667)
    (1.1094,0.9667)
    (1.1595,0.9833)
    (1.1903,0.9833)
    (1.1943,1.0000)
    (1.25,1.0000)
  };
  \addlegendentry{NMDT}
  \addplot+[mark=none, thick, green!40!gray, densely dotted] coordinates {
    (1.0000,0.4333)
    (1.0000,0.6833)
    (1.0001,0.6833)
    (1.0001,0.7500)
    (1.0002,0.7667)
    (1.0003,0.7833)
    (1.0004,0.7833)
    (1.0004,0.8000)
    (1.0020,0.8000)
    (1.0026,0.8167)
    (1.0028,0.8167)
    (1.0030,0.8333)
    (1.0052,0.8333)
    (1.0053,0.8500)
    (1.0060,0.8500)
    (1.0062,0.8667)
    (1.0082,0.8667)
    (1.0084,0.8833)
    (1.0129,0.8833)
    (1.0136,0.9000)
    (1.0286,0.9000)
    (1.0326,0.9167)
    (1.0355,0.9167)
    (1.0360,0.9333)
    (1.0866,0.9333)
    (1.0899,0.9500)
    (1.0983,0.9667)
    (1.1867,0.9667)
    (1.1903,0.9833)
    (1.25,1.0000)
  };
  \addlegendentry{T-D-NMDT}
  \addplot+[mark=none, thick, blue!40!gray, densely dotted] coordinates {
    (1.0000,0.2333)
    (1.0000,0.5000)
    (1.0001,0.5167)
    (1.0001,0.5333)
    (1.0002,0.5333)
    (1.0002,0.5500)
    (1.0003,0.5500)
    (1.0003,0.5667)
    (1.0004,0.5667)
    (1.0004,0.5833)
    (1.0008,0.5833)
    (1.0008,0.6000)
    (1.0014,0.6000)
    (1.0018,0.6167)
    (1.0019,0.6333)
    (1.0020,0.6500)
    (1.0036,0.6500)
    (1.0037,0.6667)
    (1.0038,0.6833)
    (1.0044,0.6833)
    (1.0046,0.7000)
    (1.0062,0.7000)
    (1.0064,0.7167)
    (1.0079,0.7333)
    (1.0082,0.7500)
    (1.0084,0.7500)
    (1.0088,0.7667)
    (1.0102,0.7667)
    (1.0109,0.7833)
    (1.0110,0.7833)
    (1.0111,0.8000)
    (1.0199,0.8000)
    (1.0203,0.8167)
    (1.0259,0.8167)
    (1.0269,0.8333)
    (1.0283,0.8333)
    (1.0286,0.8500)
    (1.0354,0.8500)
    (1.0355,0.8667)
    (1.0433,0.8667)
    (1.0437,0.8833)
    (1.0540,0.9000)
    (1.0540,0.9167)
    (1.0696,0.9167)
    (1.0707,0.9333)
    (1.0804,0.9333)
    (1.0830,0.9500)
    (1.0866,0.9667)
    (1.1943,0.9667)
    (1.2384,0.9833)
    (1.25,1.0000)
  };
  \addlegendentry{T-NMDT}
  \draw node[right,draw,align=left] {$L=6$\\};
  \end{axis}
\end{tikzpicture}
            \end{center}
        \end{minipage}
    \end{center}
    \hfill
    \caption{Performance profiles to dual bounds of NMDT-based methods on all instances.}\label{dnmdt_vs_nmdt_all}
\end{figure}
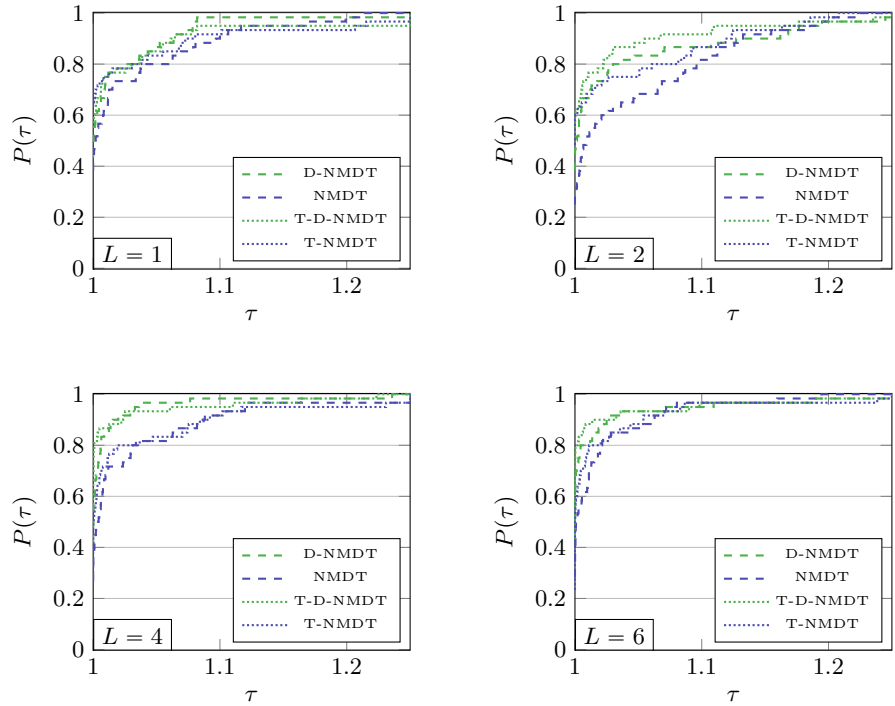

To gain a deeper insight into the benefits of \DNMDT and the tightening of \NMDT-based relaxations, we divide the benchmark set into sparse and dense instances.
For sparse instances, the advantage of the new methods is rather small; see \cref{dnmdt_vs_nmdt_sparse}.
Here, T-\DNMDT provides marginally better bounds than the other methods in case of $L=4$ and $L=6$.
For $L=1$ and $L=2$, however, T-\NMDT dominates all other approaches.
Moreover, the tightened versions outperform their counterparts for all depths $L$.

\begin{figure}
    \begin{center}
        \begin{minipage}{0.475\textwidth}
            \begin{center}
                 \begin{tikzpicture}
  \begin{axis}[const plot,
  cycle list={
  {blue!40!gray,solid},
  {red,dashed},
  {black,dotted},
  {brown,dashdotted}},
    xmin=1, xmax=1.17,
    ymin=-0.003, ymax=1.003,
    ymajorgrids,
    ytick={0,0.2,0.4,0.6,0.8,1.0},
    xlabel={$\tau$},
    ylabel={$P(\tau)$},
,
    legend pos={south east},
    legend style={font=\tiny},
    width=\textwidth
    ]
  \addplot+[mark=none, thick, green!40!gray, dashed] coordinates {
    (1.0000,0.3667)
    (1.0000,0.4333)
    (1.0006,0.4667)
    (1.0007,0.4667)
    (1.0008,0.5000)
    (1.0014,0.5000)
    (1.0017,0.5333)
    (1.0022,0.5333)
    (1.0032,0.5667)
    (1.0051,0.5667)
    (1.0061,0.6000)
    (1.0083,0.6000)
    (1.0089,0.6333)
    (1.0093,0.6667)
    (1.0096,0.7000)
    (1.0105,0.7333)
    (1.0116,0.7333)
    (1.0119,0.7667)
    (1.0178,0.7667)
    (1.0254,0.8000)
    (1.0303,0.8000)
    (1.0335,0.8333)
    (1.0336,0.8333)
    (1.0364,0.8667)
    (1.0592,0.8667)
    (1.0634,0.9000)
    (1.0645,0.9333)
    (1.0692,0.9333)
    (1.0784,0.9667)
    (1.0820,0.9667)
    (1.0820,1.0000)
    (1.1724,1.0000)
  };
  \addlegendentry{D-NMDT}
  \addplot+[mark=none, thick, blue!40!gray, dashed] coordinates {
    (1.0000,0.2000)
    (1.0000,0.3000)
    (1.0006,0.3000)
    (1.0007,0.3333)
    (1.0008,0.3333)
    (1.0010,0.3667)
    (1.0014,0.4000)
    (1.0032,0.4000)
    (1.0034,0.4333)
    (1.0035,0.4667)
    (1.0040,0.5000)
    (1.0072,0.5000)
    (1.0078,0.5333)
    (1.0083,0.5667)
    (1.0105,0.5667)
    (1.0108,0.6000)
    (1.0116,0.6333)
    (1.0119,0.6333)
    (1.0122,0.6667)
    (1.0152,0.7000)
    (1.0178,0.7333)
    (1.0335,0.7333)
    (1.0336,0.7667)
    (1.0364,0.7667)
    (1.0364,0.8000)
    (1.0377,0.8333)
    (1.0550,0.8333)
    (1.0592,0.8667)
    (1.0645,0.8667)
    (1.0692,0.9000)
    (1.0820,0.9000)
    (1.0968,0.9333)
    (1.1000,0.9667)
    (1.1166,1.0000)
    (1.1724,1.0000)
  };
  \addlegendentry{NMDT}
  \addplot+[mark=none, thick, green!40!gray, densely dotted] coordinates {
    (1.0000,0.5333)
    (1.0000,0.5667)
    (1.0006,0.5667)
    (1.0006,0.6000)
    (1.0019,0.6000)
    (1.0022,0.6333)
    (1.0061,0.6333)
    (1.0061,0.6667)
    (1.0072,0.7000)
    (1.0078,0.7000)
    (1.0078,0.7333)
    (1.0116,0.7333)
    (1.0119,0.7667)
    (1.0254,0.7667)
    (1.0254,0.8000)
    (1.0303,0.8333)
    (1.0336,0.8333)
    (1.0364,0.8667)
    (1.0377,0.8667)
    (1.0550,0.9000)
    (1.0645,0.9000)
    (1.0645,0.9333)
    (1.0692,0.9333)
    (1.0784,0.9667)
    (1.0820,1.0000)
    (1.1724,1.0000)
  };
  \addlegendentry{T-D-NMDT}
  \addplot+[mark=none, thick, blue!40!gray, densely dotted] coordinates {
    (1.0000,0.7000)
    (1.0000,0.7333)
    (1.0007,0.7333)
    (1.0007,0.7667)
    (1.0010,0.7667)
    (1.0014,0.8000)
    (1.0017,0.8000)
    (1.0019,0.8333)
    (1.0034,0.8333)
    (1.0035,0.8667)
    (1.0040,0.8667)
    (1.0051,0.9000)
    (1.0078,0.9000)
    (1.0083,0.9333)
    (1.0122,0.9333)
    (1.0152,0.9667)
    (1.0645,0.9667)
    (1.0692,1.0000)
    (1.1724,1.0000)
  };
  \addlegendentry{T-NMDT}
  \draw node[right,draw,align=left] {$L=1$\\};
  \end{axis}
\end{tikzpicture}
            \end{center}
        \end{minipage}
        \hfill
        \begin{minipage}{0.475\textwidth}
            \begin{center}
                 \begin{tikzpicture}
  \begin{axis}[const plot,
  cycle list={
  {blue!40!gray,solid},
  {red,dashed},
  {black,dotted},
  {brown,dashdotted}},
    xmin=1, xmax=1.24,
    ymin=-0.003, ymax=1.003,
    ymajorgrids,
    ytick={0,0.2,0.4,0.6,0.8,1.0},
    xlabel={$\tau$},
    ylabel={$P(\tau)$},
,
    legend pos={south east},
    legend style={font=\tiny},
    width=\textwidth
    ]
  \addplot+[mark=none, thick, green!40!gray, dashed] coordinates {
    (1.0000,0.2667)
    (1.0000,0.3000)
    (1.0001,0.3667)
    (1.0003,0.3667)
    (1.0010,0.4000)
    (1.0016,0.4000)
    (1.0016,0.4333)
    (1.0023,0.4333)
    (1.0028,0.4667)
    (1.0032,0.4667)
    (1.0032,0.5000)
    (1.0047,0.5000)
    (1.0048,0.5333)
    (1.0051,0.5333)
    (1.0051,0.5667)
    (1.0054,0.5667)
    (1.0056,0.6000)
    (1.0097,0.6000)
    (1.0110,0.6333)
    (1.0172,0.6333)
    (1.0175,0.6667)
    (1.0209,0.6667)
    (1.0223,0.7000)
    (1.0254,0.7000)
    (1.0255,0.7333)
    (1.0266,0.7333)
    (1.0300,0.7667)
    (1.0364,0.8000)
    (1.0469,0.8000)
    (1.0470,0.8333)
    (1.0676,0.8333)
    (1.0676,0.8667)
    (1.0706,0.9000)
    (1.0971,0.9000)
    (1.1072,0.9333)
    (1.1096,0.9333)
    (1.1271,0.9667)
    (1.1766,0.9667)
    (1.1766,1.0000)
    (1.2355,1.0000)
  };
  \addlegendentry{D-NMDT}
  \addplot+[mark=none, thick, blue!40!gray, dashed] coordinates {
    (1.0000,0.1333)
    (1.0000,0.1667)
    (1.0010,0.1667)
    (1.0012,0.2000)
    (1.0016,0.2333)
    (1.0019,0.2333)
    (1.0023,0.2667)
    (1.0032,0.2667)
    (1.0032,0.3000)
    (1.0035,0.3333)
    (1.0041,0.3667)
    (1.0046,0.4000)
    (1.0051,0.4000)
    (1.0054,0.4333)
    (1.0056,0.4333)
    (1.0060,0.4667)
    (1.0066,0.5000)
    (1.0088,0.5000)
    (1.0090,0.5333)
    (1.0097,0.5667)
    (1.0110,0.5667)
    (1.0138,0.6000)
    (1.0159,0.6333)
    (1.0172,0.6667)
    (1.0179,0.6667)
    (1.0209,0.7000)
    (1.0266,0.7000)
    (1.0266,0.7333)
    (1.0300,0.7667)
    (1.0364,0.7667)
    (1.0364,0.8000)
    (1.0461,0.8333)
    (1.0550,0.8333)
    (1.0652,0.8667)
    (1.0706,0.8667)
    (1.0971,0.9000)
    (1.1072,0.9333)
    (1.1271,0.9333)
    (1.1326,0.9667)
    (1.1548,1.0000)
    (1.2355,1.0000)
  };
  \addlegendentry{NMDT}
  \addplot+[mark=none, thick, green!40!gray, densely dotted] coordinates {
    (1.0000,0.3667)
    (1.0000,0.4000)
    (1.0001,0.4667)
    (1.0003,0.4667)
    (1.0003,0.5000)
    (1.0016,0.5000)
    (1.0019,0.5333)
    (1.0028,0.5333)
    (1.0032,0.5667)
    (1.0046,0.5667)
    (1.0047,0.6000)
    (1.0048,0.6000)
    (1.0051,0.6333)
    (1.0066,0.6333)
    (1.0088,0.6667)
    (1.0175,0.6667)
    (1.0179,0.7000)
    (1.0209,0.7000)
    (1.0223,0.7333)
    (1.0254,0.7667)
    (1.0266,0.7667)
    (1.0300,0.8000)
    (1.0461,0.8000)
    (1.0469,0.8333)
    (1.0470,0.8333)
    (1.0550,0.8667)
    (1.0652,0.8667)
    (1.0676,0.9000)
    (1.0971,0.9000)
    (1.1072,0.9333)
    (1.1096,0.9667)
    (1.1548,0.9667)
    (1.1766,1.0000)
    (1.2355,1.0000)
  };
  \addlegendentry{T-D-NMDT}
  \addplot+[mark=none, thick, blue!40!gray, densely dotted] coordinates {
    (1.0000,0.6333)
    (1.0000,0.7333)
    (1.0001,0.7333)
    (1.0003,0.7667)
    (1.0019,0.7667)
    (1.0023,0.8000)
    (1.0051,0.8000)
    (1.0054,0.8333)
    (1.0088,0.8333)
    (1.0090,0.8667)
    (1.0110,0.8667)
    (1.0138,0.9000)
    (1.0159,0.9333)
    (1.0179,0.9333)
    (1.0209,0.9667)
    (1.0255,0.9667)
    (1.0266,1.0000)
    (1.2355,1.0000)
  };
  \addlegendentry{T-NMDT}
  \draw node[right,draw,align=left] {$L=2$\\};
  \end{axis}
\end{tikzpicture}
            \end{center}
        \end{minipage}
    \end{center}
    \hfill
    \begin{center}
        \begin{minipage}{0.475\textwidth}
            \begin{center}
                 \begin{tikzpicture}
  \begin{axis}[const plot,
  cycle list={
  {blue!40!gray,solid},
  {red,dashed},
  {black,dotted},
  {brown,dashdotted}},
    xmin=1, xmax=1.14,
    ymin=-0.003, ymax=1.003,
    ymajorgrids,
    ytick={0,0.2,0.4,0.6,0.8,1.0},
    xlabel={$\tau$},
    ylabel={$P(\tau)$},
,
    legend pos={south east},
    legend style={font=\tiny},
    width=\textwidth
    ]
  \addplot+[mark=none, thick, green!40!gray, dashed] coordinates {
    (1.0000,0.2333)
    (1.0000,0.3333)
    (1.0001,0.3333)
    (1.0001,0.3667)
    (1.0002,0.4000)
    (1.0003,0.4000)
    (1.0003,0.4333)
    (1.0004,0.4333)
    (1.0004,0.5000)
    (1.0008,0.5000)
    (1.0009,0.5333)
    (1.0010,0.5333)
    (1.0011,0.5667)
    (1.0012,0.5667)
    (1.0013,0.6000)
    (1.0014,0.6000)
    (1.0014,0.6333)
    (1.0016,0.6667)
    (1.0025,0.6667)
    (1.0025,0.7000)
    (1.0034,0.7000)
    (1.0038,0.7333)
    (1.0040,0.7667)
    (1.0043,0.7667)
    (1.0043,0.8000)
    (1.0044,0.8333)
    (1.0055,0.8333)
    (1.0059,0.8667)
    (1.0061,0.8667)
    (1.0064,0.9000)
    (1.0064,0.9333)
    (1.0119,0.9333)
    (1.0124,0.9667)
    (1.0162,0.9667)
    (1.0198,1.0000)
    (1.1414,1.0000)
  };
  \addlegendentry{D-NMDT}
  \addplot+[mark=none, thick, blue!40!gray, dashed] coordinates {
    (1.0000,0.1667)
    (1.0000,0.2667)
    (1.0001,0.2667)
    (1.0001,0.3333)
    (1.0003,0.3333)
    (1.0004,0.3667)
    (1.0009,0.3667)
    (1.0010,0.4000)
    (1.0013,0.4000)
    (1.0014,0.4333)
    (1.0018,0.4333)
    (1.0018,0.4667)
    (1.0021,0.4667)
    (1.0021,0.5000)
    (1.0023,0.5333)
    (1.0025,0.5667)
    (1.0040,0.5667)
    (1.0043,0.6000)
    (1.0049,0.6000)
    (1.0050,0.6333)
    (1.0059,0.6333)
    (1.0060,0.6667)
    (1.0061,0.7000)
    (1.0064,0.7000)
    (1.0073,0.7333)
    (1.0077,0.7667)
    (1.0079,0.8000)
    (1.0086,0.8333)
    (1.0112,0.8667)
    (1.0114,0.8667)
    (1.0119,0.9000)
    (1.0198,0.9000)
    (1.0256,0.9333)
    (1.0308,0.9667)
    (1.0842,0.9667)
    (1.0871,1.0000)
    (1.1414,1.0000)
  };
  \addlegendentry{NMDT}
  \addplot+[mark=none, thick, green!40!gray, densely dotted] coordinates {
    (1.0000,0.5333)
    (1.0000,0.6333)
    (1.0001,0.6667)
    (1.0001,0.7667)
    (1.0002,0.7667)
    (1.0002,0.8000)
    (1.0003,0.8333)
    (1.0003,0.8667)
    (1.0004,0.8667)
    (1.0004,0.9333)
    (1.0008,0.9333)
    (1.0009,0.9667)
    (1.0308,0.9667)
    (1.0605,1.0000)
    (1.1414,1.0000)
  };
  \addlegendentry{T-D-NMDT}
  \addplot+[mark=none, thick, blue!40!gray, densely dotted] coordinates {
    (1.0000,0.4333)
    (1.0000,0.4667)
    (1.0001,0.4667)
    (1.0001,0.5333)
    (1.0004,0.5333)
    (1.0004,0.5667)
    (1.0008,0.6000)
    (1.0011,0.6000)
    (1.0012,0.6333)
    (1.0016,0.6333)
    (1.0018,0.6667)
    (1.0021,0.7000)
    (1.0025,0.7000)
    (1.0026,0.7333)
    (1.0029,0.7667)
    (1.0034,0.8000)
    (1.0044,0.8000)
    (1.0049,0.8333)
    (1.0050,0.8333)
    (1.0055,0.8667)
    (1.0073,0.8667)
    (1.0077,0.9000)
    (1.0112,0.9000)
    (1.0114,0.9333)
    (1.0124,0.9333)
    (1.0162,0.9667)
    (1.0605,0.9667)
    (1.0842,1.0000)
    (1.1414,1.0000)
  };
  \addlegendentry{T-NMDT}
  \draw node[right,draw,align=left] {$L=4$\\};
  \end{axis}
\end{tikzpicture}
            \end{center}
        \end{minipage}
        \hfill
        \begin{minipage}{0.475\textwidth}
            \begin{center}
                 \begin{tikzpicture}
  \begin{axis}[const plot,
  cycle list={
  {blue!40!gray,solid},
  {red,dashed},
  {black,dotted},
  {brown,dashdotted}},
    xmin=1, xmax=1.14,
    ymin=-0.003, ymax=1.003,
    ymajorgrids,
    ytick={0,0.2,0.4,0.6,0.8,1.0},
    xlabel={$\tau$},
    ylabel={$P(\tau)$},
,
    legend pos={south east},
    legend style={font=\tiny},
    width=\textwidth
    ]
  \addplot+[mark=none, thick, green!40!gray, dashed] coordinates {
    (1.0000,0.3333)
    (1.0000,0.4667)
    (1.0001,0.4667)
    (1.0001,0.6333)
    (1.0003,0.6333)
    (1.0003,0.7333)
    (1.0004,0.7333)
    (1.0004,0.7667)
    (1.0009,0.7667)
    (1.0014,0.8000)
    (1.0018,0.8000)
    (1.0019,0.8333)
    (1.0028,0.8333)
    (1.0036,0.8667)
    (1.0111,0.8667)
    (1.0120,0.9000)
    (1.0129,0.9000)
    (1.0145,0.9333)
    (1.0152,0.9333)
    (1.0186,0.9667)
    (1.0259,1.0000)
    (1.1371,1.0000)
  };
  \addlegendentry{D-NMDT}
  \addplot+[mark=none, thick, blue!40!gray, dashed] coordinates {
    (1.0000,0.1667)
    (1.0000,0.2333)
    (1.0001,0.2333)
    (1.0001,0.3333)
    (1.0003,0.3333)
    (1.0003,0.3667)
    (1.0004,0.4000)
    (1.0004,0.4667)
    (1.0006,0.5333)
    (1.0008,0.5667)
    (1.0009,0.6000)
    (1.0019,0.6000)
    (1.0020,0.6333)
    (1.0026,0.6333)
    (1.0028,0.6667)
    (1.0037,0.6667)
    (1.0043,0.7000)
    (1.0046,0.7000)
    (1.0052,0.7333)
    (1.0053,0.7333)
    (1.0055,0.7667)
    (1.0060,0.8000)
    (1.0088,0.8000)
    (1.0101,0.8333)
    (1.0109,0.8333)
    (1.0110,0.8667)
    (1.0120,0.8667)
    (1.0129,0.9000)
    (1.0145,0.9000)
    (1.0152,0.9333)
    (1.0259,0.9333)
    (1.0628,0.9667)
    (1.0804,1.0000)
    (1.1371,1.0000)
  };
  \addlegendentry{NMDT}
  \addplot+[mark=none, thick, green!40!gray, densely dotted] coordinates {
    (1.0000,0.5333)
    (1.0000,0.7333)
    (1.0001,0.7333)
    (1.0001,0.8333)
    (1.0003,0.8667)
    (1.0004,0.8667)
    (1.0004,0.9000)
    (1.0020,0.9000)
    (1.0026,0.9333)
    (1.0052,0.9333)
    (1.0053,0.9667)
    (1.0082,0.9667)
    (1.0084,1.0000)
    (1.1371,1.0000)
  };
  \addlegendentry{T-D-NMDT}
  \addplot+[mark=none, thick, blue!40!gray, densely dotted] coordinates {
    (1.0000,0.3333)
    (1.0000,0.5333)
    (1.0001,0.5667)
    (1.0001,0.6000)
    (1.0003,0.6000)
    (1.0003,0.6333)
    (1.0004,0.6333)
    (1.0004,0.6667)
    (1.0008,0.6667)
    (1.0008,0.7000)
    (1.0014,0.7000)
    (1.0018,0.7333)
    (1.0019,0.7333)
    (1.0020,0.7667)
    (1.0036,0.7667)
    (1.0037,0.8000)
    (1.0043,0.8000)
    (1.0046,0.8333)
    (1.0060,0.8333)
    (1.0082,0.8667)
    (1.0084,0.8667)
    (1.0088,0.9000)
    (1.0101,0.9000)
    (1.0109,0.9333)
    (1.0110,0.9333)
    (1.0111,0.9667)
    (1.0804,0.9667)
    (1.0830,1.0000)
    (1.1371,1.0000)
  };
  \addlegendentry{T-NMDT}
  \draw node[right,draw,align=left] {$L=6$\\};
  \end{axis}
\end{tikzpicture}
            \end{center}
        \end{minipage}
    \end{center}
    \hfill
    \caption{Performance profiles to dual bounds of \NMDT-based methods on sparse instances. }\label{dnmdt_vs_nmdt_sparse}
  \vspace*{\floatsep}%
    \begin{center}
        \begin{minipage}{0.475\textwidth}
            \begin{center}
                 \begin{tikzpicture}
  \begin{axis}[const plot,
  cycle list={
  {blue!40!gray,solid},
  {red,dashed},
  {black,dotted},
  {brown,dashdotted}},
    xmin=1, xmax=1.25,
    ymin=-0.003, ymax=1.003,
    ymajorgrids,
    ytick={0,0.2,0.4,0.6,0.8,1.0},
    xlabel={$\tau$},
    ylabel={$P(\tau)$},
,
    legend pos={south east},
    legend style={font=\tiny},
    width=\textwidth
    ]
  \addplot+[mark=none, thick, green!40!gray, dashed] coordinates {
    (1.0000,0.5000)
    (1.0000,0.5667)
    (1.0003,0.5667)
    (1.0005,0.6000)
    (1.0012,0.6000)
    (1.0015,0.6333)
    (1.0020,0.6333)
    (1.0022,0.6667)
    (1.0030,0.6667)
    (1.0046,0.7000)
    (1.0059,0.7000)
    (1.0060,0.7333)
    (1.0067,0.7333)
    (1.0086,0.7667)
    (1.0126,0.7667)
    (1.0185,0.8000)
    (1.0428,0.8000)
    (1.0430,0.8333)
    (1.0438,0.8333)
    (1.0511,0.8667)
    (1.0525,0.9000)
    (1.0676,0.9000)
    (1.0708,0.9333)
    (1.0789,0.9333)
    (1.0795,0.9667)
    (1.2154,0.9667)
    (1.25,1.0000)
  };
  \addlegendentry{D-NMDT}
  \addplot+[mark=none, thick, blue!40!gray, dashed] coordinates {
    (1.0000,0.5667)
    (1.0001,0.5667)
    (1.0003,0.6000)
    (1.0015,0.6000)
    (1.0020,0.6333)
    (1.0062,0.6333)
    (1.0067,0.6667)
    (1.0086,0.6667)
    (1.0117,0.7000)
    (1.0125,0.7333)
    (1.0306,0.7333)
    (1.0373,0.7667)
    (1.0610,0.7667)
    (1.0627,0.8000)
    (1.0726,0.8000)
    (1.0789,0.8333)
    (1.0805,0.8333)
    (1.0807,0.8667)
    (1.1012,0.8667)
    (1.1065,0.9000)
    (1.1693,0.9333)
    (1.1992,0.9667)
    (1.2065,0.9667)
    (1.2110,1.0000)
    (1.25,1.0000)
  };
  \addlegendentry{NMDT}
  \addplot+[mark=none, thick, green!40!gray, densely dotted] coordinates {
    (1.0000,0.5667)
    (1.0001,0.6000)
    (1.0005,0.6000)
    (1.0012,0.6333)
    (1.0020,0.6333)
    (1.0022,0.6667)
    (1.0030,0.7000)
    (1.0046,0.7000)
    (1.0059,0.7333)
    (1.0060,0.7333)
    (1.0062,0.7667)
    (1.0380,0.7667)
    (1.0387,0.8000)
    (1.0430,0.8000)
    (1.0438,0.8333)
    (1.0525,0.8333)
    (1.0536,0.8667)
    (1.0553,0.8667)
    (1.0610,0.9000)
    (1.2154,0.9000)
    (1.25,1.0000)
  };
  \addlegendentry{T-D-NMDT}
  \addplot+[mark=none, thick, blue!40!gray, densely dotted] coordinates {
    (1.0000,0.5667)
    (1.0125,0.5667)
    (1.0126,0.6000)
    (1.0185,0.6000)
    (1.0306,0.6333)
    (1.0373,0.6333)
    (1.0380,0.6667)
    (1.0387,0.6667)
    (1.0428,0.7000)
    (1.0536,0.7000)
    (1.0553,0.7333)
    (1.0627,0.7333)
    (1.0676,0.7667)
    (1.0708,0.7667)
    (1.0726,0.8000)
    (1.0795,0.8000)
    (1.0805,0.8333)
    (1.0807,0.8333)
    (1.1012,0.8667)
    (1.1992,0.8667)
    (1.2065,0.9000)
    (1.2110,0.9000)
    (1.2154,0.9333)
    (1.25,1.0000)
  };
  \addlegendentry{T-NMDT}
  \draw node[right,draw,align=left] {$L=1$\\};
  \end{axis}
\end{tikzpicture}
            \end{center}
        \end{minipage}
        \hfill
        \begin{minipage}{0.475\textwidth}
            \begin{center}
                 \begin{tikzpicture}
  \begin{axis}[const plot,
  cycle list={
  {blue!40!gray,solid},
  {red,dashed},
  {black,dotted},
  {brown,dashdotted}},
    xmin=1, xmax=1.25,
    ymin=-0.003, ymax=1.003,
    ymajorgrids,
    ytick={0,0.2,0.4,0.6,0.8,1.0},
    xlabel={$\tau$},
    ylabel={$P(\tau)$},
,
    legend pos={south east},
    legend style={font=\tiny},
    width=\textwidth
    ]
  \addplot+[mark=none, thick, green!40!gray, dashed] coordinates {
    (1.0000,0.5000)
    (1.0000,0.5667)
    (1.0005,0.6000)
    (1.0019,0.6000)
    (1.0029,0.6333)
    (1.0032,0.6667)
    (1.0036,0.7000)
    (1.0058,0.7000)
    (1.0063,0.7333)
    (1.0107,0.7333)
    (1.0107,0.7667)
    (1.0114,0.7667)
    (1.0130,0.8000)
    (1.0226,0.8000)
    (1.0263,0.8333)
    (1.1285,0.8333)
    (1.1620,0.8667)
    (1.1633,0.8667)
    (1.1682,0.9000)
    (1.1858,0.9000)
    (1.1918,0.9333)
    (1.2351,0.9333)
    (1.2448,0.9667)
    (1.5641,1.0000)
  };
  \addlegendentry{D-NMDT}
  \addplot+[mark=none, thick, blue!40!gray, dashed] coordinates {
    (1.0000,0.3667)
    (1.0000,0.4333)
    (1.0036,0.4333)
    (1.0044,0.4667)
    (1.0107,0.4667)
    (1.0114,0.5000)
    (1.0295,0.5000)
    (1.0471,0.5333)
    (1.0606,0.5333)
    (1.0664,0.5667)
    (1.0686,0.6000)
    (1.0811,0.6333)
    (1.0815,0.6333)
    (1.0854,0.6667)
    (1.0861,0.6667)
    (1.0870,0.7000)
    (1.0934,0.7000)
    (1.0959,0.7333)
    (1.1123,0.7667)
    (1.1125,0.8000)
    (1.1248,0.8000)
    (1.1257,0.8333)
    (1.1285,0.8667)
    (1.1813,0.8667)
    (1.1853,0.9000)
    (1.1918,0.9000)
    (1.1949,0.9333)
    (1.2014,0.9667)
    (1.2079,0.9667)
    (1.2229,1.0000)
    (1.5641,1.0000)
  };
  \addlegendentry{NMDT}
  \addplot+[mark=none, thick, green!40!gray, densely dotted] coordinates {
    (1.0000,0.6667)
    (1.0000,0.7333)
    (1.0032,0.7333)
    (1.0032,0.7667)
    (1.0044,0.7667)
    (1.0058,0.8000)
    (1.0063,0.8000)
    (1.0066,0.8333)
    (1.0088,0.8333)
    (1.0107,0.8667)
    (1.0130,0.8667)
    (1.0226,0.9000)
    (1.0263,0.9000)
    (1.0295,0.9333)
    (1.2229,0.9333)
    (1.2351,0.9667)
    (1.2448,0.9667)
    (1.5641,1.0000)
  };
  \addlegendentry{T-D-NMDT}
  \addplot+[mark=none, thick, blue!40!gray, densely dotted] coordinates {
    (1.0000,0.3333)
    (1.0000,0.4333)
    (1.0005,0.4333)
    (1.0019,0.4667)
    (1.0066,0.4667)
    (1.0088,0.5000)
    (1.0471,0.5000)
    (1.0508,0.5333)
    (1.0508,0.5667)
    (1.0606,0.6000)
    (1.0811,0.6000)
    (1.0815,0.6333)
    (1.0854,0.6333)
    (1.0861,0.6667)
    (1.0870,0.6667)
    (1.0915,0.7000)
    (1.0934,0.7333)
    (1.1125,0.7333)
    (1.1131,0.7667)
    (1.1187,0.8000)
    (1.1237,0.8333)
    (1.1248,0.8667)
    (1.1620,0.8667)
    (1.1633,0.9000)
    (1.1682,0.9000)
    (1.1813,0.9333)
    (1.1853,0.9333)
    (1.1858,0.9667)
    (1.2014,0.9667)
    (1.2079,1.0000)
    (1.5641,1.0000)
  };
  \addlegendentry{T-NMDT}
  \draw node[right,draw,align=left] {$L=2$\\};
  \end{axis}
\end{tikzpicture}
            \end{center}
        \end{minipage}
    \end{center}
    \hfill
    \begin{center}
        \begin{minipage}{0.475\textwidth}
            \begin{center}
                 \begin{tikzpicture}
  \begin{axis}[const plot,
  cycle list={
  {blue!40!gray,solid},
  {red,dashed},
  {black,dotted},
  {brown,dashdotted}},
    xmin=1, xmax=1.25,
    ymin=-0.003, ymax=1.003,
    ymajorgrids,
    ytick={0,0.2,0.4,0.6,0.8,1.0},
    xlabel={$\tau$},
    ylabel={$P(\tau)$},
,
    legend pos={south east},
    legend style={font=\tiny},
    width=\textwidth
    ]
  \addplot+[mark=none, thick, green!40!gray, dashed] coordinates {
    (1.0000,0.6000)
    (1.0001,0.6333)
    (1.0009,0.6333)
    (1.0017,0.6667)
    (1.0035,0.6667)
    (1.0036,0.7000)
    (1.0050,0.7000)
    (1.0058,0.7333)
    (1.0066,0.7333)
    (1.0087,0.7667)
    (1.0127,0.7667)
    (1.0137,0.8000)
    (1.0155,0.8333)
    (1.0251,0.8333)
    (1.0292,0.8667)
    (1.0296,0.8667)
    (1.0329,0.9000)
    (1.0337,0.9000)
    (1.0362,0.9333)
    (1.0754,0.9333)
    (1.0765,0.9667)
    (1.2310,0.9667)
    (1.2358,1.0000)
    (1.3760,1.0000)
  };
  \addlegendentry{D-NMDT}
  \addplot+[mark=none, thick, blue!40!gray, dashed] coordinates {
    (1.0000,0.3667)
    (1.0000,0.4000)
    (1.0001,0.4000)
    (1.0009,0.4333)
    (1.0036,0.4333)
    (1.0048,0.4667)
    (1.0058,0.4667)
    (1.0066,0.5000)
    (1.0087,0.5000)
    (1.0097,0.5333)
    (1.0222,0.5333)
    (1.0235,0.5667)
    (1.0292,0.5667)
    (1.0295,0.6000)
    (1.0296,0.6333)
    (1.0329,0.6333)
    (1.0337,0.6667)
    (1.0470,0.6667)
    (1.0574,0.7000)
    (1.0627,0.7333)
    (1.0645,0.7667)
    (1.0765,0.7667)
    (1.0819,0.8000)
    (1.0881,0.8333)
    (1.1020,0.8333)
    (1.1043,0.8667)
    (1.1169,0.8667)
    (1.1197,0.9000)
    (1.1216,0.9333)
    (1.3760,1.0000)
  };
  \addlegendentry{NMDT}
  \addplot+[mark=none, thick, green!40!gray, densely dotted] coordinates {
    (1.0000,0.6000)
    (1.0000,0.6333)
    (1.0009,0.6333)
    (1.0009,0.6667)
    (1.0017,0.6667)
    (1.0030,0.7000)
    (1.0031,0.7000)
    (1.0035,0.7333)
    (1.0048,0.7333)
    (1.0050,0.7667)
    (1.0126,0.7667)
    (1.0127,0.8000)
    (1.0190,0.8000)
    (1.0222,0.8333)
    (1.0235,0.8333)
    (1.0237,0.8667)
    (1.0251,0.9000)
    (1.1043,0.9000)
    (1.1102,0.9333)
    (1.1216,0.9333)
    (1.1638,0.9667)
    (1.2234,1.0000)
    (1.3760,1.0000)
  };
  \addlegendentry{T-D-NMDT}
  \addplot+[mark=none, thick, blue!40!gray, densely dotted] coordinates {
    (1.0000,0.1667)
    (1.0000,0.4667)
    (1.0030,0.4667)
    (1.0031,0.5000)
    (1.0048,0.5000)
    (1.0048,0.5333)
    (1.0097,0.5333)
    (1.0113,0.5667)
    (1.0126,0.6000)
    (1.0155,0.6000)
    (1.0190,0.6333)
    (1.0362,0.6333)
    (1.0367,0.6667)
    (1.0470,0.7000)
    (1.0645,0.7000)
    (1.0703,0.7333)
    (1.0743,0.7667)
    (1.0754,0.8000)
    (1.0881,0.8000)
    (1.0921,0.8333)
    (1.1020,0.8667)
    (1.1102,0.8667)
    (1.1169,0.9000)
    (1.2234,0.9000)
    (1.2310,0.9333)
    (1.2358,0.9333)
    (1.3760,1.0000)
  };
  \addlegendentry{T-NMDT}
  \draw node[right,draw,align=left] {$L=4$\\};
  \end{axis}
\end{tikzpicture}
            \end{center}
        \end{minipage}
        \hfill
        \begin{minipage}{0.475\textwidth}
            \begin{center}
                 \begin{tikzpicture}
  \begin{axis}[const plot,
  cycle list={
  {blue!40!gray,solid},
  {red,dashed},
  {black,dotted},
  {brown,dashdotted}},
    xmin=1, xmax=1.25,
    ymin=-0.003, ymax=1.003,
    ymajorgrids,
    ytick={0,0.2,0.4,0.6,0.8,1.0},
    xlabel={$\tau$},
    ylabel={$P(\tau)$},
,
    legend pos={south east},
    legend style={font=\tiny},
    width=\textwidth
    ]
  \addplot+[mark=none, thick, green!40!gray, dashed] coordinates {
    (1.0000,0.5333)
    (1.0000,0.5667)
    (1.0002,0.5667)
    (1.0003,0.6000)
    (1.0008,0.6333)
    (1.0019,0.6333)
    (1.0027,0.6667)
    (1.0038,0.6667)
    (1.0039,0.7000)
    (1.0044,0.7333)
    (1.0136,0.7333)
    (1.0136,0.7667)
    (1.0222,0.7667)
    (1.0229,0.8000)
    (1.0250,0.8333)
    (1.0326,0.8333)
    (1.0354,0.8667)
    (1.0685,0.8667)
    (1.0696,0.9000)
    (1.0983,0.9000)
    (1.1094,0.9333)
    (1.1595,0.9333)
    (1.1867,0.9667)
    (1.6758,1.0000)
  };
  \addlegendentry{D-NMDT}
  \addplot+[mark=none, thick, blue!40!gray, dashed] coordinates {
    (1.0000,0.4000)
    (1.0000,0.4333)
    (1.0079,0.4333)
    (1.0096,0.4667)
    (1.0102,0.5000)
    (1.0110,0.5333)
    (1.0129,0.5667)
    (1.0136,0.5667)
    (1.0181,0.6000)
    (1.0198,0.6333)
    (1.0199,0.6667)
    (1.0203,0.6667)
    (1.0215,0.7000)
    (1.0222,0.7333)
    (1.0269,0.7333)
    (1.0283,0.7667)
    (1.0360,0.7667)
    (1.0433,0.8000)
    (1.0540,0.8000)
    (1.0544,0.8333)
    (1.0616,0.8667)
    (1.0685,0.9000)
    (1.0707,0.9000)
    (1.0719,0.9333)
    (1.1094,0.9333)
    (1.1595,0.9667)
    (1.1903,0.9667)
    (1.1943,1.0000)
    (1.6758,1.0000)
  };
  \addlegendentry{NMDT}
  \addplot+[mark=none, thick, green!40!gray, densely dotted] coordinates {
    (1.0000,0.3333)
    (1.0000,0.6333)
    (1.0001,0.6667)
    (1.0002,0.7000)
    (1.0027,0.7000)
    (1.0030,0.7333)
    (1.0044,0.7333)
    (1.0062,0.7667)
    (1.0129,0.7667)
    (1.0136,0.8000)
    (1.0286,0.8000)
    (1.0326,0.8333)
    (1.0355,0.8333)
    (1.0360,0.8667)
    (1.0866,0.8667)
    (1.0899,0.9000)
    (1.0983,0.9333)
    (1.1867,0.9333)
    (1.1903,0.9667)
    (1.6758,1.0000)
  };
  \addlegendentry{T-D-NMDT}
  \addplot+[mark=none, thick, blue!40!gray, densely dotted] coordinates {
    (1.0000,0.1333)
    (1.0000,0.4667)
    (1.0002,0.4667)
    (1.0002,0.5000)
    (1.0008,0.5000)
    (1.0019,0.5333)
    (1.0030,0.5333)
    (1.0038,0.5667)
    (1.0062,0.5667)
    (1.0064,0.6000)
    (1.0079,0.6333)
    (1.0199,0.6333)
    (1.0203,0.6667)
    (1.0250,0.6667)
    (1.0269,0.7000)
    (1.0283,0.7000)
    (1.0286,0.7333)
    (1.0354,0.7333)
    (1.0355,0.7667)
    (1.0433,0.7667)
    (1.0437,0.8000)
    (1.0540,0.8333)
    (1.0540,0.8667)
    (1.0696,0.8667)
    (1.0707,0.9000)
    (1.0719,0.9000)
    (1.0866,0.9333)
    (1.1943,0.9333)
    (1.2384,0.9667)
    (1.6758,1.0000)
  };
  \addlegendentry{T-NMDT}
  \draw node[right,draw,align=left] {$L=6$\\};
  \end{axis}
\end{tikzpicture}
            \end{center}
        \end{minipage}
    \end{center}
    \hfill
    \caption{Performance profiles to dual bounds of NMDT-based methods on dense instances. }\label{dnmdt_vs_nmdt_dense}
\end{figure}

For dense instances, \DNMDT and T-\DNMDT are clearly superior to \NMDT and T-\NMDT; see \cref{dnmdt_vs_nmdt_dense}.
Regardless of the relaxation depth, the new methods yield the tightest dual bounds, with T-\DNMDT being superior to \DNMDT only in case of $L=2$, where the tightened version T-\DNMDT is able to find the best dual bound for roughly 10\% more instances than \DNMDT.
Tightening the \NMDT method does not deliver better bounds, in fact, T-\NMDT is surpassed by \NMDT for $L=1$.

Regarding the run times of the various \NMDT-based approaches, \cref{table_sgm_nmdt_all} shows significantly lower run times for \DNMDT and T-\DNMDT.
Again, T-\DNMDT is slightly ahead of \DNMDT.
\begin{table}
    \caption{Shifted geometric mean for run times on all 60 instances in NMDT-based MIP relaxations.}
    \label{table_sgm_nmdt_all}
    \centering
    \begin{tabular} {l r  r  r  r}
        \toprule
        {\qquad} & {\quad \NMDT} & {\quad T-\NMDT} & {\quad \DNMDT} & {\quad T-\DNMDT} \\  
        \midrule
         L1 & 82.09 & 83.42 & 68.91 & \textbf{51.13} \\
         L2 & 234.73 & 231.66 & 87.3 & \textbf{78.75} \\
         L4 & 450.63 & 395.93 & 196.4 & \textbf{192.97} \\
         L6 & 851.91 & 713.73 & 443.49 & \textbf{429.23} \\
        \bottomrule
    \end{tabular}
\end{table}

In \cref{table_feasible_nmdt_all}, we can see that the QP heuristic (IPOPT) we mentioned at the beginning of this section delivers high-quality feasible solutions for the original (MIQC-)QP instances.
With increasing $L$ values, IPOPT is able to find more feasible solutions with all \NMDT-based methods quite similarly.
For $L=6$, T-\DNMDT combined with IPOPT yields feasible solutions for $50$ out of $60$ benchmark instances, 47 of which have a relative optimality gap below 1\% and 46 of which are even globally optimal, i.e., which have a gap below $0.01$\%.

\begin{table}
    \caption{Number of instances with feasible solutions found with different relative optimality gaps. The first number corresponds to a gap of less than 0.01\%, the second to a gap of less than 1\% and the third number indicates finding a feasible solution.}
    \label{table_feasible_nmdt_all}
    \centering
    \begin{tabular} {l r  r  r  r}
        \toprule
        {\qquad} & {\qquad \NMDT} & {\qquad T-\NMDT} & {\qquad \DNMDT} & {\qquad T-\DNMDT} \\  
        \midrule
        L1 & \textbf{32}/34/40 & 31/\textbf{35}/41 & 29/33/\textbf{42} & 29/33/40 \\
        L2 & 32/37/45 & \textbf{34}/37/43 & \textbf{34}/38/42 & \textbf{34}/37/42 \\
        L4 & 42/44/48 & 39/44/48 & 37/42/49 & \textbf{45}/\textbf{47}/\textbf{51} \\
        L6 & 43/45/48 & 42/43/47 & 44/\textbf{47}/\textbf{50} & \textbf{46}/\textbf{47}/\textbf{50} \\
        \bottomrule
    \end{tabular}
\end{table}

In summary, both T-\DNMDT and \DNMDT are clearly superior to the previously known \NMDT approach. 
The double discretization and the associated reduction in the number of binary variables while maintaining the same relaxation error are most likely the reason for this. 
Surprisingly, the tightening of the lower bounds in the univariate quadratic terms and the resulting introduction of new constraints does not lead to higher run times. 
Thus, the latter is recommended.
Moreover, T-\DNMDT is slightly ahead of the other methods in computing good solutions for the MIP relaxations that are used by the NLP solver IPOPT to find feasible solutions for the original MIQCQP instances.
Altogether, we consider T-\DNMDT to be the winner among the \NMDT-based methods.

\subsubsection{Comparison with state-of-the-art MIQCQP Solver Gurobi}
\label{subsec:comp_with_gurobi}

Finally, we compare the two winners T-\DNMDT and \HybS of the \NMDT-based and separable Methods (Part I) with the state-of-the-art MIQCQP solver Gurobi 9.5.1.
We perform the comparison in two ways. 
Firstly, with Gurobi's default settings, and secondly, with cuts disabled, i.e., we set the parameter \emph{"Cuts = 0"}.
The reason for running Gurobi again with cuts turned off is that cuts are one of the most important components of MIQCQP/MIP solvers that rely on the structure of the problem.
While constructing the MIP relaxations with T-\DNMDT and \HybS, the original problem is transformed in such a way that Gurobi can no longer recognize the original quadratic structure of the problem. 
However, many cuts would still be valid and applicable in the MIP relaxations, for instance, RLT and PSD cuts.

We start our comparison with showing performance profiles for Gurobi, T-\DNMDT, \HybS, and their variants without cuts ("-NC") on all instances in \cref{gurobi_all}.
As expected, Gurobi performs best for all $L$ values, followed by its variant without cuts in second place.
However, as the depth $L$ increases, the MIP relaxations provide gradually tighter dual bounds.
For $L=6$, T-\DNMDT and \HybS are able to find the best dual bounds for more than 50\% of the cases, while Gurobi delivers the best bounds for roughly 90\% and its variant without cuts for about 70\% of the cases.
Surprisingly, in contrast to T-\DNMDT, disabling cuts in case of \HybS has little effect on the quality of the dual bounds.

As before, we divide the benchmark set into sparse and dense instances.
For sparse instances, the dual bounds computed by T-\DNMDT and \HybS become progressively tighter with increasing $L$; see \cref{gurobi_sparse}.
For $L=4$ and $L=6$, T-\DNMDT and \HybS are able to find the best dual bounds in about 60\% of the instances, while Gurobi delivers the best bounds for roughly 80\%.
Compared to Gurobi-NC, our new methods T-\DNMDT, \HybS, and most notably \HybS-NC perform almost equally well.

In the case of dense instances, a different picture emerges, see \cref{gurobi_dense}.
Again, Gurobi and also Gurobi-NC are dominant for all approximation depths.
However, for $L=1$, T-\DNMDT delivers dual bounds that are within a factor $1.1$ of the dual bounds provided by the variant of Gurobi without cuts.
With higher $L$ values, T-\DNMDT, \HybS, and \HybS-NC compute in about 40\% of the cases the best bounds, while Gurobi yields the best bounds in all cases and Gurobi-NC for roughly 70\% of the instances.

In \cref{table_sgm_gurobi_all} we show the shifted geometric mean values of the run times for solving all instances with Gurobi and the corresponding MIP relaxations constructed with T-\DNMDT and \HybS.
The variants of Gurobi, T-\DNMDT, and \HybS without cuts are also contained.
Gurobi has significantly shorter run times than all other approaches.
However, with $L=1$ and $L=2$, T-\DNMDT, \HybS, T-\DNMDT-NC and \HybS-NC are somewhat faster than Gurobi-NC.
\begin{remark}
Note, that for calculating the shifted geometric mean only those instances are used for which at least one method computed the optimal solution within the run time limit of 8 hours.
Since with higher $L$ values the complexity of the MIP relaxations increases, fewer instances are solved to optimality by T-\DNMDT and \HybS.
Therefore, the shifted geometric mean decreases for Gurobi and Gurobi-NC with higher $L$ values.
This inherent nature of the shifted geometric mean is also the reason why we see different values in  \cref{table_sgm_gurobi_all,table_sgm_nmdt_all} for the same methods.\hfill$\diamond$ 
\end{remark}

In combination with IPOPT as a QP heuristic, T-\DNMDT, \HybS, and their variants without cuts are competitive with Gurobi for high $L$ values when it comes to finding feasible solutions, as \cref{table_feasible_gurobi_all} shows.
\HybS-NC with IPOPT is able to find feasible with a relative optimality gap below 1\% for 48 out of $60$ benchmark instances, while Gurobi finds 50 feasible solutions with a gap below 1\%.
T-\DNMDT computes 46 solutions that are globally optimal, whereas Gurobi achieves this for 50 instances.
Surprisingly, the variant without cuts of \HybS delivers more feasible solutions than its variant with cuts enabled.
Finally, we note that some MIQCQP instances have been solved to global optimality by the MIP relaxation methods, while Gurobi reached the run time limit of 8 hours.
For instance, T-\DNMDT with IPOPT is able to solve the QPLIB instance ``QPLIB\_0698" to global optimality for $L \in \{2,4,6\}$ with a run time below 5 minutes, while Gurobi has a relative optimality gap of more than 5\% after a run time of 8 hours.

Overall, the comparison with Gurobi as a state-of-the-art MIQCQP solver has shown that the new methods T-\DNMDT and \HybS can be relevant for practical applications. 
For sparse instances, the dual bounds provided by T-\DNMDT and \HybS are of similar quality to those provided by Gurobi. 
In terms of MIQCQP-feasible solutions, for most instances the two methods are able to find very high quality solutions in combination with IPOPT as NLP solver.

Moreover, there is still plenty of room for improvement.
First, numerical studies have shown before that an adaptive refinement of nonlinearities drastically decreases run times for solving MINLPs by piecewise linear MIP relaxations; see \cite{Burlacu-et-al:2020} for example.
Hence, an approach with an adaptive refinement of the approximation depth $L$ is even more promising.
Second, \HybS and its variant without cuts \HybS-NC have performed very similarly in our computational study.
In addition, \HybS-NC was relatively close to Gurobi-NC in both solution quality and dual bounds for the MIQCQPs.
Since most MIQCQP-specific cuts can still be integrated into the \HybS approach, we believe that \HybS can be further improved by embedding it in a branch-and-cut solution framework that is able to add MIQCQP-specific cuts, such as BQP and PSD cuts, to the MIP relaxations.
In this way, we obtain both tighter dual bounds and MIP relaxation solutions that are more likely to yield feasible solutions for the MIQCQP in combination with IPOPT.

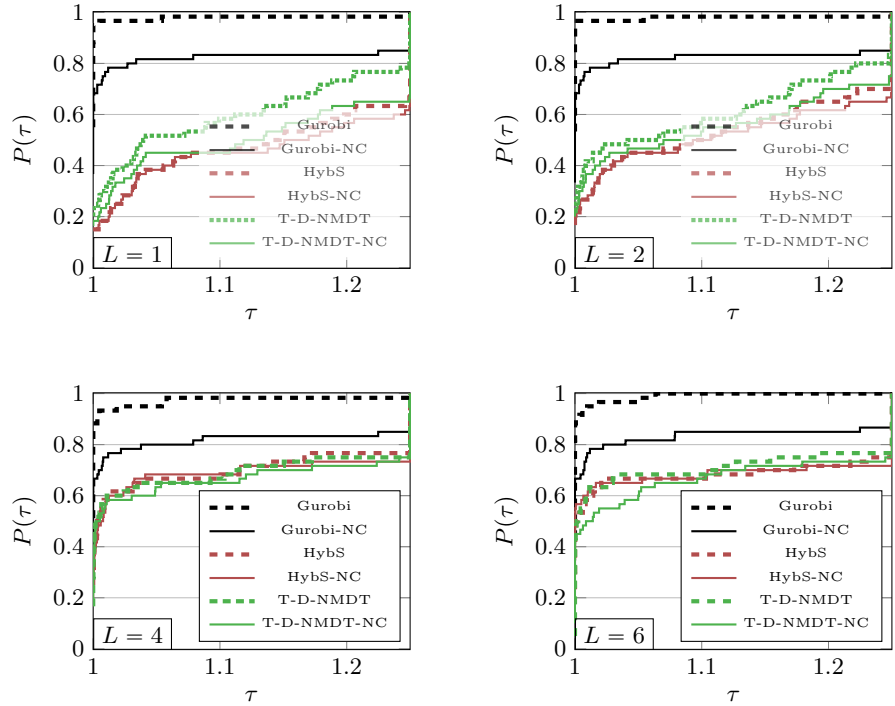
\begin{figure}
    \begin{center}
        \begin{minipage}{0.475\textwidth}
            \begin{center}
                 \begin{tikzpicture}
  \begin{axis}[const plot,
  cycle list={
  {blue,solid},
  {red!40!gray,dashed},
  {black,dotted},
  {brown,dashdotted},
  {green!40!gray!80!black,dashdotdotted},
  {magenta!80!black,densely dotted}},
    xmin=1, xmax=1.25,
    ymin=-0.003, ymax=1.003,
    ymajorgrids,
    ytick={0,0.2,0.4,0.6,0.8,1.0},
    xlabel={$\tau$},
    ylabel={$P(\tau)$},
,
    legend pos={south east},
    legend style={font=\tiny,fill opacity=0.7,draw=none},
    width=\textwidth
    ]
  \addplot+[mark=none, ultra thick, black, dashed] coordinates {
    (1.0000,0.5500)
    (1.0000,0.8833)
    (1.0001,0.8833)
    (1.0001,0.9500)
    (1.0049,0.9500)
    (1.0049,0.9667)
    (1.0415,0.9667)
    (1.0542,0.9833)
    (1.25,1.0000)
  };
  \addlegendentry{Gurobi}
  \addplot+[mark=none, thick, black, solid] coordinates {
    (1.0000,0.3667)
    (1.0000,0.6500)
    (1.0001,0.6667)
    (1.0001,0.6833)
    (1.0032,0.7000)
    (1.0032,0.7167)
    (1.0061,0.7167)
    (1.0068,0.7333)
    (1.0075,0.7500)
    (1.0094,0.7667)
    (1.0102,0.7667)
    (1.0118,0.7833)
    (1.0244,0.7833)
    (1.0270,0.8000)
    (1.0329,0.8000)
    (1.0339,0.8167)
    (1.0771,0.8167)
    (1.0789,0.8333)
    (1.2067,0.8333)
    (1.2248,0.8500)
    (1.25,1.0000)
  };
  \addlegendentry{Gurobi-NC}
  \addplot+[mark=none, ultra thick, red!40!gray, dashed] coordinates {
    (1.0000,0.1500)
    (1.0042,0.1500)
    (1.0043,0.1667)
    (1.0049,0.1833)
    (1.0118,0.1833)
    (1.0122,0.2000)
    (1.0127,0.2000)
    (1.0130,0.2167)
    (1.0141,0.2167)
    (1.0141,0.2333)
    (1.0146,0.2500)
    (1.0221,0.2500)
    (1.0222,0.2667)
    (1.0232,0.2833)
    (1.0301,0.2833)
    (1.0303,0.3000)
    (1.0304,0.3000)
    (1.0314,0.3167)
    (1.0321,0.3167)
    (1.0329,0.3333)
    (1.0340,0.3333)
    (1.0340,0.3500)
    (1.0346,0.3500)
    (1.0351,0.3667)
    (1.0394,0.3667)
    (1.0415,0.3833)
    (1.0542,0.3833)
    (1.0555,0.4000)
    (1.0631,0.4167)
    (1.0645,0.4333)
    (1.0719,0.4333)
    (1.0771,0.4500)
    (1.0993,0.4500)
    (1.1068,0.4667)
    (1.1186,0.4667)
    (1.1318,0.4833)
    (1.1366,0.4833)
    (1.1377,0.5000)
    (1.1471,0.5167)
    (1.1508,0.5167)
    (1.1508,0.5333)
    (1.1693,0.5333)
    (1.1700,0.5500)
    (1.1790,0.5500)
    (1.1790,0.5667)
    (1.1885,0.5667)
    (1.1885,0.5833)
    (1.1979,0.6000)
    (1.2058,0.6000)
    (1.2058,0.6167)
    (1.2067,0.6333)
    (1.2426,0.6333)
    (1.2460,0.6500)
    (1.2494,0.6667)
    (1.25,1.0000)
  };
  \addlegendentry{HybS}
  \addplot+[mark=none, thick, red!40!gray, solid] coordinates {
    (1.0000,0.1500)
    (1.0039,0.1500)
    (1.0042,0.1667)
    (1.0043,0.1667)
    (1.0049,0.1833)
    (1.0118,0.1833)
    (1.0122,0.2000)
    (1.0130,0.2000)
    (1.0131,0.2167)
    (1.0141,0.2167)
    (1.0141,0.2333)
    (1.0159,0.2333)
    (1.0172,0.2500)
    (1.0175,0.2500)
    (1.0221,0.2667)
    (1.0232,0.2667)
    (1.0235,0.2833)
    (1.0303,0.2833)
    (1.0304,0.3000)
    (1.0314,0.3167)
    (1.0321,0.3167)
    (1.0329,0.3333)
    (1.0340,0.3333)
    (1.0340,0.3500)
    (1.0346,0.3500)
    (1.0351,0.3667)
    (1.0394,0.3667)
    (1.0415,0.3833)
    (1.0555,0.3833)
    (1.0555,0.4000)
    (1.0631,0.4000)
    (1.0631,0.4167)
    (1.0645,0.4333)
    (1.0719,0.4333)
    (1.0771,0.4500)
    (1.1377,0.4500)
    (1.1377,0.4667)
    (1.1471,0.4667)
    (1.1471,0.4833)
    (1.1489,0.4833)
    (1.1508,0.5000)
    (1.1693,0.5000)
    (1.1700,0.5167)
    (1.1789,0.5167)
    (1.1790,0.5333)
    (1.1885,0.5500)
    (1.1979,0.5667)
    (1.2052,0.5667)
    (1.2058,0.5833)
    (1.2248,0.5833)
    (1.2371,0.6000)
    (1.2426,0.6000)
    (1.2460,0.6167)
    (1.2494,0.6333)
    (1.25,1.0000)
  };
  \addlegendentry{HybS-NC}
  \addplot+[mark=none, ultra thick, green!40!gray, densely dotted] coordinates {
    (1.0000,0.2000)
    (1.0000,0.2167)
    (1.0001,0.2167)
    (1.0001,0.2333)
    (1.0032,0.2333)
    (1.0039,0.2500)
    (1.0043,0.2500)
    (1.0049,0.2667)
    (1.0050,0.2833)
    (1.0061,0.3000)
    (1.0094,0.3000)
    (1.0102,0.3167)
    (1.0122,0.3167)
    (1.0124,0.3333)
    (1.0141,0.3333)
    (1.0141,0.3500)
    (1.0146,0.3500)
    (1.0153,0.3667)
    (1.0159,0.3833)
    (1.0232,0.3833)
    (1.0232,0.4000)
    (1.0301,0.4000)
    (1.0301,0.4167)
    (1.0314,0.4167)
    (1.0321,0.4333)
    (1.0339,0.4333)
    (1.0340,0.4500)
    (1.0346,0.4667)
    (1.0374,0.4667)
    (1.0374,0.4833)
    (1.0394,0.5000)
    (1.0415,0.5167)
    (1.0645,0.5167)
    (1.0719,0.5333)
    (1.0789,0.5333)
    (1.0885,0.5500)
    (1.0888,0.5667)
    (1.0993,0.5833)
    (1.1079,0.5833)
    (1.1120,0.6000)
    (1.1318,0.6000)
    (1.1345,0.6167)
    (1.1366,0.6167)
    (1.1366,0.6333)
    (1.1508,0.6333)
    (1.1517,0.6500)
    (1.1526,0.6667)
    (1.1693,0.6833)
    (1.1773,0.6833)
    (1.1773,0.7000)
    (1.1789,0.7167)
    (1.1885,0.7167)
    (1.1885,0.7333)
    (1.1979,0.7333)
    (1.2052,0.7500)
    (1.2058,0.7500)
    (1.2058,0.7667)
    (1.2371,0.7667)
    (1.2426,0.7833)
    (1.2460,0.7833)
    (1.2494,0.8000)
    (1.25,1.0000)
  };
  \addlegendentry{T-D-NMDT}
  \addplot+[mark=none, thick, green!40!gray, solid] coordinates {
    (1.0000,0.1667)
    (1.0001,0.1667)
    (1.0001,0.1833)
    (1.0039,0.1833)
    (1.0039,0.2000)
    (1.0043,0.2000)
    (1.0049,0.2167)
    (1.0050,0.2167)
    (1.0061,0.2333)
    (1.0094,0.2333)
    (1.0102,0.2500)
    (1.0124,0.2500)
    (1.0125,0.2667)
    (1.0127,0.2833)
    (1.0131,0.2833)
    (1.0141,0.3000)
    (1.0146,0.3000)
    (1.0153,0.3167)
    (1.0172,0.3167)
    (1.0175,0.3333)
    (1.0235,0.3333)
    (1.0244,0.3500)
    (1.0270,0.3500)
    (1.0301,0.3667)
    (1.0321,0.3667)
    (1.0321,0.3833)
    (1.0340,0.3833)
    (1.0346,0.4000)
    (1.0351,0.4000)
    (1.0374,0.4167)
    (1.0394,0.4333)
    (1.0415,0.4500)
    (1.1068,0.4500)
    (1.1079,0.4667)
    (1.1120,0.4833)
    (1.1186,0.5000)
    (1.1345,0.5000)
    (1.1346,0.5167)
    (1.1366,0.5333)
    (1.1471,0.5333)
    (1.1489,0.5500)
    (1.1517,0.5500)
    (1.1517,0.5667)
    (1.1526,0.5667)
    (1.1693,0.5833)
    (1.1700,0.5833)
    (1.1773,0.6000)
    (1.1789,0.6167)
    (1.1885,0.6167)
    (1.1885,0.6333)
    (1.2058,0.6333)
    (1.2058,0.6500)
    (1.2494,0.6500)
    (1.2494,0.6667)
    (1.25,1.0000)
  };
  \addlegendentry{T-D-NMDT-NC}
  \draw node[right,draw,align=left] {$L=1$\\};
  \end{axis}
\end{tikzpicture}
            \end{center}
        \end{minipage}
        \hfill
        \begin{minipage}{0.475\textwidth}
            \begin{center}
                 \begin{tikzpicture}
  \begin{axis}[const plot,
  cycle list={
  {blue,solid},
  {red!40!gray,dashed},
  {black,dotted},
  {brown,dashdotted},
  {green!40!gray!80!black,dashdotdotted},
  {magenta!80!black,densely dotted}},
    xmin=1, xmax=1.25,
    ymin=-0.003, ymax=1.003,
    ymajorgrids,
    ytick={0,0.2,0.4,0.6,0.8,1.0},
    xlabel={$\tau$},
    ylabel={$P(\tau)$},
,
    legend pos={south east},
    legend style={font=\tiny, fill opacity=0.7, draw=none},
    width=\textwidth
    ]
  \addplot+[mark=none, ultra thick, black, dashed] coordinates {
    (1.0000,0.5333)
    (1.0000,0.9000)
    (1.0001,0.9000)
    (1.0001,0.9667)
    (1.0415,0.9667)
    (1.0542,0.9833)
    (1.25,1.0000)
  };
  \addlegendentry{Gurobi}
  \addplot+[mark=none, thick, black, solid] coordinates {
    (1.0000,0.3333)
    (1.0000,0.6333)
    (1.0001,0.6500)
    (1.0001,0.6833)
    (1.0031,0.6833)
    (1.0032,0.7000)
    (1.0032,0.7167)
    (1.0041,0.7167)
    (1.0044,0.7333)
    (1.0059,0.7333)
    (1.0068,0.7500)
    (1.0070,0.7500)
    (1.0075,0.7667)
    (1.0100,0.7667)
    (1.0118,0.7833)
    (1.0237,0.7833)
    (1.0270,0.8000)
    (1.0308,0.8000)
    (1.0339,0.8167)
    (1.0698,0.8167)
    (1.0789,0.8333)
    (1.2227,0.8333)
    (1.2248,0.8500)
    (1.25,1.0000)
  };
  \addlegendentry{Gurobi-NC}
  \addplot+[mark=none,ultra thick, red!40!gray, dashed] coordinates {
    (1.0000,0.1667)
    (1.0000,0.2000)
    (1.0001,0.2000)
    (1.0007,0.2167)
    (1.0034,0.2167)
    (1.0036,0.2333)
    (1.0040,0.2667)
    (1.0086,0.2667)
    (1.0098,0.2833)
    (1.0118,0.2833)
    (1.0140,0.3000)
    (1.0141,0.3000)
    (1.0141,0.3167)
    (1.0142,0.3167)
    (1.0149,0.3333)
    (1.0158,0.3333)
    (1.0185,0.3500)
    (1.0196,0.3667)
    (1.0221,0.3667)
    (1.0221,0.3833)
    (1.0271,0.3833)
    (1.0287,0.4000)
    (1.0308,0.4167)
    (1.0339,0.4167)
    (1.0372,0.4333)
    (1.0415,0.4500)
    (1.0697,0.4500)
    (1.0698,0.4667)
    (1.0812,0.4667)
    (1.0827,0.4833)
    (1.0861,0.4833)
    (1.0930,0.5000)
    (1.1075,0.5000)
    (1.1076,0.5167)
    (1.1158,0.5167)
    (1.1159,0.5333)
    (1.1241,0.5500)
    (1.1366,0.5500)
    (1.1366,0.5667)
    (1.1468,0.5667)
    (1.1468,0.5833)
    (1.1540,0.5833)
    (1.1568,0.6000)
    (1.1693,0.6000)
    (1.1693,0.6167)
    (1.1709,0.6167)
    (1.1750,0.6333)
    (1.1775,0.6333)
    (1.1775,0.6500)
    (1.2018,0.6500)
    (1.2131,0.6667)
    (1.2163,0.6833)
    (1.2218,0.6833)
    (1.2227,0.7000)
    (1.2477,0.7000)
    (1.2494,0.7167)
    (1.25,1.0000)
  };
  \addlegendentry{HybS}
  \addplot+[mark=none, thick, red!40!gray, solid] coordinates {
    (1.0000,0.1833)
    (1.0000,0.2000)
    (1.0008,0.2000)
    (1.0013,0.2167)
    (1.0034,0.2167)
    (1.0036,0.2333)
    (1.0038,0.2500)
    (1.0040,0.2500)
    (1.0041,0.2667)
    (1.0098,0.2667)
    (1.0100,0.2833)
    (1.0140,0.2833)
    (1.0141,0.3000)
    (1.0149,0.3000)
    (1.0153,0.3167)
    (1.0154,0.3333)
    (1.0158,0.3333)
    (1.0185,0.3500)
    (1.0196,0.3500)
    (1.0196,0.3667)
    (1.0198,0.3667)
    (1.0221,0.3833)
    (1.0271,0.3833)
    (1.0287,0.4000)
    (1.0308,0.4000)
    (1.0308,0.4167)
    (1.0372,0.4167)
    (1.0372,0.4333)
    (1.0415,0.4500)
    (1.0789,0.4500)
    (1.0812,0.4667)
    (1.0827,0.4667)
    (1.0827,0.4833)
    (1.0930,0.4833)
    (1.0930,0.5000)
    (1.1012,0.5000)
    (1.1075,0.5167)
    (1.1076,0.5167)
    (1.1158,0.5333)
    (1.1356,0.5333)
    (1.1366,0.5500)
    (1.1436,0.5500)
    (1.1468,0.5667)
    (1.1688,0.5667)
    (1.1693,0.5833)
    (1.1709,0.5833)
    (1.1750,0.6000)
    (1.1775,0.6000)
    (1.1775,0.6167)
    (1.2131,0.6167)
    (1.2131,0.6333)
    (1.2163,0.6500)
    (1.2248,0.6500)
    (1.2458,0.6667)
    (1.2477,0.6667)
    (1.2494,0.6833)
    (1.25,1.0000)
  };
  \addlegendentry{HybS-NC}
  \addplot+[mark=none, ultra thick, green!40!gray, densely dotted] coordinates {
    (1.0000,0.2167)
    (1.0000,0.2500)
    (1.0007,0.2500)
    (1.0008,0.2667)
    (1.0013,0.2667)
    (1.0013,0.2833)
    (1.0020,0.2833)
    (1.0020,0.3000)
    (1.0024,0.3167)
    (1.0033,0.3167)
    (1.0034,0.3333)
    (1.0044,0.3333)
    (1.0053,0.3500)
    (1.0059,0.3667)
    (1.0068,0.3667)
    (1.0070,0.3833)
    (1.0075,0.3833)
    (1.0084,0.4000)
    (1.0085,0.4000)
    (1.0085,0.4167)
    (1.0140,0.4167)
    (1.0140,0.4333)
    (1.0141,0.4333)
    (1.0142,0.4500)
    (1.0196,0.4500)
    (1.0197,0.4667)
    (1.0237,0.4667)
    (1.0237,0.4833)
    (1.0372,0.4833)
    (1.0415,0.5000)
    (1.0542,0.5000)
    (1.0664,0.5167)
    (1.0697,0.5333)
    (1.0827,0.5333)
    (1.0835,0.5500)
    (1.0930,0.5500)
    (1.0994,0.5667)
    (1.1011,0.5667)
    (1.1012,0.5833)
    (1.1241,0.5833)
    (1.1270,0.6000)
    (1.1332,0.6167)
    (1.1356,0.6333)
    (1.1366,0.6333)
    (1.1436,0.6500)
    (1.1468,0.6500)
    (1.1540,0.6667)
    (1.1568,0.6667)
    (1.1688,0.6833)
    (1.1693,0.7000)
    (1.1709,0.7167)
    (1.1750,0.7167)
    (1.1775,0.7333)
    (1.1958,0.7333)
    (1.2016,0.7500)
    (1.2018,0.7667)
    (1.2131,0.7667)
    (1.2163,0.7833)
    (1.2218,0.8000)
    (1.2477,0.8000)
    (1.2477,0.8333)
    (1.25,1.0000)
  };
  \addlegendentry{T-D-NMDT}
  \addplot+[mark=none, thick, green!40!gray, solid] coordinates {
    (1.0000,0.2000)
    (1.0000,0.2167)
    (1.0013,0.2167)
    (1.0013,0.2333)
    (1.0020,0.2500)
    (1.0024,0.2500)
    (1.0031,0.2667)
    (1.0032,0.2667)
    (1.0033,0.2833)
    (1.0034,0.2833)
    (1.0034,0.3000)
    (1.0059,0.3000)
    (1.0059,0.3167)
    (1.0070,0.3167)
    (1.0070,0.3333)
    (1.0084,0.3333)
    (1.0085,0.3500)
    (1.0086,0.3667)
    (1.0140,0.3667)
    (1.0140,0.3833)
    (1.0154,0.3833)
    (1.0158,0.4000)
    (1.0197,0.4000)
    (1.0198,0.4167)
    (1.0221,0.4167)
    (1.0237,0.4333)
    (1.0270,0.4333)
    (1.0271,0.4500)
    (1.0372,0.4500)
    (1.0415,0.4667)
    (1.0664,0.4667)
    (1.0664,0.4833)
    (1.0698,0.4833)
    (1.0698,0.5000)
    (1.0835,0.5000)
    (1.0861,0.5167)
    (1.0930,0.5167)
    (1.0994,0.5333)
    (1.1011,0.5500)
    (1.1270,0.5500)
    (1.1270,0.5667)
    (1.1436,0.5667)
    (1.1436,0.5833)
    (1.1468,0.5833)
    (1.1540,0.6000)
    (1.1568,0.6000)
    (1.1688,0.6167)
    (1.1693,0.6333)
    (1.1775,0.6333)
    (1.1776,0.6500)
    (1.1884,0.6667)
    (1.1936,0.6833)
    (1.1958,0.7000)
    (1.2131,0.7000)
    (1.2163,0.7167)
    (1.2458,0.7167)
    (1.2477,0.7333)
    (1.2477,0.7500)
    (1.25,1.0000)
  };
  \addlegendentry{T-D-NMDT-NC}
  \draw node[right,draw,align=left] {$L=2$\\};
  \end{axis}
\end{tikzpicture}
            \end{center}
        \end{minipage}
    \end{center}
    \hfill
    \begin{center}
        \begin{minipage}{0.475\textwidth}
            \begin{center}
                 \begin{tikzpicture}
  \begin{axis}[const plot,
  cycle list={
  {blue,solid},
  {red!40!gray,dashed},
  {black,dotted},
  {brown,dashdotted},
  {green!40!gray!80!black,dashdotdotted},
  {magenta!80!black,densely dotted}},
    xmin=1, xmax=1.25,
    ymin=-0.003, ymax=1.003,
    ymajorgrids,
    ytick={0,0.2,0.4,0.6,0.8,1.0},
    xlabel={$\tau$},
    ylabel={$P(\tau)$},
,
    legend pos={south east},
    legend style={font=\tiny},
    width=\textwidth
    ]
  \addplot+[mark=none, ultra thick, black, dashed] coordinates {
    (1.0000,0.4500)
    (1.0000,0.8167)
    (1.0001,0.8167)
    (1.0001,0.8833)
    (1.0032,0.8833)
    (1.0033,0.9000)
    (1.0037,0.9000)
    (1.0037,0.9167)
    (1.0044,0.9167)
    (1.0048,0.9333)
    (1.0123,0.9333)
    (1.0186,0.9500)
    (1.0535,0.9500)
    (1.0542,0.9667)
    (1.0579,0.9833)
    (1.25,1.0000)
  };
  \addlegendentry{Gurobi}
  \addplot+[mark=none, thick, black, solid] coordinates {
    (1.0000,0.3333)
    (1.0000,0.6167)
    (1.0001,0.6333)
    (1.0001,0.6667)
    (1.0032,0.6667)
    (1.0033,0.6833)
    (1.0037,0.6833)
    (1.0044,0.7000)
    (1.0067,0.7000)
    (1.0068,0.7167)
    (1.0075,0.7333)
    (1.0078,0.7333)
    (1.0079,0.7500)
    (1.0115,0.7500)
    (1.0118,0.7667)
    (1.0186,0.7667)
    (1.0219,0.7833)
    (1.0351,0.7833)
    (1.0377,0.8000)
    (1.0579,0.8000)
    (1.0789,0.8167)
    (1.0864,0.8333)
    (1.2237,0.8333)
    (1.2248,0.8500)
    (1.25,1.0000)
  };
  \addlegendentry{Gurobi-NC}
  \addplot+[mark=none,ultra thick, red!40!gray, dashed] coordinates {
    (1.0000,0.2500)
    (1.0000,0.3000)
    (1.0002,0.3000)
    (1.0002,0.3167)
    (1.0005,0.3167)
    (1.0005,0.3500)
    (1.0007,0.3500)
    (1.0008,0.3667)
    (1.0010,0.3667)
    (1.0010,0.4000)
    (1.0011,0.4000)
    (1.0012,0.4167)
    (1.0016,0.4167)
    (1.0016,0.4333)
    (1.0022,0.4333)
    (1.0030,0.4500)
    (1.0031,0.4667)
    (1.0031,0.4833)
    (1.0032,0.5000)
    (1.0033,0.5000)
    (1.0036,0.5167)
    (1.0054,0.5167)
    (1.0060,0.5333)
    (1.0064,0.5333)
    (1.0065,0.5500)
    (1.0075,0.5500)
    (1.0078,0.5667)
    (1.0079,0.5667)
    (1.0083,0.5833)
    (1.0107,0.5833)
    (1.0113,0.6000)
    (1.0115,0.6167)
    (1.0304,0.6167)
    (1.0311,0.6333)
    (1.0324,0.6333)
    (1.0324,0.6500)
    (1.0377,0.6500)
    (1.0410,0.6667)
    (1.0926,0.6667)
    (1.1005,0.6833)
    (1.1154,0.6833)
    (1.1154,0.7000)
    (1.1159,0.7167)
    (1.1295,0.7167)
    (1.1370,0.7333)
    (1.1504,0.7333)
    (1.1666,0.7500)
    (1.1726,0.7500)
    (1.1726,0.7667)
    (1.25,1.0000)
  };
  \addlegendentry{HybS}
  \addplot+[mark=none, thick, red!40!gray, solid] coordinates {
    (1.0000,0.2000)
    (1.0000,0.2667)
    (1.0001,0.2667)
    (1.0001,0.2833)
    (1.0002,0.2833)
    (1.0002,0.3000)
    (1.0005,0.3000)
    (1.0005,0.3167)
    (1.0007,0.3333)
    (1.0010,0.3333)
    (1.0010,0.3667)
    (1.0011,0.3667)
    (1.0012,0.3833)
    (1.0016,0.3833)
    (1.0016,0.4000)
    (1.0020,0.4000)
    (1.0022,0.4167)
    (1.0030,0.4167)
    (1.0030,0.4333)
    (1.0032,0.4333)
    (1.0032,0.4500)
    (1.0033,0.4500)
    (1.0033,0.4667)
    (1.0036,0.4667)
    (1.0037,0.4833)
    (1.0054,0.4833)
    (1.0060,0.5000)
    (1.0078,0.5000)
    (1.0078,0.5167)
    (1.0083,0.5167)
    (1.0083,0.5333)
    (1.0095,0.5333)
    (1.0096,0.5500)
    (1.0097,0.5667)
    (1.0099,0.5833)
    (1.0118,0.5833)
    (1.0123,0.6000)
    (1.0219,0.6000)
    (1.0238,0.6167)
    (1.0241,0.6167)
    (1.0262,0.6333)
    (1.0311,0.6333)
    (1.0311,0.6500)
    (1.0324,0.6667)
    (1.0410,0.6667)
    (1.0410,0.6833)
    (1.1154,0.6833)
    (1.1154,0.7000)
    (1.1159,0.7167)
    (1.1726,0.7167)
    (1.1726,0.7333)
    (1.2414,0.7333)
    (1.25,1.0000)
  };
  \addlegendentry{HybS-NC}
  \addplot+[mark=none,ultra thick, green!40!gray, densely dashed] coordinates {
    (1.0000,0.1667)
    (1.0000,0.2500)
    (1.0001,0.2500)
    (1.0001,0.2833)
    (1.0002,0.2833)
    (1.0002,0.3167)
    (1.0003,0.3333)
    (1.0003,0.3500)
    (1.0004,0.3667)
    (1.0005,0.3833)
    (1.0007,0.3833)
    (1.0007,0.4000)
    (1.0008,0.4000)
    (1.0009,0.4167)
    (1.0010,0.4167)
    (1.0011,0.4333)
    (1.0011,0.4500)
    (1.0012,0.4500)
    (1.0014,0.4667)
    (1.0016,0.4833)
    (1.0019,0.5000)
    (1.0019,0.5167)
    (1.0048,0.5167)
    (1.0050,0.5333)
    (1.0051,0.5500)
    (1.0060,0.5500)
    (1.0064,0.5667)
    (1.0095,0.5667)
    (1.0095,0.5833)
    (1.0099,0.5833)
    (1.0107,0.6000)
    (1.0238,0.6000)
    (1.0241,0.6167)
    (1.0262,0.6167)
    (1.0275,0.6333)
    (1.0324,0.6333)
    (1.0351,0.6500)
    (1.0864,0.6500)
    (1.0926,0.6667)
    (1.1005,0.6667)
    (1.1036,0.6833)
    (1.1110,0.7000)
    (1.1154,0.7167)
    (1.1370,0.7167)
    (1.1504,0.7333)
    (1.1666,0.7333)
    (1.1725,0.7500)
    (1.25,1.0000)
  };
  \addlegendentry{T-D-NMDT}
  \addplot+[mark=none, thick, green!40!gray, solid] coordinates {
    (1.0000,0.1667)
    (1.0000,0.2500)
    (1.0001,0.2500)
    (1.0001,0.2833)
    (1.0002,0.3000)
    (1.0002,0.3333)
    (1.0003,0.3333)
    (1.0003,0.3667)
    (1.0005,0.3667)
    (1.0005,0.3833)
    (1.0007,0.3833)
    (1.0007,0.4000)
    (1.0009,0.4000)
    (1.0010,0.4167)
    (1.0011,0.4167)
    (1.0011,0.4333)
    (1.0014,0.4333)
    (1.0016,0.4500)
    (1.0016,0.4667)
    (1.0019,0.4833)
    (1.0020,0.5000)
    (1.0051,0.5000)
    (1.0054,0.5167)
    (1.0065,0.5167)
    (1.0067,0.5333)
    (1.0083,0.5333)
    (1.0084,0.5500)
    (1.0095,0.5667)
    (1.0099,0.5667)
    (1.0107,0.5833)
    (1.0275,0.5833)
    (1.0304,0.6000)
    (1.0410,0.6000)
    (1.0488,0.6167)
    (1.0491,0.6333)
    (1.0535,0.6500)
    (1.1110,0.6500)
    (1.1154,0.6667)
    (1.1159,0.6667)
    (1.1205,0.6833)
    (1.1295,0.7000)
    (1.1725,0.7000)
    (1.1726,0.7167)
    (1.2237,0.7333)
    (1.2248,0.7333)
    (1.2414,0.7500)
    (1.25,1.0000)
  };
  \addlegendentry{T-D-NMDT-NC}
  \draw node[right,draw,align=left] {$L=4$\\};
  \end{axis}
\end{tikzpicture}
            \end{center}
        \end{minipage}
        \hfill
        \begin{minipage}{0.475\textwidth}
            \begin{center}
                 \begin{tikzpicture}
  \begin{axis}[const plot,
  cycle list={
  {blue,solid},
  {red!40!gray,dashed},
  {black,dotted},
  {brown,dashdotted},
  {green!40!gray!80!black,dashdotdotted},
  {magenta!80!black,densely dotted}},
    xmin=1, xmax=1.25,
    ymin=-0.003, ymax=1.003,
    ymajorgrids,
    ytick={0,0.2,0.4,0.6,0.8,1.0},
    xlabel={$\tau$},
    ylabel={$P(\tau)$},
,
    legend pos={south east},
    legend style={font=\tiny},
    width=\textwidth
    ]
  \addplot+[mark=none, ultra thick, black, dashed] coordinates {
    (1.0000,0.4333)
    (1.0000,0.8167)
    (1.0001,0.8167)
    (1.0001,0.8833)
    (1.0007,0.8833)
    (1.0007,0.9000)
    (1.0034,0.9000)
    (1.0035,0.9167)
    (1.0050,0.9167)
    (1.0059,0.9333)
    (1.0095,0.9333)
    (1.0096,0.9500)
    (1.0190,0.9500)
    (1.0190,0.9667)
    (1.0499,0.9667)
    (1.0503,0.9833)
    (1.0630,0.9833)
    (1.0635,1.0000)
    (1.25,1.0000)
  };
  \addlegendentry{Gurobi}
  \addplot+[mark=none, thick, black, solid] coordinates {
    (1.0000,0.3167)
    (1.0000,0.6333)
    (1.0001,0.6500)
    (1.0001,0.6667)
    (1.0042,0.6667)
    (1.0044,0.6833)
    (1.0066,0.6833)
    (1.0066,0.7000)
    (1.0075,0.7167)
    (1.0075,0.7333)
    (1.0087,0.7333)
    (1.0088,0.7500)
    (1.0095,0.7500)
    (1.0096,0.7667)
    (1.0113,0.7667)
    (1.0118,0.7833)
    (1.0207,0.7833)
    (1.0222,0.8000)
    (1.0392,0.8000)
    (1.0399,0.8167)
    (1.0635,0.8167)
    (1.0787,0.8333)
    (1.0789,0.8500)
    (1.2173,0.8500)
    (1.2248,0.8667)
    (1.25,1.0000)
  };
  \addlegendentry{Gurobi-NC}
  \addplot+[mark=none, ultra thick, red!40!gray, dashed] coordinates {
    (1.0000,0.2333)
    (1.0000,0.3000)
    (1.0001,0.3000)
    (1.0001,0.4000)
    (1.0002,0.4000)
    (1.0002,0.4500)
    (1.0004,0.4500)
    (1.0004,0.4667)
    (1.0006,0.4667)
    (1.0007,0.4833)
    (1.0008,0.5000)
    (1.0013,0.5000)
    (1.0018,0.5167)
    (1.0026,0.5167)
    (1.0034,0.5333)
    (1.0059,0.5333)
    (1.0062,0.5500)
    (1.0063,0.5500)
    (1.0065,0.5667)
    (1.0082,0.5667)
    (1.0084,0.5833)
    (1.0088,0.5833)
    (1.0094,0.6000)
    (1.0141,0.6000)
    (1.0144,0.6167)
    (1.0149,0.6167)
    (1.0157,0.6333)
    (1.0190,0.6333)
    (1.0207,0.6500)
    (1.0222,0.6500)
    (1.0244,0.6667)
    (1.0964,0.6667)
    (1.1033,0.6833)
    (1.1343,0.6833)
    (1.1409,0.7000)
    (1.1540,0.7000)
    (1.1729,0.7167)
    (1.1906,0.7167)
    (1.2173,0.7333)
    (1.2248,0.7333)
    (1.2300,0.7500)
    (1.25,1.0000)
  };
  \addlegendentry{HybS}
  \addplot+[mark=none, thick, red!40!gray, solid] coordinates {
    (1.0000,0.2167)
    (1.0000,0.3500)
    (1.0001,0.3500)
    (1.0001,0.4500)
    (1.0002,0.4500)
    (1.0002,0.5333)
    (1.0003,0.5333)
    (1.0003,0.5500)
    (1.0007,0.5500)
    (1.0007,0.5667)
    (1.0045,0.5667)
    (1.0050,0.5833)
    (1.0065,0.5833)
    (1.0066,0.6000)
    (1.0096,0.6000)
    (1.0101,0.6167)
    (1.0112,0.6333)
    (1.0157,0.6333)
    (1.0165,0.6500)
    (1.0522,0.6500)
    (1.0530,0.6667)
    (1.0964,0.6667)
    (1.1033,0.6833)
    (1.1045,0.7000)
    (1.1793,0.7000)
    (1.1827,0.7167)
    (1.25,1.0000)
  };
  \addlegendentry{HybS-NC}
  \addplot+[mark=none, ultra thick, green!40!gray, dashed] coordinates {
    (1.0000,0.0500)
    (1.0000,0.3000)
    (1.0001,0.3000)
    (1.0001,0.4167)
    (1.0002,0.4333)
    (1.0003,0.4333)
    (1.0005,0.4667)
    (1.0006,0.4833)
    (1.0008,0.4833)
    (1.0013,0.5000)
    (1.0018,0.5000)
    (1.0026,0.5167)
    (1.0035,0.5167)
    (1.0037,0.5333)
    (1.0044,0.5333)
    (1.0045,0.5500)
    (1.0075,0.5500)
    (1.0080,0.5667)
    (1.0082,0.5833)
    (1.0084,0.5833)
    (1.0087,0.6000)
    (1.0094,0.6000)
    (1.0095,0.6167)
    (1.0112,0.6167)
    (1.0113,0.6333)
    (1.0244,0.6333)
    (1.0279,0.6500)
    (1.0298,0.6667)
    (1.0303,0.6833)
    (1.0964,0.6833)
    (1.1033,0.7000)
    (1.1045,0.7000)
    (1.1143,0.7167)
    (1.1154,0.7167)
    (1.1270,0.7333)
    (1.1409,0.7333)
    (1.1540,0.7500)
    (1.1827,0.7500)
    (1.1906,0.7667)
    (1.25,1.0000)
  };
  \addlegendentry{T-D-NMDT}
  \addplot+[mark=none, thick, green!40!gray, solid] coordinates {
    (1.0000,0.1000)
    (1.0000,0.3500)
    (1.0001,0.3500)
    (1.0001,0.4167)
    (1.0002,0.4167)
    (1.0002,0.4333)
    (1.0003,0.4500)
    (1.0037,0.4500)
    (1.0042,0.4667)
    (1.0062,0.4667)
    (1.0063,0.4833)
    (1.0084,0.4833)
    (1.0087,0.5000)
    (1.0118,0.5000)
    (1.0141,0.5167)
    (1.0144,0.5167)
    (1.0149,0.5333)
    (1.0165,0.5333)
    (1.0190,0.5500)
    (1.0303,0.5500)
    (1.0351,0.5667)
    (1.0392,0.5833)
    (1.0399,0.5833)
    (1.0499,0.6000)
    (1.0503,0.6000)
    (1.0505,0.6167)
    (1.0522,0.6333)
    (1.0530,0.6333)
    (1.0630,0.6500)
    (1.0789,0.6500)
    (1.0964,0.6667)
    (1.1033,0.6833)
    (1.1143,0.6833)
    (1.1154,0.7000)
    (1.1270,0.7000)
    (1.1343,0.7167)
    (1.1729,0.7167)
    (1.1793,0.7333)
    (1.2300,0.7333)
    (1.2448,0.7500)
    (1.25,1.0000)
  };
  \addlegendentry{T-D-NMDT-NC}
  \draw node[right,draw,align=left] {$L=6$\\};
  \end{axis}
\end{tikzpicture}
            \end{center}
        \end{minipage}
    \end{center}
    \hfill
    \caption{Performance profiles on dual bounds of best MIP relaxation compared to Gurobi as MIQCQP solver, with and without cuts, on all 60 instances. }\label{gurobi_all}
\end{figure}

\begin{table}
    \caption{Shifted geometric mean for run times on all instances for best MIP relaxation compared to Gurobi as MIQCQP solver with cuts and without cuts (-NC).}
    \label{table_sgm_gurobi_all}
    \centering
    \begin{tabular} {l r r r r r r}
        \toprule
        {\qquad} & {\quad \HybS} & {\quad \HybS-NC} & {\quad T-\DNMDT} & {\quad T-\DNMDT-NC} & {\quad Gurobi} & {\quad Gurobi-NC} \\  
        \midrule
         L1 & 188.23 & 225.0 & 163.52 & 244.73 & \textbf{77.32} & 388.74 \\
         L2 & 342.32 & 279.0 & 266.37 & 340.8 & \textbf{54.0} & 307.03 \\
         L4 & 1008.09 & 964.63 & 950.73 & 1012.11 & \textbf{25.16} & 193.75 \\
         L6 & 2548.31 & 2315.29 & 1665.28 & 1618.91 & \textbf{20.75} & 174.54 \\
        \bottomrule
    \end{tabular}
\end{table}

\begin{figure}
    \begin{center}
        \begin{minipage}{0.475\textwidth}
            \begin{center}
                 \begin{tikzpicture}
  \begin{axis}[const plot,
  cycle list={
  {blue,solid},
  {red!40!gray,dashed},
  {black,dotted},
  {brown,dashdotted},
  {green!40!gray!80!black,dashdotdotted},
  {magenta!80!black,densely dotted}},
    xmin=1, xmax=1.25,
    ymin=-0.003, ymax=1.003,
    ymajorgrids,
    ytick={0,0.2,0.4,0.6,0.8,1.0},
    xlabel={$\tau$},
    ylabel={$P(\tau)$},
,
    legend pos={south east},
    legend style={font=\tiny, fill opacity=0.7, draw=none},
    width=\textwidth
    ]
  \addplot+[mark=none, ultra thick, black, dashed] coordinates {
    (1.0000,0.5000)
    (1.0000,0.8000)
    (1.0001,0.8333)
    (1.0001,0.9000)
    (1.0049,0.9000)
    (1.0049,0.9333)
    (1.0415,0.9333)
    (1.0542,0.9667)
    (1.25,1.0000)
  };
  \addlegendentry{Gurobi}
  \addplot+[mark=none, thick, black, solid] coordinates {
    (1.0000,0.5000)
    (1.0000,0.6667)
    (1.0001,0.6667)
    (1.0001,0.7000)
    (1.0032,0.7333)
    (1.0032,0.7667)
    (1.0049,0.7667)
    (1.0068,0.8000)
    (1.0075,0.8333)
    (1.0094,0.8667)
    (1.0102,0.8667)
    (1.0118,0.9000)
    (1.0244,0.9000)
    (1.0270,0.9333)
    (1.0321,0.9333)
    (1.0339,0.9667)
    (1.25,1.0000)
  };
  \addlegendentry{Gurobi-NC}
  \addplot+[mark=none, ultra thick, red!40!gray, dashed] coordinates {
    (1.0000,0.0667)
    (1.0042,0.0667)
    (1.0043,0.1000)
    (1.0049,0.1333)
    (1.0127,0.1333)
    (1.0130,0.1667)
    (1.0141,0.1667)
    (1.0141,0.2000)
    (1.0146,0.2333)
    (1.0221,0.2333)
    (1.0222,0.2667)
    (1.0232,0.3000)
    (1.0301,0.3000)
    (1.0303,0.3333)
    (1.0304,0.3333)
    (1.0314,0.3667)
    (1.0339,0.3667)
    (1.0340,0.4000)
    (1.0394,0.4000)
    (1.0415,0.4333)
    (1.0542,0.4333)
    (1.0631,0.4667)
    (1.0771,0.5000)
    (1.1366,0.5000)
    (1.1377,0.5333)
    (1.1471,0.5667)
    (1.1508,0.5667)
    (1.1508,0.6000)
    (1.1693,0.6000)
    (1.1700,0.6333)
    (1.1790,0.6333)
    (1.1790,0.6667)
    (1.1885,0.6667)
    (1.1885,0.7000)
    (1.1979,0.7333)
    (1.2058,0.7333)
    (1.2058,0.7667)
    (1.2460,0.8000)
    (1.2494,0.8333)
    (1.25,1.0000)
  };
  \addlegendentry{HybS}
  \addplot+[mark=none, thick, red!40!gray, solid] coordinates {
    (1.0000,0.0667)
    (1.0039,0.0667)
    (1.0042,0.1000)
    (1.0043,0.1000)
    (1.0049,0.1333)
    (1.0130,0.1333)
    (1.0131,0.1667)
    (1.0141,0.1667)
    (1.0141,0.2000)
    (1.0159,0.2000)
    (1.0172,0.2333)
    (1.0175,0.2333)
    (1.0221,0.2667)
    (1.0232,0.2667)
    (1.0235,0.3000)
    (1.0303,0.3000)
    (1.0304,0.3333)
    (1.0314,0.3667)
    (1.0339,0.3667)
    (1.0340,0.4000)
    (1.0394,0.4000)
    (1.0415,0.4333)
    (1.0631,0.4333)
    (1.0631,0.4667)
    (1.0771,0.5000)
    (1.1377,0.5000)
    (1.1377,0.5333)
    (1.1471,0.5333)
    (1.1471,0.5667)
    (1.1508,0.6000)
    (1.1693,0.6000)
    (1.1700,0.6333)
    (1.1789,0.6333)
    (1.1790,0.6667)
    (1.1885,0.7000)
    (1.1979,0.7333)
    (1.2058,0.7667)
    (1.2460,0.8000)
    (1.2494,0.8333)
    (1.25,1.0000)
  };
  \addlegendentry{HybS-NC}
  \addplot+[mark=none, ultra thick, green!40!gray, densely dotted] coordinates {
    (1.0000,0.0667)
    (1.0032,0.0667)
    (1.0039,0.1000)
    (1.0043,0.1000)
    (1.0049,0.1333)
    (1.0094,0.1333)
    (1.0102,0.1667)
    (1.0118,0.1667)
    (1.0124,0.2000)
    (1.0141,0.2000)
    (1.0141,0.2333)
    (1.0146,0.2333)
    (1.0153,0.2667)
    (1.0159,0.3000)
    (1.0232,0.3000)
    (1.0232,0.3333)
    (1.0301,0.3333)
    (1.0301,0.3667)
    (1.0314,0.3667)
    (1.0321,0.4000)
    (1.0340,0.4000)
    (1.0346,0.4333)
    (1.0394,0.4667)
    (1.0415,0.5000)
    (1.0771,0.5000)
    (1.1120,0.5333)
    (1.1345,0.5667)
    (1.1366,0.5667)
    (1.1366,0.6000)
    (1.1508,0.6000)
    (1.1517,0.6333)
    (1.1693,0.6667)
    (1.1773,0.6667)
    (1.1773,0.7000)
    (1.1789,0.7333)
    (1.1885,0.7333)
    (1.1885,0.7667)
    (1.2058,0.7667)
    (1.2058,0.8000)
    (1.2460,0.8000)
    (1.2494,0.8333)
    (1.25,1.0000)
  };
  \addlegendentry{T-D-NMDT}
  \addplot+[mark=none, thick, green!40!gray, solid] coordinates {
    (1.0000,0.0333)
    (1.0039,0.0333)
    (1.0039,0.0667)
    (1.0043,0.0667)
    (1.0049,0.1000)
    (1.0094,0.1000)
    (1.0102,0.1333)
    (1.0124,0.1333)
    (1.0125,0.1667)
    (1.0127,0.2000)
    (1.0131,0.2000)
    (1.0141,0.2333)
    (1.0146,0.2333)
    (1.0153,0.2667)
    (1.0172,0.2667)
    (1.0175,0.3000)
    (1.0235,0.3000)
    (1.0244,0.3333)
    (1.0270,0.3333)
    (1.0301,0.3667)
    (1.0321,0.3667)
    (1.0321,0.4000)
    (1.0340,0.4000)
    (1.0346,0.4333)
    (1.0394,0.4667)
    (1.0415,0.5000)
    (1.0771,0.5000)
    (1.1120,0.5333)
    (1.1345,0.5333)
    (1.1346,0.5667)
    (1.1366,0.6000)
    (1.1517,0.6000)
    (1.1517,0.6333)
    (1.1693,0.6667)
    (1.1700,0.6667)
    (1.1773,0.7000)
    (1.1789,0.7333)
    (1.1885,0.7333)
    (1.1885,0.7667)
    (1.2058,0.7667)
    (1.2058,0.8000)
    (1.2494,0.8000)
    (1.2494,0.8333)
    (1.25,1.0000)
  };
  \addlegendentry{T-D-NMDT-NC}
  \draw node[right,draw,align=left] {$L=1$\\};
  \end{axis}
\end{tikzpicture}
            \end{center}
        \end{minipage}
        \hfill
        \begin{minipage}{0.475\textwidth}
            \begin{center}
                 \begin{tikzpicture}
  \begin{axis}[const plot,
  cycle list={
  {blue,solid},
  {red!40!gray,dashed},
  {black,dotted},
  {brown,dashdotted},
  {green!40!gray!80!black,dashdotdotted},
  {magenta!80!black,densely dotted}},
    xmin=1, xmax=1.25,
    ymin=-0.003, ymax=1.003,
    ymajorgrids,
    ytick={0,0.2,0.4,0.6,0.8,1.0},
    xlabel={$\tau$},
    ylabel={$P(\tau)$},
,
    legend pos={south east},
    legend style={font=\tiny, fill opacity=0.7, draw=none},
    width=\textwidth
    ]
  \addplot+[mark=none, ultra thick, black, dashed] coordinates {
    (1.0000,0.5333)
    (1.0000,0.8333)
    (1.0001,0.8667)
    (1.0001,0.9333)
    (1.0415,0.9333)
    (1.0542,0.9667)
    (1.25,1.0000)
  };
  \addlegendentry{Gurobi}
  \addplot+[mark=none, thick, black, solid] coordinates {
    (1.0000,0.4667)
    (1.0000,0.6333)
    (1.0001,0.6333)
    (1.0001,0.7000)
    (1.0031,0.7000)
    (1.0032,0.7333)
    (1.0032,0.7667)
    (1.0041,0.7667)
    (1.0044,0.8000)
    (1.0053,0.8000)
    (1.0068,0.8333)
    (1.0070,0.8333)
    (1.0075,0.8667)
    (1.0100,0.8667)
    (1.0118,0.9000)
    (1.0237,0.9000)
    (1.0270,0.9333)
    (1.0308,0.9333)
    (1.0339,0.9667)
    (1.25,1.0000)
  };
  \addlegendentry{Gurobi-NC}
  \addplot+[mark=none, ultra thick, red!40!gray, dashed] coordinates {
    (1.0000,0.0333)
    (1.0000,0.0667)
    (1.0001,0.0667)
    (1.0007,0.1000)
    (1.0034,0.1000)
    (1.0038,0.1333)
    (1.0040,0.1667)
    (1.0086,0.1667)
    (1.0098,0.2000)
    (1.0118,0.2000)
    (1.0140,0.2333)
    (1.0141,0.2333)
    (1.0141,0.2667)
    (1.0142,0.2667)
    (1.0149,0.3000)
    (1.0158,0.3000)
    (1.0185,0.3333)
    (1.0221,0.3333)
    (1.0221,0.3667)
    (1.0271,0.3667)
    (1.0287,0.4000)
    (1.0308,0.4333)
    (1.0339,0.4333)
    (1.0372,0.4667)
    (1.0415,0.5000)
    (1.0697,0.5000)
    (1.0698,0.5333)
    (1.0812,0.5333)
    (1.0827,0.5667)
    (1.0861,0.5667)
    (1.0930,0.6000)
    (1.1075,0.6000)
    (1.1076,0.6333)
    (1.1158,0.6333)
    (1.1159,0.6667)
    (1.1366,0.6667)
    (1.1366,0.7000)
    (1.1468,0.7000)
    (1.1468,0.7333)
    (1.1693,0.7333)
    (1.1693,0.7667)
    (1.1750,0.8000)
    (1.1775,0.8000)
    (1.1775,0.8333)
    (1.1776,0.8333)
    (1.2131,0.8667)
    (1.2163,0.9000)
    (1.2494,0.9333)
    (1.25,1.0000)
  };
  \addlegendentry{HybS}
  \addplot+[mark=none, thick, red!40!gray, solid] coordinates {
    (1.0000,0.0667)
    (1.0008,0.0667)
    (1.0013,0.1000)
    (1.0034,0.1000)
    (1.0038,0.1333)
    (1.0040,0.1333)
    (1.0041,0.1667)
    (1.0098,0.1667)
    (1.0100,0.2000)
    (1.0140,0.2000)
    (1.0141,0.2333)
    (1.0149,0.2333)
    (1.0153,0.2667)
    (1.0154,0.3000)
    (1.0158,0.3000)
    (1.0185,0.3333)
    (1.0198,0.3333)
    (1.0221,0.3667)
    (1.0271,0.3667)
    (1.0287,0.4000)
    (1.0308,0.4000)
    (1.0308,0.4333)
    (1.0372,0.4333)
    (1.0372,0.4667)
    (1.0415,0.5000)
    (1.0698,0.5000)
    (1.0812,0.5333)
    (1.0827,0.5333)
    (1.0827,0.5667)
    (1.0930,0.5667)
    (1.0930,0.6000)
    (1.1012,0.6000)
    (1.1075,0.6333)
    (1.1076,0.6333)
    (1.1158,0.6667)
    (1.1270,0.6667)
    (1.1366,0.7000)
    (1.1436,0.7000)
    (1.1468,0.7333)
    (1.1688,0.7333)
    (1.1693,0.7667)
    (1.1750,0.8000)
    (1.1775,0.8000)
    (1.1775,0.8333)
    (1.2131,0.8333)
    (1.2131,0.8667)
    (1.2163,0.9000)
    (1.2494,0.9333)
    (1.25,1.0000)
  };
  \addlegendentry{HybS-NC}
  \addplot+[mark=none, ultra thick, green!40!gray, densely dashed] coordinates {
    (1.0000,0.0667)
    (1.0007,0.0667)
    (1.0008,0.1000)
    (1.0013,0.1000)
    (1.0013,0.1333)
    (1.0020,0.1333)
    (1.0020,0.1667)
    (1.0024,0.2000)
    (1.0033,0.2000)
    (1.0034,0.2333)
    (1.0044,0.2333)
    (1.0053,0.2667)
    (1.0068,0.2667)
    (1.0070,0.3000)
    (1.0075,0.3000)
    (1.0084,0.3333)
    (1.0085,0.3333)
    (1.0085,0.3667)
    (1.0140,0.3667)
    (1.0140,0.4000)
    (1.0141,0.4000)
    (1.0142,0.4333)
    (1.0185,0.4333)
    (1.0197,0.4667)
    (1.0237,0.4667)
    (1.0237,0.5000)
    (1.0372,0.5000)
    (1.0415,0.5333)
    (1.0542,0.5333)
    (1.0664,0.5667)
    (1.0697,0.6000)
    (1.0827,0.6000)
    (1.0835,0.6333)
    (1.0930,0.6333)
    (1.0994,0.6667)
    (1.1011,0.6667)
    (1.1012,0.7000)
    (1.1159,0.7000)
    (1.1270,0.7333)
    (1.1366,0.7333)
    (1.1436,0.7667)
    (1.1468,0.7667)
    (1.1540,0.8000)
    (1.1688,0.8333)
    (1.1693,0.8667)
    (1.1750,0.8667)
    (1.1775,0.9000)
    (1.2131,0.9000)
    (1.2163,0.9333)
    (1.25,1.0000)
  };
  \addlegendentry{T-D-NMDT}
  \addplot+[mark=none, thick, green!40!gray, solid] coordinates {
    (1.0000,0.0667)
    (1.0013,0.0667)
    (1.0013,0.1000)
    (1.0020,0.1333)
    (1.0024,0.1333)
    (1.0031,0.1667)
    (1.0032,0.1667)
    (1.0033,0.2000)
    (1.0034,0.2000)
    (1.0034,0.2333)
    (1.0070,0.2333)
    (1.0070,0.2667)
    (1.0084,0.2667)
    (1.0086,0.3333)
    (1.0140,0.3333)
    (1.0140,0.3667)
    (1.0154,0.3667)
    (1.0158,0.4000)
    (1.0197,0.4000)
    (1.0198,0.4333)
    (1.0221,0.4333)
    (1.0237,0.4667)
    (1.0270,0.4667)
    (1.0271,0.5000)
    (1.0372,0.5000)
    (1.0415,0.5333)
    (1.0664,0.5333)
    (1.0664,0.5667)
    (1.0698,0.5667)
    (1.0698,0.6000)
    (1.0835,0.6000)
    (1.0861,0.6333)
    (1.0930,0.6333)
    (1.0994,0.6667)
    (1.1011,0.7000)
    (1.1270,0.7000)
    (1.1270,0.7333)
    (1.1436,0.7333)
    (1.1436,0.7667)
    (1.1468,0.7667)
    (1.1540,0.8000)
    (1.1688,0.8333)
    (1.1693,0.8667)
    (1.1775,0.8667)
    (1.1776,0.9000)
    (1.2131,0.9000)
    (1.2163,0.9333)
    (1.25,1.0000)
  };
  \addlegendentry{T-D-NMDT-NC}
  \draw node[right,draw,align=left] {$L=2$\\};
  \end{axis}
\end{tikzpicture}
            \end{center}
        \end{minipage}
    \end{center}
    \hfill
    \begin{center}
        \begin{minipage}{0.475\textwidth}
            \begin{center}
                 \begin{tikzpicture}
  \begin{axis}[const plot,
  cycle list={
  {blue,solid},
  {red!40!gray,dashed},
  {black,dotted},
  {brown,dashdotted},
  {green!40!gray!80!black,dashdotdotted},
  {magenta!80!black,densely dotted}},
    xmin=1, xmax=1.25,
    ymin=-0.003, ymax=1.003,
    ymajorgrids,
    ytick={0,0.2,0.4,0.6,0.8,1.0},
    xlabel={$\tau$},
    ylabel={$P(\tau)$},
,
    legend pos={south east},
    legend style={font=\tiny},
    width=\textwidth
    ]
  \addplot+[mark=none, ultra thick, black, dashed] coordinates {
    (1.0000,0.3667)
    (1.0000,0.6667)
    (1.0001,0.7000)
    (1.0001,0.7667)
    (1.0032,0.7667)
    (1.0033,0.8000)
    (1.0037,0.8000)
    (1.0037,0.8333)
    (1.0044,0.8333)
    (1.0048,0.8667)
    (1.0123,0.8667)
    (1.0186,0.9000)
    (1.0535,0.9000)
    (1.0542,0.9333)
    (1.0579,0.9667)
    (1.1726,0.9667)
    (1.25,1.0000)
  };
  \addlegendentry{Gurobi}
  \addplot+[mark=none, thick, black, solid] coordinates {
    (1.0000,0.4333)
    (1.0000,0.6000)
    (1.0001,0.6000)
    (1.0001,0.6667)
    (1.0032,0.6667)
    (1.0033,0.7000)
    (1.0037,0.7000)
    (1.0044,0.7333)
    (1.0067,0.7333)
    (1.0068,0.7667)
    (1.0075,0.8000)
    (1.0078,0.8000)
    (1.0079,0.8333)
    (1.0115,0.8333)
    (1.0118,0.8667)
    (1.0186,0.8667)
    (1.0219,0.9000)
    (1.0351,0.9000)
    (1.0377,0.9333)
    (1.0579,0.9333)
    (1.0864,0.9667)
    (1.25,1.0000)
  };
  \addlegendentry{Gurobi-NC}
  \addplot+[mark=none, ultra thick, red!40!gray, dashed] coordinates {
    (1.0000,0.2000)
    (1.0000,0.2333)
    (1.0002,0.2333)
    (1.0002,0.2667)
    (1.0005,0.2667)
    (1.0005,0.3333)
    (1.0007,0.3333)
    (1.0008,0.3667)
    (1.0010,0.3667)
    (1.0010,0.4000)
    (1.0011,0.4000)
    (1.0012,0.4333)
    (1.0016,0.4333)
    (1.0016,0.4667)
    (1.0022,0.4667)
    (1.0030,0.5000)
    (1.0031,0.5333)
    (1.0031,0.5667)
    (1.0032,0.6000)
    (1.0033,0.6000)
    (1.0036,0.6333)
    (1.0054,0.6333)
    (1.0060,0.6667)
    (1.0064,0.6667)
    (1.0065,0.7000)
    (1.0075,0.7000)
    (1.0078,0.7333)
    (1.0079,0.7333)
    (1.0083,0.7667)
    (1.0099,0.7667)
    (1.0113,0.8000)
    (1.0115,0.8333)
    (1.0377,0.8333)
    (1.0410,0.8667)
    (1.0864,0.8667)
    (1.1005,0.9000)
    (1.1154,0.9000)
    (1.1154,0.9333)
    (1.1159,0.9667)
    (1.1726,0.9667)
    (1.1726,1.0000)
    (1.25,1.0000)
  };
  \addlegendentry{HybS}
  \addplot+[mark=none, thick, red!40!gray, solid] coordinates {
    (1.0000,0.1000)
    (1.0000,0.1667)
    (1.0001,0.1667)
    (1.0001,0.2000)
    (1.0002,0.2000)
    (1.0002,0.2333)
    (1.0005,0.2333)
    (1.0005,0.2667)
    (1.0007,0.3000)
    (1.0010,0.3000)
    (1.0010,0.3333)
    (1.0011,0.3333)
    (1.0012,0.3667)
    (1.0016,0.3667)
    (1.0016,0.4000)
    (1.0020,0.4000)
    (1.0022,0.4333)
    (1.0030,0.4333)
    (1.0030,0.4667)
    (1.0032,0.4667)
    (1.0032,0.5000)
    (1.0033,0.5000)
    (1.0033,0.5333)
    (1.0036,0.5333)
    (1.0037,0.5667)
    (1.0054,0.5667)
    (1.0060,0.6000)
    (1.0078,0.6000)
    (1.0078,0.6333)
    (1.0083,0.6333)
    (1.0083,0.6667)
    (1.0084,0.6667)
    (1.0096,0.7000)
    (1.0097,0.7333)
    (1.0099,0.7667)
    (1.0118,0.7667)
    (1.0123,0.8000)
    (1.0219,0.8000)
    (1.0238,0.8333)
    (1.0241,0.8333)
    (1.0262,0.8667)
    (1.0410,0.8667)
    (1.0410,0.9000)
    (1.1154,0.9000)
    (1.1154,0.9333)
    (1.1159,0.9667)
    (1.1726,0.9667)
    (1.1726,1.0000)
    (1.25,1.0000)
  };
  \addlegendentry{HybS-NC}
  \addplot+[mark=none, ultra thick, green!40!gray, densely dashed] coordinates {
    (1.0000,0.0333)
    (1.0000,0.1333)
    (1.0001,0.1333)
    (1.0001,0.2000)
    (1.0002,0.2000)
    (1.0002,0.2333)
    (1.0003,0.2667)
    (1.0003,0.3000)
    (1.0004,0.3333)
    (1.0005,0.3667)
    (1.0007,0.3667)
    (1.0007,0.4000)
    (1.0008,0.4000)
    (1.0009,0.4333)
    (1.0010,0.4333)
    (1.0011,0.4667)
    (1.0011,0.5000)
    (1.0012,0.5000)
    (1.0014,0.5333)
    (1.0016,0.5667)
    (1.0019,0.6000)
    (1.0019,0.6333)
    (1.0048,0.6333)
    (1.0050,0.6667)
    (1.0051,0.7000)
    (1.0060,0.7000)
    (1.0064,0.7333)
    (1.0238,0.7333)
    (1.0241,0.7667)
    (1.0262,0.7667)
    (1.0275,0.8000)
    (1.0304,0.8000)
    (1.0351,0.8333)
    (1.1005,0.8333)
    (1.1036,0.8667)
    (1.1110,0.9000)
    (1.1154,0.9333)
    (1.1295,0.9333)
    (1.1725,0.9667)
    (1.25,1.0000)
  };
  \addlegendentry{T-D-NMDT}
  \addplot+[mark=none, thick, green!40!gray, solid] coordinates {
    (1.0000,0.0333)
    (1.0000,0.1333)
    (1.0001,0.1333)
    (1.0001,0.2000)
    (1.0002,0.2333)
    (1.0002,0.2667)
    (1.0003,0.2667)
    (1.0003,0.3333)
    (1.0005,0.3333)
    (1.0005,0.3667)
    (1.0007,0.3667)
    (1.0007,0.4000)
    (1.0009,0.4000)
    (1.0010,0.4333)
    (1.0011,0.4333)
    (1.0011,0.4667)
    (1.0014,0.4667)
    (1.0016,0.5000)
    (1.0016,0.5333)
    (1.0019,0.5667)
    (1.0020,0.6000)
    (1.0051,0.6000)
    (1.0054,0.6333)
    (1.0065,0.6333)
    (1.0067,0.6667)
    (1.0083,0.6667)
    (1.0084,0.7000)
    (1.0275,0.7000)
    (1.0304,0.7333)
    (1.0410,0.7333)
    (1.0488,0.7667)
    (1.0491,0.8000)
    (1.0535,0.8333)
    (1.1110,0.8333)
    (1.1154,0.8667)
    (1.1159,0.8667)
    (1.1205,0.9000)
    (1.1295,0.9333)
    (1.1725,0.9333)
    (1.1726,0.9667)
    (1.25, 1.0000)
  };
  \addlegendentry{T-D-NMDT-NC}
  \draw node[right,draw,align=left] {$L=4$\\};
  \end{axis}
\end{tikzpicture}
            \end{center}
        \end{minipage}
        \hfill
        \begin{minipage}{0.475\textwidth}
            \begin{center}
                 \begin{tikzpicture}
  \begin{axis}[const plot,
  cycle list={
  {blue,solid},
  {red!40!gray,dashed},
  {black,dotted},
  {brown,dashdotted},
  {green!40!gray!80!black,dashdotdotted},
  {magenta!80!black,densely dotted}},
    xmin=1, xmax=1.25,
    ymin=-0.003, ymax=1.003,
    ymajorgrids,
    ytick={0,0.2,0.4,0.6,0.8,1.0},
    xlabel={$\tau$},
    ylabel={$P(\tau)$},
,
    legend pos={south east},
    legend style={font=\tiny},
    width=\textwidth
    ]
  \addplot+[mark=none, ultra thick, black, dashed] coordinates {
    (1.0000,0.3333)
    (1.0000,0.6667)
    (1.0001,0.7000)
    (1.0001,0.7667)
    (1.0004,0.7667)
    (1.0007,0.8000)
    (1.0034,0.8000)
    (1.0035,0.8333)
    (1.0050,0.8333)
    (1.0059,0.8667)
    (1.0095,0.8667)
    (1.0096,0.9000)
    (1.0190,0.9000)
    (1.0190,0.9333)
    (1.0499,0.9333)
    (1.0503,0.9667)
    (1.0630,0.9667)
    (1.0635,1.0000)
    (1.25,1.0000)
  };
  \addlegendentry{Gurobi}
  \addplot+[mark=none, thick, black, solid] coordinates {
    (1.0000,0.4000)
    (1.0000,0.6333)
    (1.0001,0.6333)
    (1.0001,0.6667)
    (1.0042,0.6667)
    (1.0044,0.7000)
    (1.0066,0.7000)
    (1.0066,0.7333)
    (1.0075,0.7667)
    (1.0075,0.8000)
    (1.0087,0.8000)
    (1.0088,0.8333)
    (1.0095,0.8333)
    (1.0096,0.8667)
    (1.0113,0.8667)
    (1.0118,0.9000)
    (1.0207,0.9000)
    (1.0222,0.9333)
    (1.0351,0.9333)
    (1.0399,0.9667)
    (1.0635,0.9667)
    (1.0787,1.0000)
    (1.25,1.0000)
  };
  \addlegendentry{Gurobi-NC}
  \addplot+[mark=none, ultra thick, red!40!gray, dashed] coordinates {
    (1.0000,0.1333)
    (1.0000,0.2333)
    (1.0001,0.2333)
    (1.0001,0.4000)
    (1.0002,0.4000)
    (1.0002,0.5000)
    (1.0004,0.5000)
    (1.0004,0.5333)
    (1.0007,0.5333)
    (1.0008,0.5667)
    (1.0013,0.5667)
    (1.0018,0.6000)
    (1.0026,0.6000)
    (1.0034,0.6333)
    (1.0059,0.6333)
    (1.0062,0.6667)
    (1.0063,0.6667)
    (1.0065,0.7000)
    (1.0082,0.7000)
    (1.0084,0.7333)
    (1.0088,0.7333)
    (1.0094,0.7667)
    (1.0141,0.7667)
    (1.0144,0.8000)
    (1.0157,0.8333)
    (1.0190,0.8333)
    (1.0207,0.8667)
    (1.0222,0.8667)
    (1.0244,0.9000)
    (1.0964,0.9000)
    (1.1033,0.9333)
    (1.1343,0.9333)
    (1.1409,0.9667)
    (1.2173,1.0000)
    (1.25,1.0000)
  };
  \addlegendentry{HybS}
  \addplot+[mark=none, thick, red!40!gray, solid] coordinates {
    (1.0000,0.1667)
    (1.0000,0.3000)
    (1.0001,0.3000)
    (1.0001,0.5000)
    (1.0002,0.5000)
    (1.0002,0.6667)
    (1.0003,0.6667)
    (1.0003,0.7000)
    (1.0045,0.7000)
    (1.0050,0.7333)
    (1.0065,0.7333)
    (1.0066,0.7667)
    (1.0096,0.7667)
    (1.0101,0.8000)
    (1.0112,0.8333)
    (1.0157,0.8333)
    (1.0165,0.8667)
    (1.0522,0.8667)
    (1.0530,0.9000)
    (1.0964,0.9000)
    (1.1033,0.9333)
    (1.1045,0.9667)
    (1.2173,0.9667)
    (1.25,1.0000)
  };
  \addlegendentry{HybS-NC}
  \addplot+[mark=none, ultra thick, green!40!gray, densely dotted] coordinates {
    (1.0000,0.0333)
    (1.0000,0.2333)
    (1.0001,0.2333)
    (1.0001,0.4667)
    (1.0002,0.5000)
    (1.0003,0.5000)
    (1.0004,0.5333)
    (1.0008,0.5333)
    (1.0013,0.5667)
    (1.0018,0.5667)
    (1.0026,0.6000)
    (1.0035,0.6000)
    (1.0037,0.6333)
    (1.0044,0.6333)
    (1.0045,0.6667)
    (1.0075,0.6667)
    (1.0082,0.7000)
    (1.0084,0.7000)
    (1.0087,0.7333)
    (1.0094,0.7333)
    (1.0095,0.7667)
    (1.0112,0.7667)
    (1.0113,0.8000)
    (1.0244,0.8000)
    (1.0279,0.8333)
    (1.0298,0.8667)
    (1.0303,0.9000)
    (1.0964,0.9000)
    (1.1033,0.9333)
    (1.1045,0.9333)
    (1.1143,0.9667)
    (1.1154,0.9667)
    (1.1270,1.0000)
    (1.25,1.0000)
  };
  \addlegendentry{T-D-NMDT}
  \addplot+[mark=none, thick, green!40!gray, solid] coordinates {
    (1.0000,0.1000)
    (1.0000,0.3333)
    (1.0001,0.3333)
    (1.0001,0.4667)
    (1.0002,0.4667)
    (1.0002,0.5000)
    (1.0003,0.5333)
    (1.0037,0.5333)
    (1.0042,0.5667)
    (1.0062,0.5667)
    (1.0063,0.6000)
    (1.0084,0.6000)
    (1.0087,0.6333)
    (1.0118,0.6333)
    (1.0141,0.6667)
    (1.0165,0.6667)
    (1.0190,0.7000)
    (1.0303,0.7000)
    (1.0351,0.7333)
    (1.0399,0.7333)
    (1.0499,0.7667)
    (1.0503,0.7667)
    (1.0505,0.8000)
    (1.0522,0.8333)
    (1.0530,0.8333)
    (1.0630,0.8667)
    (1.0787,0.8667)
    (1.0964,0.9000)
    (1.1033,0.9333)
    (1.1143,0.9333)
    (1.1154,0.9667)
    (1.1270,0.9667)
    (1.1343,1.0000)
    (1.25,1.0000)
  };
  \addlegendentry{T-D-NMDT-NC}
  \draw node[right,draw,align=left] {$L=6$\\};
  \end{axis}
\end{tikzpicture}
            \end{center}
        \end{minipage}
    \end{center}
    \hfill
    \caption{Performance profiles on dual bounds of best MIP relaxation compared to Gurobi as MIQCQP solver, with and without cuts, on sparse instances. }\label{gurobi_sparse}
  \vspace*{\floatsep}%
    \begin{center}
        \begin{minipage}{0.475\textwidth}
            \begin{center}
                 \begin{tikzpicture}
  \begin{axis}[const plot,
  cycle list={
  {blue,solid},
  {red!40!gray,dashed},
  {black,dotted},
  {brown,dashdotted},
  {green!40!gray!80!black,dashdotdotted},
  {magenta!80!black,densely dotted}},
    xmin=1, xmax=1.25,
    ymin=-0.003, ymax=1.003,
    ymajorgrids,
    ytick={0,0.2,0.4,0.6,0.8,1.0},
    xlabel={$\tau$},
    ylabel={$P(\tau)$},
,
    legend pos={south east},
    legend style={font=\tiny, fill opacity=0.7, draw=none},
    width=\textwidth
    ]
  \addplot+[mark=none, ultra thick, black, dashed] coordinates {
    (1.0000,0.6000)
    (1.0000,0.9667)
    (1.0001,0.9667)
    (1.0001,1.0000)
    (1.25,1.0000)
  };
  \addlegendentry{Gurobi}
  \addplot+[mark=none, thick, black, solid] coordinates {
    (1.0000,0.2333)
    (1.0000,0.6333)
    (1.0001,0.6667)
    (1.0719,0.6667)
    (1.0789,0.7000)
    (1.2067,0.7000)
    (1.2248,0.7333)
    (1.25,1.0000)
  };
  \addlegendentry{Gurobi-NC}
  \addplot+[mark=none, ultra thick, red!40!gray, dashed] coordinates {
    (1.0000,0.2333)
    (1.0061,0.2333)
    (1.0122,0.2667)
    (1.0329,0.3000)
    (1.0340,0.3000)
    (1.0351,0.3333)
    (1.0374,0.3333)
    (1.0555,0.3667)
    (1.0645,0.4000)
    (1.0993,0.4000)
    (1.1068,0.4333)
    (1.1186,0.4333)
    (1.1318,0.4667)
    (1.2052,0.4667)
    (1.2067,0.5000)
    (1.25,1.0000)
  };
  \addlegendentry{HybS}
  \addplot+[mark=none, thick, red!40!gray, solid] coordinates {
    (1.0000,0.2333)
    (1.0061,0.2333)
    (1.0122,0.2667)
    (1.0329,0.3000)
    (1.0340,0.3000)
    (1.0351,0.3333)
    (1.0555,0.3333)
    (1.0555,0.3667)
    (1.0645,0.4000)
    (1.2248,0.4000)
    (1.2371,0.4333)
    (1.2426,0.4333)
    (1.25,1.0000)
  };
  \addlegendentry{HybS-NC}
  \addplot+[mark=none, ultra thick, green!40!gray, densely dotted] coordinates {
    (1.0000,0.3333)
    (1.0000,0.3667)
    (1.0001,0.3667)
    (1.0001,0.4000)
    (1.0050,0.4333)
    (1.0061,0.4667)
    (1.0329,0.4667)
    (1.0340,0.5000)
    (1.0374,0.5000)
    (1.0374,0.5333)
    (1.0645,0.5333)
    (1.0719,0.5667)
    (1.0789,0.5667)
    (1.0885,0.6000)
    (1.0888,0.6333)
    (1.0993,0.6667)
    (1.1489,0.6667)
    (1.1526,0.7000)
    (1.2052,0.7333)
    (1.2371,0.7333)
    (1.2426,0.7667)
    (1.25,1.0000)
  };
  \addlegendentry{T-D-NMDT}
  \addplot+[mark=none, thick, green!40!gray, solid] coordinates {
    (1.0000,0.3000)
    (1.0001,0.3000)
    (1.0001,0.3333)
    (1.0050,0.3333)
    (1.0061,0.3667)
    (1.0351,0.3667)
    (1.0374,0.4000)
    (1.1068,0.4000)
    (1.1079,0.4333)
    (1.1186,0.4667)
    (1.1318,0.4667)
    (1.1489,0.5000)
    (1.25,1.0000)
  };
  \addlegendentry{T-D-NMDT-NC}
  \draw node[right,draw,align=left] {$L=1$\\};
  \end{axis}
\end{tikzpicture}
            \end{center}
        \end{minipage}
        \hfill
        \begin{minipage}{0.475\textwidth}
            \begin{center}
                 \begin{tikzpicture}
  \begin{axis}[const plot,
  cycle list={
  {blue,solid},
  {red!40!gray,dashed},
  {black,dotted},
  {brown,dashdotted},
  {green!40!gray!80!black,dashdotdotted},
  {magenta!80!black,densely dotted}},
    xmin=1, xmax=1.25,
    ymin=-0.003, ymax=1.003,
    ymajorgrids,
    ytick={0,0.2,0.4,0.6,0.8,1.0},
    xlabel={$\tau$},
    ylabel={$P(\tau)$},
,
    legend pos={south east},
    legend style={font=\tiny, fill opacity=0.7, draw=none},
    width=\textwidth
    ]
  \addplot+[mark=none, ultra thick, black, dashed] coordinates {
    (1.0000,0.5333)
    (1.0000,0.9667)
    (1.0001,0.9667)
    (1.0001,1.0000)
    (1.25,1.0000)
  };
  \addlegendentry{Gurobi}
  \addplot+[mark=none, thick, black, solid] coordinates {
    (1.0000,0.2000)
    (1.0000,0.6333)
    (1.0001,0.6667)
    (1.0196,0.6667)
    (1.0789,0.7000)
    (1.2227,0.7000)
    (1.2248,0.7333)
    (1.25,1.0000)
  };
  \addlegendentry{Gurobi-NC}
  \addplot+[mark=none, ultra thick, red!40!gray, dashed] coordinates {
    (1.0000,0.3000)
    (1.0000,0.3333)
    (1.0001,0.3333)
    (1.0036,0.3667)
    (1.0059,0.3667)
    (1.0196,0.4000)
    (1.0789,0.4000)
    (1.1241,0.4333)
    (1.1356,0.4333)
    (1.1568,0.4667)
    (1.2218,0.4667)
    (1.2227,0.5000)
    (1.25,1.0000)
  };
  \addlegendentry{HybS}
  \addplot+[mark=none, thick, red!40!gray, solid] coordinates {
    (1.0000,0.3000)
    (1.0000,0.3333)
    (1.0001,0.3333)
    (1.0036,0.3667)
    (1.0196,0.3667)
    (1.0196,0.4000)
    (1.2248,0.4000)
    (1.2458,0.4333)
    (1.2477,0.4333)
    (1.25,1.0000)
  };
  \addlegendentry{HybS-NC}
  \addplot+[mark=none, ultra thick, green!40!gray, densely dotted] coordinates {
    (1.0000,0.3667)
    (1.0000,0.4333)
    (1.0036,0.4333)
    (1.0059,0.4667)
    (1.1241,0.4667)
    (1.1332,0.5000)
    (1.1356,0.5333)
    (1.1568,0.5333)
    (1.1709,0.5667)
    (1.1958,0.5667)
    (1.2016,0.6000)
    (1.2018,0.6333)
    (1.2218,0.6667)
    (1.2477,0.6667)
    (1.2477,0.7333)
    (1.25, 1.0000)
  };
  \addlegendentry{T-D-NMDT}
  \addplot+[mark=none, thick, green!40!gray, solid] coordinates {
    (1.0000,0.3333)
    (1.0000,0.3667)
    (1.0059,0.3667)
    (1.0059,0.4000)
    (1.1709,0.4000)
    (1.1884,0.4333)
    (1.1936,0.4667)
    (1.1958,0.5000)
    (1.2458,0.5000)
    (1.2477,0.5333)
    (1.2477,0.5667)
    (1.25,1.0000)
  };
  \addlegendentry{T-D-NMDT-NC}
  \draw node[right,draw,align=left] {$L=2$\\};
  \end{axis}
\end{tikzpicture}
            \end{center}
        \end{minipage}
    \end{center}
    \hfill
    \begin{center}
        \begin{minipage}{0.475\textwidth}
            \begin{center}
                 \begin{tikzpicture}
  \begin{axis}[const plot,
  cycle list={
  {blue,solid},
  {red!40!gray,dashed},
  {black,dotted},
  {brown,dashdotted},
  {green!40!gray!80!black,dashdotdotted},
  {magenta!80!black,densely dotted}},
    xmin=1, xmax=1.25,
    ymin=-0.003, ymax=1.003,
    ymajorgrids,
    ytick={0,0.2,0.4,0.6,0.8,1.0},
    xlabel={$\tau$},
    ylabel={$P(\tau)$},
,
    legend pos={south east},
    legend style={font=\tiny, fill opacity=0.7, draw=none},
    width=\textwidth
    ]
  \addplot+[mark=none, ultra thick, black, dashed] coordinates {
    (1.0000,0.5333)
    (1.0000,0.9667)
    (1.0001,0.9667)
    (1.0001,1.0000)
    (1.25,1.0000)
  };
  \addlegendentry{Gurobi}
  \addplot+[mark=none, thick, black, solid] coordinates {
    (1.0000,0.2333)
    (1.0000,0.6333)
    (1.0001,0.6667)
    (1.0324,0.6667)
    (1.0789,0.7000)
    (1.2237,0.7000)
    (1.2248,0.7333)
    (1.25,1.0000)
  };
  \addlegendentry{Gurobi-NC}
  \addplot+[mark=none,ultra thick, red!40!gray, dashed] coordinates {
    (1.0000,0.3000)
    (1.0000,0.3667)
    (1.0010,0.3667)
    (1.0010,0.4000)
    (1.0107,0.4000)
    (1.0311,0.4333)
    (1.0324,0.4333)
    (1.0324,0.4667)
    (1.0926,0.4667)
    (1.1370,0.5000)
    (1.1504,0.5000)
    (1.1666,0.5333)
    (1.25,1.0000)
  };
  \addlegendentry{HybS}
  \addplot+[mark=none, thick, red!40!gray, solid] coordinates {
    (1.0000,0.3000)
    (1.0000,0.3667)
    (1.0002,0.3667)
    (1.0010,0.4000)
    (1.0311,0.4000)
    (1.0311,0.4333)
    (1.0324,0.4667)
    (1.2414,0.4667)
    (1.25,1.0000)
  };
  \addlegendentry{HybS-NC}
  \addplot+[mark=none,ultra thick, green!40!gray, densely dotted] coordinates {
    (1.0000,0.3000)
    (1.0000,0.3667)
    (1.0002,0.3667)
    (1.0002,0.4000)
    (1.0095,0.4000)
    (1.0095,0.4333)
    (1.0107,0.4667)
    (1.0789,0.4667)
    (1.0926,0.5000)
    (1.1370,0.5000)
    (1.1504,0.5333)
    (1.25,1.0000)
  };
  \addlegendentry{T-D-NMDT}
  \addplot+[mark=none, thick, green!40!gray, solid] coordinates {
    (1.0000,0.3000)
    (1.0000,0.3667)
    (1.0001,0.3667)
    (1.0002,0.4000)
    (1.0010,0.4000)
    (1.0095,0.4333)
    (1.0107,0.4667)
    (1.1666,0.4667)
    (1.2237,0.5000)
    (1.2248,0.5000)
    (1.2414,0.5333)
    (1.25,1.0000)
  };
  \addlegendentry{T-D-NMDT-NC}
  \draw node[right,draw,align=left] {$L=4$\\};
  \end{axis}
\end{tikzpicture}
            \end{center}
        \end{minipage}
        \hfill
        \begin{minipage}{0.475\textwidth}
            \begin{center}
                 \begin{tikzpicture}
  \begin{axis}[const plot,
  cycle list={
  {blue,solid},
  {red!40!gray,dashed},
  {black,dotted},
  {brown,dashdotted},
  {green!40!gray!80!black,dashdotdotted},
  {magenta!80!black,densely dotted}},
    xmin=1, xmax=1.25,
    ymin=-0.003, ymax=1.003,
    ymajorgrids,
    ytick={0,0.2,0.4,0.6,0.8,1.0},
    xlabel={$\tau$},
    ylabel={$P(\tau)$},
,
    legend pos={south east},
    legend style={font=\tiny, fill opacity=0.7, draw=none},
    width=\textwidth
    ]
  \addplot+[mark=none,ultra thick,black, dashed] coordinates {
    (1.0000,0.5333)
    (1.0000,0.9667)
    (1.0001,0.9667)
    (1.0001,1.0000)
    (1.25,1.0000)
  };
  \addlegendentry{Gurobi}
  \addplot+[mark=none, thick, black, solid] coordinates {
    (1.0000,0.2333)
    (1.0000,0.6333)
    (1.0001,0.6667)
    (1.0392,0.6667)
    (1.0789,0.7000)
    (1.1906,0.7000)
    (1.2248,0.7333)
    (1.25,1.0000)
  };
  \addlegendentry{Gurobi-NC}
  \addplot+[mark=none,ultra thick, red!40!gray, dashed] coordinates {
    (1.0000,0.3333)
    (1.0000,0.3667)
    (1.0001,0.3667)
    (1.0001,0.4000)
    (1.0006,0.4000)
    (1.0007,0.4333)
    (1.1540,0.4333)
    (1.1729,0.4667)
    (1.2248,0.4667)
    (1.2300,0.5000)
    (1.25,1.0000)
  };
  \addlegendentry{HybS}
  \addplot+[mark=none, thick, red!40!gray, solid] coordinates {
    (1.0000,0.2667)
    (1.0000,0.4000)
    (1.0007,0.4000)
    (1.0007,0.4333)
    (1.1793,0.4333)
    (1.1827,0.4667)
    (11.2937,1.0000)
  };
  \addlegendentry{HybS-NC}
  \addplot+[mark=none,ultra thick, green!40!gray, densely dotted] coordinates {
    (1.0000,0.0667)
    (1.0000,0.3667)
    (1.0001,0.3667)
    (1.0005,0.4000)
    (1.0006,0.4333)
    (1.0007,0.4333)
    (1.0080,0.4667)
    (1.0789,0.4667)
    (1.1540,0.5000)
    (1.1827,0.5000)
    (1.1906,0.5333)
    (1.25,1.0000)
  };
  \addlegendentry{T-D-NMDT}
  \addplot+[mark=none, thick, green!40!gray, solid] coordinates {
    (1.0000,0.1000)
    (1.0000,0.3667)
    (1.0080,0.3667)
    (1.0149,0.4000)
    (1.0392,0.4333)
    (1.1729,0.4333)
    (1.1793,0.4667)
    (1.2300,0.4667)
    (1.2448,0.5000)
    (1.25,1.0000)
  };
  \addlegendentry{T-D-NMDT-NC}
  \draw node[right,draw,align=left] {$L=6$\\};
  \end{axis}
\end{tikzpicture}
            \end{center}
        \end{minipage}
    \end{center}
    \hfill
    \caption{Performance profiles on dual bounds of the best MIP relaxation compared to Gurobi as MIQCQP solver, with and without cuts, on dense instances. }\label{gurobi_dense}
\end{figure}

\begin{table}
    \caption{Number of feasible solutions found with different relative optimality gaps. The first number corresponds to a gap of less than 0.01\%, the second to a gap of less than 1\% and the third number indicates the number of feasible solutions.}
    \label{table_feasible_gurobi_all}
    \centering
    \begin{tabular} {l c  c  c  c c c}
        \toprule
        {\qquad} & {    \HybS    } & {      \HybS-NC      } & {  T-\DNMDT  } & {  T-\DNMDT-NC  } & {  Gurobi  } & {  Gurobi-NC  } \\ 
        \midrule
        L1 & 31/33/40 & 28/31/40 & 29/33/40 & 31/34/42 & \textbf{50}/\textbf{50}/\textbf{57} & 46/49/56 \\
        L2 & 32/37/44 & 31/36/41 & 34/37/42 & 35/41/44 & \textbf{50}/\textbf{50}/\textbf{57} & 46/49/56 \\
        L4 & 41/44/50 & 40/45/53 & 45/47/51 & 40/45/50 & \textbf{50}/\textbf{50}/\textbf{57} & 46/49/56 \\
        L6 & 40/43/51 & 43/48/50 & 46/47/50 & 40/46/49 & \textbf{50}/\textbf{50}/\textbf{57} & 46/49/56 \\
        \bottomrule
    \end{tabular}
\end{table}

\newpage
\section{Conclusion}

We introduced an enhanced \emph{mixed-integer programming} (MIP) relaxation technique for \non convex \emph{mixed-integer quadratically constrained quadratic programs} (MIQCQP), called \emph{doubly discretized normalized multiparametric disaggregation technique} (\DNMDT).
We showed that it has clear theoretical advantages over its predecessor \NMDT, \ie it requires a significantly lower number of binary variables to achieve the same accuracy. 
In addition, we combined both, \DNMDT and \NMDT, with the \emph{sawtooth epigraph relaxation} from Part I~\cite{Part_I} to further strengthen the relaxations for univariate quadratic terms.

In a two-part computational study, we first compared \DNMDT to \NMDT. 
We showed that \DNMDT determines far better dual bounds than \NMDT and also has shorter run times.
Furthermore, we were able to show that our tightening in both methods led to better dual bounds while simultaneously shortening the computation time.
In the second part of the computational study, we compared
the \emph{tightened D-NMDT} (T-\DNMDT) against
\emph{Hybrid Separable} (\HybS), the best-performing MIP relaxation from Part I.
We showed that HybS does perform slightly better in terms of dual bounds.
However, both new methods were able to find high-quality solutions to the original quadratic problems when used in conjunction with a primal solution callback function and a local \non linear programming solver.
Furthermore, we showed that they both method can partially compete with the state-of-the-art MIQCQP solver Gurobi.

Finally, we gave some indications on how to further improve the new approaches.
Two of the most promising directions in this context are employing adaptivity and adding MIQCQP-specific cuts that are valid but not recognized by the MIP solvers.  This is the subject of future work.

\section*{Data availability statement} 
The boxQP instances are publicly available at
\href{https://github.com/joehuchette/quadratic-relaxation-experiments}{https://github.com/joehuchette/\\quadratic-relaxation-experiments}.
The ACOPF instances are publicly available at \href{https://github.com/robburlacu/acopflib}{https://github.com/robburlacu/acopflib}. 
The QPLIB instances are publicly available at \href{https://qplib.zib.de/}{https://qplib.zib.de/}.

\section*{Conflict of interest} 
The authors declare that they have no confict of interest.
\bibliographystyle{plain}  
\bibliography{references}


\appendix

\section{Detailed Derivation of the MIP Relaxation D-NMDT}
\label{ssec:derivations}
For the derivation of the MIP relaxation \DNMDT for $ \gra_{[0, 1]^2}(xy) $, we first define 
\begin{equation}
    \begin{array}{rll}
         x &= \dsum_{j = 1}^L 2^{-j}\beta^x_j + \Del x, \quad
         y = \dsum_{j = 1}^L 2^{-j}\beta^y_j + \Del y,\\
         \Del x &\in [0, 2^{-L}],\, \Del y \in [0, 2^{-L}],\,
         \bm \beta^x \in \{0, 1\}^L,\, \bm \beta^y \in \{0, 1\}^L.
    \end{array}
\end{equation}
Then we use the NMDT representation~\eqref{eq:NMDT-xy-exact},
expand the~$ \Del x y $-term and obtain
\begin{equation*}
    \begin{array}{rll}
         \z = xy &= y \lrp{\dsum_{j = 1}^L 2^{-j}\beta^x_j + \Del x} \\
           &= \dsum_{j = 1}^L 2^{-j} \beta^x_j y + y \Del x\\
           &= \dsum_{j = 1}^L 2^{-j} \beta^x_j y + \Del x \lrp{\dsum_{j = 1}^L 2^{-j}\beta^y_j + \Del y}\\
           &= \dsum_{j = 1}^L 2^{-j} (\beta^x_j y + \beta^y_j \Del x) + \Del x \Del y.\\
         \end{array}
\end{equation*}
Alternatively, if we discretize $y$ first, then expand the term~$ \Del y x $, we obtain
\begin{equation*}
    \begin{array}{rll}
         \z  = \dsum_{j=1}^L 2^{-j} (\beta^y_j x + \beta^x_j \Del y) + \Del x \Del y.\\
    \end{array}
\end{equation*}
Finally, to balance between the two formulations, we observe for any $ \lambda \in [0, 1] $ that
\begin{equation*}
    \begin{array}{rll}
         \z &= xy = \lambda xy + (1 - \lambda) xy\\
           &= \lambda \lrp{\dsum_{j = 1}^L 2^{-j} (\beta^y_j x + \beta^x_j \Del y) + \Del x \Del y}\\
           &+ (1 - \lambda) \lrp{\dsum_{j = 1}^L 2^{-j} (\beta^x_j y + \beta^y_j \Del x) + \Del x \Del y}\\
           &= \dsum_{j = 1}^L 2^{-j}[\beta^y_j ((1-\lambda) \Del x + \lambda x) + \beta^x_j (\lambda \Del y + (1 - \lambda) y)] + \Del x \Del y
         \end{array}
\end{equation*}
holds.
This yields
\begin{equation}
    \label{eq:D-NMDT-xy-01-exact}
    \begin{array}{rll}
         x &= \dsum_{j = 1}^L 2^{-j}\beta^x_j + \Del x, \quad
         y = \dsum_{j = 1}^L 2^{-j}\beta^y_j + \Del y\\
         z &= \dsum_{j = 1}^L 2^{-j}[\beta^y_j ((1 - \lambda) \Del x + \lambda x) + \beta^x_j (\lambda \Del y + (1 - \lambda) y)] + \Del x \Del y\\
         \Del x, \Del y &\in [0, 2^{-L}], \quad
         x, y \in [0, 1], \quad
         \bm \beta^x, \bm \beta^y \in \{0, 1\}^L.
    \end{array}
\end{equation}
Finally, we obtain the complete MIP relaxation \DNMDT stated in~\eqref{eq:D-NMDT}
by applying McCormick envelopes to the product terms $ \beta^y_j ((1 - \lambda) \Del x + \lambda x) $,
$ \beta^x_j (\lambda \Del y + (1 - \lambda) y) $ and~$ \Del x \Del y $.
For bounds on the terms $ ((1 - \lambda) \Del x + \lambda x) $
and $ (\lambda \Del y + (1 - \lambda) y) $,
see \Cref{app:gen-bnds}.

\section{MIP Relaxations on General Intervals}
\label{app:gen-bnds}

In this section, we generalize the MIP relaxations
for $ \gra_{[0, 1]^2}(xy) $ and $ \gra_{[0, 1]}^2(x^2) $ discussed in this article
to general box domains $ (x, y) \in [\xmin, \xmax] \times \in [\ymin, \ymax] $
and $ x \in [\xmin, \xmax] $, where $ \xmin < \xmax $, $ \ymin < \ymax $
and $ \xmin, \xmax, \ymin, \ymax \in \R$.
by giving explicit formulations for general bounds on~$x$ and~$y$.

\subsection{MIP Relaxations for Bivariate Quadratic Equations}
First, we consider MIP relaxations for $ z = xy $
and give explicit models of \NMDT and \DNMDT for general box domains.

Next, we consider the MIP relaxation \NMDT.
To derive the general formulation, we first introduce $ \xhat \in [0, 1] $ and define $ \zhat = \xhat y $,
then use the definitions $ x = \lx \xhat + \xmin $ and
\begin{equation*}
     \z = xy = (\lx \xhat + \xmin)y = \lx \zhat + \xmin \cdot y
\end{equation*}
to obtain
\begin{equation}
    \label{eq:NMDT-xy-relax}
    \begin{array}{rll}
        x &= \lx\dsum_{j = 1}^L 2^{-j}\beta_j + \Del x + \xmin\\
        z &= \lx\dsum_{j = 1}^L 2^{-j} \beta_j y + \Del x \cdot y + \xmin \cdot y\\
        \Del x &\in [0, 2^{-L}(\xmax - \xmin)], \quad
        y \in [\ymin, \ymax], \quad
        \bm \beta &\in \{0, 1\}^L.
    \end{array}
\end{equation}
In this way, we are able to formulate the MIP relaxation \NMDT on a general box domain as follows:
\begin{equation}
    \label{eq:NMDT-gen}
    \begin{array}{rll}
         x &= \lx \dsum_{j = 1}^L 2^{-i}\beta_j + \Del x + \xmin\\
         \z &= \lx \dsum_{j = 1}^L 2^{-j}u_j + \Del z + \xmin \cdot y\\
         (x, \alpha_j, u_j) &\in \mathcal{M}(x, \beta_j) & j \in 1, \ldots, L\\
         (\Del x, y, \Del z) &\in \mathcal{M}(\Del x, y)\\
         \Del x &\in [0, 2^{-L}\lx], \quad
         y \in [\ymin, \ymax], \quad
         \bm \beta \in \{0, 1\}^L
    \end{array}
\end{equation}

Finally, we present the modelling of \DNMDT on general box domains.
Analogously as for \NMDT, we apply McCormick envelopes
to model all remaining product terms~$ \alpha_j y $ and~$ \Del x \cdot y $.
Further, we introduce the variables $ \xhat \in [0, 1] $ and $ \yhat \in [0, 1] $
to map the domain to $ [0, 1] $ intervals
by using the transformations $ x \define \lx \xhat + \xmin $ and $ y \define \ly \yhat + \ymin $
as well as 
\begin{equation*}
    \begin{array}{rl}
        z &= xy = (\lx \xhat + \xmin)(\ly \yhat + \ymin)\\
            &= \lx \ly\xhat \yhat + \lx \xhat \ymin + \ly \yhat \xmin + \xmin \ymin\\
            &= \lx \ly \zhat + \lx \xhat \ymin + \ly \yhat \xmin + \xmin \ymin.
    \end{array}
\end{equation*}
As in the derivation of~\eqref{eq:D-NMDT}, we then obtain the formulation \DNMDT
by applying McCormick envelopes to the product terms~$ \beta_i ((1-\lambda) \Del \xhat + \lambda \xhat) $, $ \alpha_i (\lambda \Del \yhat + (1 - \lambda) \yhat) $ and $ \Del \xhat \Del \yhat $.
As in~\eqref{eq:D-NMDT}, we incorporate the following bounds to construct McCormick envelopes:
\begin{equation*}
    \begin{array}{rl}
        (1 - \lambda) \Del \xhat + \lambda \xhat &\in [0, (1-\lambda) 2^{-L} + \lambda] \\
        \lambda \Del \yhat + (1-\lambda) \yhat &\in [0, \lambda 2^{-L} + (1-\lambda)].
    \end{array}
\end{equation*}
Altogether, we are now ready to state the MIP relaxation \DNMDT on general box domains:
\begin{equation}
    \label{eq:D-NMDT-gen}
    \begin{array}{rll}
          x &= \lx \xhat + \xmin, \quad 
          y = \ly \yhat + \ymin\\
          \z &= \lx \ly \zhat + \lx \xhat \ymin + \ly \yhat \xmin + \xmin \ymin\\
         \xhat &= \dsum_{j = 1}^L 2^{-j}\beta^x_j + \Del \xhat, \quad 
         \yhat = \dsum_{j = 1}^L 2^{-j}\beta^y_j + \Del \yhat\\
         \zhat &=  \dsum_{j = 1}^L 2^{-j}(u_j + v_j) + \Del \zhat\\
         (\lambda \Del \yhat + (1-\lambda) \yhat, \beta^x_j, u_j) &\in \mathcal{M}(\lambda \Del \yhat + (1 - \lambda) \yhat, \alpha_j) &\quad j \in 1, \ldots, L\\
         ((1 - \lambda) \Del \xhat + \lambda \xhat, \beta^y_j, v_j) &\in \mathcal{M}((1 - \lambda) \Del \xhat + \lambda \xhat, \beta_j) & \quad j \in 1, \ldots, L\\
         (\Del \xhat, \Del \yhat, \Del \zhat) &\in \mathcal{M}(\Del \xhat, \Del \yhat)\\
         \Del \xhat, \Del \yhat &\in [0, 2^{-L}], \quad
         \xhat, \yhat \in [0, 1], \quad
         \bm \beta^x, \bm \beta^y \in \{0, 1\}^L
    \end{array}
\end{equation}

\subsection{MIP Relaxations for Univariate Quadratic Equations}

For \NMDT and \DNMDT, 
we derive the general formulations by using the derivations in \Cref{ssec:derivations} with $ x = y $.
In the case of \NMDT, where the original model is~\eqref{eq:NMDT-xsq}, this leads to
\begin{equation}
    \label{eq:NMDT-xsq-gen}
    \begin{array}{rll}
         x &= \lx \dsum_{i = 1}^L 2^{-i}\beta_i + \Del x + \xmin\\
         \y &= \lx \dsum_{i = 1}^L 2^{-i}u_i + \Del \y + \xmin \cdot x\\
         (x, \beta_i, u_i) &\in \mathcal{M}(x, \alpha_i) & i \in 1, \ldots, L\\
         (\Del x, x, \Del \y) &\in \mathcal{M}(\Del x, x)\\
         \Del x &\in [0, 2^{-L} \lx], \quad
         x \in [\xmin, \xmax], \quad
         \bm \beta \in \{0, 1\}^L.
    \end{array}
\end{equation}
For \DNMDT, we obtain~\eqref{eq:D-NMDT-xsq} for general domains as follows:
\begin{equation}
    \label{eq:D-NMDT-xsq-gen}
    \begin{array}{rll}
         x &= \lx\dsum_{i = 1}^L 2^{-i}\beta_i + \lx \Del x + \xmin\\
         z &= \lx\dsum_{i = 1}^L 2^{-i} u_i + \lx^2 \Del z + \xmin(x + \lx\Del x)\\
         (\lx \Del x + x, \beta_i, u_i) &\in \mathcal{M}(\lx \Del x + x, \beta_i) &\quad i \in 1, \ldots, L\\
         (\Del x, \Del z) &\in \mathcal{M}(\Del x)\\
         \Del x &\in [0, 2^{-L}], \quad
         x \in [\xmin, \xmax], \quad
         \bm \beta \in \{0, 1\}^L,
    \end{array}
\end{equation}
with $ \lx \Del x + x \in [\xmin, \lx 2^{-L} + \xmax] $.
\section{Instance set}
\label{sec:instance_set}
In \cref{table_instance} we show a listing of all instances of the computational study from \cref{sec:computations}.
The boxQP instances are publicly available at
\href{https://github.com/joehuchette/quadratic-relaxation-experiments}{https://github.com /joehuchette/quadratic-relaxation-experiments}.
The ACOPF instances are also publicly available at \href{https://github.com/robburlacu/acopflib}{https://github.com/robburlacu/acopflib}. 
The QPLIB instances are available at \href{https://qplib.zib.de/}{https://qplib.zib.de/}.
In total, we have 60 instances, of which 30 are dense and 30 are sparse.

\begin{table}[h]
\caption{IDs of all 60 instances used in the computational study. In bold are the IDs of the instances that are dense. }
\centering
\begin{tabular}{rrrrrrrrr}
\toprule
\multicolumn{9}{c}{\textnormal{boxQP instances: spar}} \\
\midrule
\textbf{020-100-1} & & \textbf{020-100-2} & & \textbf{030-060-1} & & \textbf{030-060-3} & & \textbf{040-030-1} \\
\textbf{040-030-2} & & \textbf{050-030-1} & & \textbf{050-030-2} & & \textbf{060-020-1} & & \textbf{060-020-2} \\
070-025-2 & & \textbf{070-050-1} & & \textbf{080-025-1} & & \textbf{080-050-2} & & \textbf{090-025-1} \\
\textbf{090-050-2} & & \textbf{100-025-1} & & \textbf{100-050-2} & & \textbf{125-025-1} & & \textbf{125-050-1} \\
\midrule
\multicolumn{9}{c}{\textnormal{ACOPF instances: miqcqp\_ac\_opf\_nesta\_case}} \\
\midrule
3\_lmbd\_api & & 4\_gs\_api & & 4\_gs\_sad & & 5\_pjm\_api & & 5\_pjm\_sad \\
6\_c\_api & & 6\_c\_sad & & 6\_ww\_sad & & 6\_ww & & 9\_wscc\_api \\
9\_wscc\_sad & & 14\_ieee\_api & & 14\_ieee\_sad & & 24\_ieee\_rts\_api & & 24\_ieee\_rts\_sad \\
29\_edin\_api & & 29\_edin\_sad & & 30\_fsr\_api & & 30\_ieee\_sad & & 9\_epri\_api \\
\midrule
\multicolumn{9}{c}{\textnormal{QPLIB instances: QPLIB\_}} \\
\midrule
\textbf{0031} & & \textbf{0032} & & \textbf{0343} & & 0681 & & 0682 \\
0684 & & 0698 & & \textbf{0911} & & \textbf{0975} & & \textbf{1055} \\
\textbf{1143} & & \textbf{1157} & & \textbf{1423} & & \textbf{1922} & & 2882 \\
2894 & & 2935 & & 2958 & & 3358 & & 3814 \\
\bottomrule
\label{table_instance}
\end{tabular}
\end{table}

\end{document}